\documentclass[11pt, reqno]{amsart}
\usepackage[margin=3cm]{geometry}
\usepackage{graphicx} 
\usepackage{adjustbox}
\usepackage{amsmath,amssymb,amsxtra,tikz,stackengine,bm,mathrsfs,bbm,tikz-cd,comment,hyperref,combelow,amsopn}
\usepackage{enumitem}

\numberwithin{equation}{section}

\newtheorem{theorem}{Theorem}[section]
\newtheorem{proposition}[theorem]{Proposition}
\newtheorem{corollary}[theorem]{Corollary}
\newtheorem{lemma}[theorem]{Lemma}

\theoremstyle{definition}
\newtheorem{remark}[theorem]{Remark}

\newtheorem{definition}[theorem]{Definition}
\newtheorem{example}[theorem]{Example}

\newtheorem{assumption}[theorem]{Assumption}

\newcommand{\PP}{\mathbb{P}}
\newcommand{\C}{\mathbb{C}}

\newcommand{\Z}{\mathbb{Z}}

\newcommand{\Cat}{\mathcal{C}}

\newcommand{\M}{\mathcal{M}}
\newcommand{\F}{\mathcal{F}}
\newcommand{\CO}{\mathcal{O}}
\newcommand{\I}{\mathcal{I}}
\newcommand{\diag}{\mathrm{diag}}
\newcommand{\Hilb}{\mathrm{Hilb}}
\newcommand{\Aqt}{\mathbb{A}_{q,t}}
\newcommand{\Bqt}{\mathbb{B}_{q,t}}
\newcommand{\Aq}{\mathbb{A}_{q}}
\newcommand{\sq}{\square}
\newcommand{\aut}{\theta}
\newcommand{\adj}{\Theta}
\newcommand{\e}{\mathbbm{1}}
\newcommand{\Sym}{\mathrm{Sym}}
\DeclareMathOperator{\Hom}{Hom}
\DeclareMathOperator{\ch}{ch}

\DeclareMathOperator{\filt}{F}
\DeclareMathOperator{\rad}{rad}
\DeclareMathOperator{\soc}{soc}
\DeclareMathOperator{\ssn}{ss}
\DeclareMathOperator{\gr}{gr}
\newcommand{\resc}{\mathrm{resc}}
\newcommand{\rev}{\mathrm{rev}}

\newcommand{\AH}{\mathrm{AH}}

\newcommand{\cB}{\mathcal{B}}

\newcommand{\EHA}{\mathcal{E}}

\newcommand{\exc}{excellent}
\newcommand{\ext}{\mathrm{ext}}
\newcommand{\Ch}{\mathrm{Ch}}
\newcommand{\gCh}{\mathrm{gCh}}

\newcommand{\newword}[1]{\emph{\textbf{#1}}}

\title{Calibrated Representations of the Double Dyck Path Algebra}
\author[N. Gonz\'alez]{Nicolle Gonz\'alez}
\address{Department of Mathematics,  University of California Berkeley , CA 94720-3840, U.S.A.}
\email{nicolle@math.berkeley.edu}
\author[E. Gorsky]{Eugene Gorsky}
\address{Department of Mathematics\\ University of California, Davis\\ One Shields Avenue, Davis, CA, USA}
\email{egorskiy@math.ucdavis.edu}
\author[J. Simental]{Jos\'e Simental}
\address{Instituto de Matem\'aticas, Universidad Nacional Aut\'onoma de M\'exico. Ciudad Universitaria, CDMX,  M\'exico}
\email{simental@im.unam.mx}
\date{}

\begin{document}

\maketitle
\begin{abstract}
The double Dyck path algebra $\Aqt$ and its polynomial representation first arose as a key figure in the proof of the celebrated Shuffle Theorem of Carlsson and Mellit. A geometric formulation for an equivalent algebra $\Bqt$ was then given by the second author and Carlsson and Mellit using the K-theory of parabolic flag Hilbert schemes.  In this article, we initiate the systematic study of the representation theory of the double Dyck path algebra $\Bqt$. We define a natural extension of this algebra and study its calibrated representations. We show that the polynomial representation is calibrated, and place it into a large family of calibrated representations constructed from posets satisfying certain conditions. We also define tensor products and duals of these representations, thus proving (under suitable conditions) the category of calibrated representations is generically monoidal. As an application, we prove that tensor powers of the polynomial representation can be constructed from the equivariant K-theory of parabolic Gieseker moduli spaces. 
\end{abstract}

\setcounter{secnumdepth}{4}
\setcounter{tocdepth}{1}
\tableofcontents

\section{Introduction}

In \cite{Shuffle}, Carlsson and Mellit introduced a remarkable algebra $\Aqt$, known as the \newword{double Dyck path algebra}, and its polynomial representation and used it to prove the long-standing Shuffle Conjecture \cite{HHLRU, HMZ}. This infinite dimensional algebra contains many well known subalgebras, such as the elliptic Hall \cite{Mellit} and affine Hecke algebras, and has been shown to have deep connections with skein theory \cite{GH}, the homology of toric knots \cite{ Mellit}, and double affine Hecke algebras \cite{milo, IonWu}. In particular, in \cite{CGM} the second author alongside Carlsson and Mellit realized the polynomial representation geometrically on the K-theory of parabolic flag Hilbert schemes by studying a closely related algebra denoted $\Bqt$. This algebra is simpler than $\Aqt$ as it contains one fewer generator; nonetheless, the algebras $\Aqt$ and $\Bqt$ are, in a sense, equivalent as they both simultaneously contain each other \cite{CGM}.

In this paper, we initiate the study of the representation theory of the $\Bqt$ algebra (see Section \ref{sec: Bqt} for the definition). Prior to our results in this article, the only known representation of $\Bqt$ was the so-called \newword{ polynomial representation} 
\begin{equation}
\label{eq: polynomial}
V=\bigoplus_{k=0}^{\infty}V_k,\qquad V_k=\Sym_{q,t}\otimes \C[y_1,\ldots,y_k]
\end{equation}
where $\Sym_{q,t}$ is the ring of symmetric functions in infinitely many variables over $\C(q,t)$. The action of the generators of $\Bqt$ on $V$ can be either defined by explicit plethystic formulas as in \cite{Shuffle}, geometrically by identifying $V_k$ with equivariant $K$-theory of certain algebraic varieties as in \cite{CGM}, or topologically (at $q=t^{-1}$) as operators on a thickened punctured annulus \cite{GH}.

In this work, we slightly extend the $\Bqt$ algebra by adjoining a family of pairwise commuting generators $\Delta_{p_m}\ (m>0)$ indexed by power sum symmetric functions $p_m$. Notably, in the polynomial representation \eqref{eq: polynomial} these Delta generators act on $V_0\simeq \Sym_{q,t}$ by Macdonald operators. Geometrically, the Delta generators correspond to certain tautological bundles. We denote by $\Bqt^{\ext}$ the algebra generated by $\Bqt$ and $\Delta_{p_m}$ (see Definition \ref{def:BqtExt} for the relations). In particular, the operators $\Delta_{p_m}$ commute with the generators $z_i$ for all $m$ and $i$.

 The notion of a \newword{calibrated representation} was coined by Ram \cite{Ram} in an effort to understand and give explicit bases for the irreducible modules of the affine Hecke algebra (AHA) by restricting to representations with a basis of simultaneous eigenvectors for a large commutative subalgebra, thus mimicking the classical construction for representations of complex semisimple Lie algebras. The restriction to this class of modules has proven to be very fruitful, with calibrated representations and Gelfand-Tsetlin bases \cite{GZ} being studied for the degenerate AHA \cite{Cherednik}, modular reps of $S_n$ \cite{Kleshchev}, the double affine Hecke algebra at generic parameters \cite{SuzukiVazirani} and roots of unity \cite{BCMN}, the rational Cherednik algebra \cite{Griffeth}, Khovanov-Lauda-Rouqui\"er algebras \cite{KleshchevRam}, the periplectic Brauer algebra \cite{ImNorton} and its degenerate affine counterpart \cite{DHIN}, to name a few. In particular, calibrated representations are very natural generalizations of unitary representations, which play an in important role in the representation theory of finite groups. These were used by the third author alongside Bowman and Norton to relate and classify unitary and calibrated representations for cyclotomic Hecke algebras \cite{BNS2}. 

Consequently, given the natural inclusion of the affine Hecke algebra into $\Bqt$, it is sensible to initiate the study of the representation theory of $\Bqt^{\ext}$ by restricting to an analogous class of modules. 

\begin{definition}
\label{def: calibrated intro}
We call a representation of $\Bqt^{\ext}$ \newword{ calibrated}, if all operators $\Delta_{p_m}$ and $z_i$ diagonalize simultaneously with joint simple spectrum.
\end{definition}
\noindent We note, that by \cite[Theorem 7.0.1]{CGM} the polynomial representation is calibrated (strictly speaking, in \cite{CGM} only the joint eigenbasis for $z_i$ is defined, but the diagonal action of $\Delta_{p_m}$ is straightforward, see Section \ref{sec: poly}). Moreover, any calibrated $\Bqt^{\ext}$ representation is a calibrated AHA representation. 

In \cite{Ram} Ram proves that isomorphism classes of irreducible calibrated representations of the AHA are in bijection with certain equivalence classes of so-called calibrated sequences $[\underline{w}]=[w_k,\dots w_1]$. Inspired by his description, in this paper we construct and classify a large class of calibrated representations of $\Bqt^{\ext}$ that arise from certain weighted posets and use these to provide a partial classification result.

\begin{definition}
\label{def: poset intro}
A \newword{weighted poset} is a graded poset $E$ with a collection of symmetric functions $p_m:E\to \C(q,t)$ such that if $\mu$ covers $\lambda$, then $p_m(\mu)=p_m(\lambda)+x^m$ for some $x$ and all $m$. We call such $x$ an addable weight for $\lambda$ and write $\mu = \lambda \cup x$. 

For $\lambda \in E$ with $x \in \C(q,t)$ addable for $\lambda$, and $y \in \C(q,t)$ addable for $\lambda\cup x$,   a weighted poset is \newword{excellent} if, in addition, $y\neq x,y\neq (qt)^{\pm 1}x$, and either $y=qx$, or $y=tx$, or $y$ is addable for $\lambda$ and $x$ is addable for $\lambda\cup y$. 
\end{definition}

We define a $\Bqt^{\ext}$-module $V(E)$ from an excellent weighted poset $E$ by considering the span of \emph{good chains} in $E$, which are certain maximal chains of the form (see Definition \ref{def: good new}):
$$
[\lambda;\underline{w}]:=\{\lambda, \lambda\cup w_k, \lambda\cup w_k\cup w_{k-1},\ldots, \lambda\cup w_k\cup\ldots\cup w_1\},
$$
and describing the $\Bqt^{\ext}$ action explicitly. While some of the operators act diagonally as expected, 
\[
z_i[\lambda; \underline{w}]=w_i[\lambda; \underline{w}] \qquad \Delta_{p_m}[\lambda; \underline{w}]=\Bigl( p_m(\lambda)+w_k^m+\ldots+w_1^m \Bigr) [\lambda; \underline{w}], 
\]
the action of $d_{+}$ relies on auxiliary coefficients $c(\lambda;x)$ defined for all $\lambda\in E$ with addable weights $x$. One of our first main results is that such a representation always exists and is calibrated. 

\begin{theorem}[Thm. \ref{thm: relations hold}]
\label{thm: rep from excellent intro}
Given any excellent weighted poset $E$, there exists a representation $V(E)=\bigoplus_{k=0}^{\infty}V_k(E)$ with basis consisting of good chains $[\lambda; \underline{w}]$ and $\Bqt^{ext}$-action given by Definition \ref{def: calibrated from excellent}. Moreover, $V(E)$ is calibrated if and only if the coefficients $c(\lambda;x)$ satisfy the additional condition in equation \eqref{eq: monodromy}.

\end{theorem}

While it may seem as if the representations $V(E)$ depend heavily on the choice of parameters $c(\lambda;x)$, we show in Theorem \ref{thm: existence and uniqueness} this is not the case, with any choice of coefficients satisfying \eqref{eq: monodromy} yielding isomorphic representations. It is interesting to note, that this defining equation bears remarkable similarity with the `wheel condition' that arises in the Shuffle algebra \cite{negut}.

Armed with Theorem \ref{thm: rep from excellent intro}, we construct many examples of calibrated representations. For instance, in Theorem \ref{thm:submodules} we classify all submodules of $V(E)$ in terms of certain coideals associated to $E$. If the underlying poset $E$ is finite, we get interesting finite-dimensional representations of $\Bqt^{\ext}$. The polynomial representation \eqref{eq: polynomial} corresponds to the poset of Young diagrams,
and many more examples can be constructed by studying Young diagrams with certain restrictions (for instance, with at most $N$ rows, $N$ columns, etc).

Next, we study the monoidal structure in the category of calibrated representations. 
We begin by defining the tensor product of the representations $V(E_1)$ and $V(E_2)$, provided that the weights of the posets $E_1$ and $E_2$ are sufficiently generic with respect to each other. If this happens, we will say that $E_1$ and $E_2$ are \emph{in general position} (see Section \ref{sec: tensor}).

\begin{theorem}[Thm. \ref{thm: c for tensor product}]
\label{thm: tensor intro}
Given two excellent weighted posets $E_1$ and $E_2$ in general position one can define an excellent weighted poset $E_1\times E_2$. The underlying poset is simply $E_1\times E_2$, but the corresponding representation $V(E_1\times E_2)$ is {\bf not} the tensor product of $V(E_1)$ and $V(E_2)$. Instead,
\begin{equation}
V_k(E_1\times E_2)\simeq \bigoplus_{i=0}^{k}V_i(E_1)\otimes V_{k-i}(E_2)\otimes \C^{\binom{k}{i}}.
\end{equation}
\end{theorem}

This construction allows us to define a faithful (but not full) functor from the category of excellent weighted posets to the category of $\Bqt^{\ext}$-modules and thus deduce that the subcategory of calibrated $\Bqt^{\ext}$-representations arising from posets is \newword{generically monoidal}; that is, tensor products are defined when the underlying posets are in very general position. 

We remark that, if $E_1$ and $E_2$ are arbitrary posets, then it is always possible to \lq\lq twist\rq\rq\ the weights of the posets $E_1$ and $E_2$ in order to obtain posets $E_1(a_1)$ and $E_2(a_2)$ which are in general position. Here, \emph{twisting} means that the underlying poset of $E_1(a_1)$ is just $E_1$, but the weight functions $p_m$ are rescaled by powers of $a_1$. 

As an example of the tensor product construction, in Section \ref{sec: Gieseker} we generalize the geometric constructions of \cite{CGM} to study the equivariant $K$-theory of parabolic Gieseker moduli spaces of rank $r$ framed sheaves on the plane $\M^{par}(r,n;n+k)$ and obtain a calibrated representation on the poset of $r$-multipartitions. In particular, the $r=1$ case recovers the Carlsson-Gorsky-Mellit polynomial representation of $\Bqt$.  Utilizing our newly defined 
tensor product of calibrated representations, we prove this representation also arises as a tensor product of modules.
\begin{theorem}[Thm. \ref{thm: gieseker as tensor power}]
\label{thm: Gieseker intro}
There is a representation of $\Bqt^{\ext}$ in the equivariant $K$-theory of 
$$
\M^{par}(r):=\bigsqcup_{k,n}\M^{par}(r,n;n+k)
$$ This representation is calibrated, and isomorphic to the $r$-th tensor power of the polynomial representation in the sense of Theorem \ref{thm: tensor intro}.   
\end{theorem}

We further investigate the structure of $\Bqt$-mod, by defining a new $\C(q,t)$-anti-linear anti-involution $\adj$ on $\Bqt^{\ext}$ from which we construct a duality functor on the category of calibrated representations. This anti-involution extends, in a very precise way, the classical Weyl involution $\omega$ on the algebra of symmetric functions (see Theorem \ref{thm: omega}).

\begin{theorem}[Thm. \ref{thm: duality poset}]\label{thm: duality intro}
The algebra $\Bqt^{\ext}$ admits a $\C(q,t)$-anti-linear anti-automorphism $\adj$ interchanging the generators $d_{+}$ and $d_{-}$. For any excellent weighted poset $E$ one can define an excellent weighted poset $E^{\vee}$ such that
\[
V(E)^{*} \simeq V(E^{\vee}),
\]
where the left-hand side is a restricted dual that becomes a $\Bqt^{\ext}$-representation via $\adj$. The underlying poset of $E^{\vee}$ is the opposite poset of $E$. 
\end{theorem}

Finally, we attempt a classification of calibrated representations. Under some mild assumptions, we associate a poset $E$ to a calibrated representation $V$. More precisely, we prove the following. 

\begin{theorem}[Thm. \ref{thm: classification}]
\label{thm: classification intro}
Assume that $V$ is a calibrated representation and that, for every weight vector $v_k \in V_k$, $d_{-}^{k}(v_k) \in V_0$ is nonzero. Then:
\begin{enumerate}
\item The representation $V$ defines a weighted poset $E$.
\item Under certain completeness assumption on $V$ (see Section \ref{sec: classification} for details), the poset $E$ is excellent and $V \simeq V(E)$. 
\end{enumerate}
\end{theorem}

In \cite{Mellit}, the action of the $\Aqt$ algebra on the polynomial representation was used in order to construct a representation of the elliptic Hall algebra $\EHA$ \cite{EHA} on the space of symmetric functions. This action of $\EHA$ coincides with the $\EHA$-action of \cite{FT, negut, SV}, see also \cite[8.2]{CGM}. Since $\Aqt$ is contained within $\Bqt$, it would be interesting to use the constructions in this paper to build new representations of $\EHA$ on vector spaces with bases given by an excellent poset $E$. We plan to pursue this direction in future work.

\subsection*{Structure of the paper} In Section \ref{sec: Bqt} we review the $\Bqt$ algebra, its polynomial representation, and define the extended algebra $\Bqt^{\ext}$. 
Section \ref{sec: construction} constitutes the technical heart of the paper. In it, we construct explicit representations of $\Bqt^{\ext}$ starting from weighted posets, and prove Theorem \ref{thm: rep from excellent intro} by combining Theorems \ref{thm: relations hold} and \ref{thm: existence and uniqueness}. We finish this section with several illustrative examples. In Section \ref{sec: structure} we study the structure of the representations $V(E)$. In particular, we describe the subrepresentations of $V(E)$, as well as morphisms between $V(E)$ and $V(E')$. Section \ref{sec: Gieseker} is devoted to the study of a single example: that of the equivariant $K$-theory of parabolic Gieseker moduli spaces. This section is inspired by the results of \cite{CGM} and serves as a motivation for Section \ref{sec: tensor}, where we define tensor products of representations. In Section \ref{sec: duality} we study an antilinear automorphism of $\Bqt^{\ext}$, and use it to define a duality functor on the category of calibrated representations. Finally, in Section \ref{sec: classification} we attempt to classify calibrated representations satisfying certain conditions by recovering a poset from a calibrated representation. 

\section*{Acknowledgments}

We are grateful to Milo Bechtloff-Weising, Erik Carlsson, Anton Mellit, Andrei Negu\cb{t}, and Monica Vazirani for useful discussions. E. G. would like to thank Anna for suggesting that $x$ is not equal to $y$.
The work of E. G. was partially supported by the NSF grant DMS-2302305. J. S. was partially supported by CONAHCyT project CF-2023-G-106 and UNAM's PAPIIT Grant IA102124.

\section{The $\Bqt$ Algebra}
\label{sec: Bqt}

The \newword{double Dyck path algebra} $\Bqt$ was originally introduced by Carlsson-Gorsky-Mellit \cite[Definition 3.2.4]{CGM}. It is defined as the non-unital $\C(q,t)$-algebra generated by a collection of orthogonal idempotents $\e_{k}$ for each $k \in \Z_{\geq 0}$ and generators $\e_{k+1}d_{+}\e_{k}$, $\e_{k}d_{-}\e_{k+1}$, $\e_{k}T_{i}\e_{k}$, $\e_{k}z_{j}\e_{k}$ for each $k \in \Z_{\geq 0}$ with $1\leq i \leq k-1$ and $1\leq j \leq k$. 

We henceforth omit the usage of the idempotents, and write the generators simply as $d_+, d_-, T_i, z_j$, with the index of the surrounding idempotents implicitly deduced. These operators satisfy the following relations:  
\begin{align}
(T_i - 1)(T_i + q) = 0, \qquad T_{i}T_{i+1}T_{i} &= T_{i+1}T_{i}T_{i+1}, \qquad T_{i}T_{j} = T_{j}T_{i}\, (|i - j| > 1) \label{eq:hecke relns} \\
T_{i}^{-1}z_{i+1}T_{i}^{-1} &= q^{-1}z_{i} \, (1 \leq i \leq k-1) \label{eq:T and z} \\
z_{i}T_{j} = T_{j}z_{i} \, (i \not\in \{j, j+1\}),& \qquad z_{i}z_{j} = z_{j}z_{i} (1 \leq i, j \leq k) \label{eq:affine Hecke relns}  \\
d^{2}_{-}T_{k-1} = d^{2}_{-},& \qquad d_{-}T_{i} = T_{i}d_{-} \, (1 \leq i \leq k-2) \label{eq:T d-} \\
T_{1}d_{+}^{2} = d_{+}^{2},& \qquad d_{+}T_{i} = T_{i+1}d_{+} \, (2 \leq i \leq k-1) \label{eq:T d+} \\
q\varphi d_{-} = d_{-}\varphi T_{k-1},& \qquad T_{1}\varphi d_{+} = qd_{+}\varphi \label{eq:phi} \\
z_{i}d_{-} = d_{-}z_{i},& \qquad d_{+}z_{i} = z_{i+1}d_{+} \label{eq: d z} \\
z_{1}(qd_{+}d_{-} - d_{-}d_{+}) &= qt(d_{+}d_{-} - d_{-}d_{+})z_{k}. \label{eq:qphi}
\end{align}
Here,
$$
\varphi=\frac{1}{q-1}(d_{+}d_{-}-d_{-}d_{+}).
$$

In what follows we will sometimes use operators
\begin{equation}\label{eqn: y}
y_i=q^{i-k}T_{i-1}^{-1}\cdots T_1^{-1}\varphi T_{k-1}\cdots T_i.
\end{equation}
By \cite[Lemma 3.1.8]{CGM} the elements $y_1, \dots, y_k$ pairwise commute and generate a copy of the (positive half of) affine Hecke algebra together with $T_1, \dots, T_{k-1}$.

\begin{remark}
Although the algebra $\Bqt$ is not unital, it is complete in the following sense:
 \begin{itemize}
     \item For every $a \in \Bqt$, there exist a finite number of $m$ such that $a\e_{m} \neq 0$. Similarly, there exist a finite number of $n$ such that $\e_{n}a \neq 0$.
     \item For every $a \in \Bqt$, $a = (\sum_{n \geq 0}\e_{n})a(\sum_{m \geq 0} \e_m)$. In particular, this product is well-defined due to the previous point.
 \end{itemize}
 \end{remark}

\subsection{The Polynomial Representation for $\Bqt$}\label{sec: polynomial} In \cite{Shuffle}, Carlsson and Mellit introduced an algebra $\Aqt$, closely related to $\Bqt$, by means of a representation called the \newword{polynomial representation}. By \cite[Theorem 3.4.2]{CGM}, this representation of $\Aqt$ can be restricted to a polynomial representation for the algebra $\Bqt$ as well. We review its construction in this section.

The space admitting the polynomial representation is
\[
V = \bigoplus_{k \geq 0} V_{k}, \qquad V_{k} := \Sym_{q,t} \otimes \C(q,t)[y_1, \dots, y_k]
\]
where $\Sym_{q,t} = \C(q,t)[x_1,x_2,\dots]^{S_\infty}$ denotes the algebra of symmetric functions. The idempotent $\e_k$ acts by projection to the summand $V_k$. The remaining generators of $\Bqt$ act as follows. 
\smallskip 

Given any $F \in V_k$, the Hecke algebra generators act via Demazure-Lusztig operators:
\begin{equation}\label{eqn:action Ti polynomial}
T_{i}F = \frac{(q-1)y_{i+1}F + (y_{i+1}-qy_{i})s_{i}F}{y_{i+1} - y_{i}}, \quad i \leq k-1. 
\end{equation}
Here, the simple transpositions $ s_i \in S_{k}$ act on $V_{k}$ by permuting the variables $y_1, \dots, y_k$. 
\smallskip

The action of the remaining operators is given using plethystic notation. For details on plethysm of symmetric functions we refer the reader to e.g. \cite{HaglundBook}. 

The raising operator $d_{+}: V_{k} \to V_{k+1}$ acts on an element $F \in V_k$ by:
      \begin{equation}\label{eqn:raising polynomial}
      d_{+}F[X] = T_1\cdots T_kF[X+(q-1)y_{k+1}],
      \end{equation}
and the lowering operator $d_{-}: V_{k} \to V_{k-1}$ acts by:
      \begin{equation}\label{eqn:lowering polynomial}
      d_{-}F[X] = -F[X-(q-1)y_{k}]\operatorname{Exp}[-y_{k}^{-1}X]|_{y_{k}^{-1}},
      \end{equation}
where $|_{y_{k}^{-1}}$ denotes taking the coefficient of $y_{k}^{-1}$ and $\operatorname{Exp}[X] = \sum_{n=0}^{\infty}h_{n}[X]$ is the plethystic exponential. We note that $\operatorname{Exp}[-y_{k}^{-1}X] = \sum_{n = 0}^{\infty}(-y_{k}^{-n})e_{n}$.

It remains to express the action of the elements $z_1, \dots, z_k$ on $V_{k}$. By \eqref{eq:T and z} it suffices to provide the action of $z_k$:
      \begin{equation}\label{eqn:z polynomial}
     z_kF = T_{k-1}\cdots T_{1} F[X + (1-q)ty_1 - (q-1)u, y_2, \dots, y_k, u]\operatorname{Exp}[u^{-1}ty_1 - u^{-1}X]|_{u^{0}}
      \end{equation}
      (cf. the proof of \cite[Theorem 3.4.2]{CGM}). 
      
      Finally, we note that in this presentation of the polynomial representation $V$, the $y$ operators in \eqref{eqn: y} act simply by multiplication.

\begin{lemma} \label{lem:polrep}
    The $\Bqt$ polynomial representation $V$ is generated by $1 \in V_{0}$.
\end{lemma}
\begin{proof}
This follows by combining \cite[Theorem 3.4.2]{CGM} with results of \cite{Shuffle}. Indeed, using the notation of \cite[Theorem 3.4.2]{CGM}, we have that $\beta(\Aqt)(1) \subseteq \Bqt(1) \subseteq V$, so it suffices to show that $\beta(\Aqt)(1) = V$. The latter follows from \cite[Theorem 5.2]{Shuffle} (together with \cite[Theorem 3.4.2 (3)]{CGM}), see also Theorem 7.3 in \cite{Shuffle}. 
\end{proof}

One disadvantage of the presentation of $V$ given above is that the operators $z_{i}$ are complicated. As proven in the main theorem of \cite{CGM}, the action of the $z$ operators on $V$ is in fact diagonalizable. Indeed, the vector space $V_{k}$ admits a basis $\{I_{\lambda, w}\}$ where $\lambda$ runs over the set of all partitions and $w= (w_1, \dots, w_k)$ is a certain $k$-tuple of monomials in $q,t$. More precisely, consider a pair of Young diagrams $\lambda\subset \mu$ such that $\mu\setminus \lambda$ is a horizontal strip, and a standard Young tableau $T$ of skew shape $\mu\setminus \lambda$. Then $w_i$ is the $(q,t)$-weight for a box labeled by $i$ in $T$, see Section \ref{sec: poly} for more details.

In this basis, we have that:
\begin{align}
z_{j}I_{\lambda, w} &= w_{j}I_{\lambda, w}, 
\\
 T_{i}I_{\lambda, w} &= \frac{(q-1)w_{i+1}}{w_{i} - w_{i+1}}I_{\lambda, w} + \frac{w_{i} - qw_{i+1}}{w_{i} - w_{i+1}}I_{\lambda, s_{i}w},
\\
d_{-}I_{\lambda, (w_1, \dots, w_{k})} &= I_{\lambda, (w_{1}, \dots, w_{k-1})},
\\
d_{+}I_{\lambda, (w_1, \dots, w_{k})} &= -q^{k}\sum_{\substack{x \; \text{an addable} \\ \text{box of} \; \lambda}} xd_{\lambda \cup x,\lambda}\prod_{i = 1}^{k}\frac{x - tw_i}{x-qtw_i}I_{\lambda \cup x, (x, w_1, \dots, w_k)}, 
\end{align}
where $d_{\lambda \cup x, \lambda}$ is a certain coefficient (see Section \ref{sec: poly}) and we abuse the notation by identifying an addable box $x$ with its $(q,t)$-content. When $k = 0$, the basis $\{I_{\lambda}\}$ of $V_{0} = \Sym_{q,t}$ coincides with the modified Macdonald basis in $\Sym_{q,t}$, see \cite[Theorem 7.0.1]{CGM}.

\subsection{Extended $\Bqt$}

We will need an extension of the algebra $\Bqt$ which involves an additional family of commuting operators $\Delta_f$ defined for an arbitrary symmetric function $f \in \Sym_{q,t}$ \footnote{Note that, technically, we are adding the family $\e_{k}\Delta_{f}\e_{k}$ for $f \in \Sym_{q,t}$ and $k \geq 0$. However, keeping in line with the way the relations \eqref{eq:hecke relns}--\eqref{eq:qphi} are written, the idempotents are omitted.}. Clearly, these are determined by the operators $\Delta_{p_m}$ for all $m \in \Z_{> 0}$, where $p_m$ is the $m^{th}$ power sum symmetric function.  

\begin{lemma}
Consider a family of commuting operators $\Delta_{p_m}$ with $m \in \Z_{>0}$ satisfying
\begin{align}  
[\Delta_{p_m},T_i]=[\Delta_{p_m},z_i]&=[\Delta_{p_m},d_{-}]=0 \label{eq:ext-Bqt-1}, \\
[\Delta_{p_m},d_+]&=z_1^{m}d_{+} \label{eq:ext-Bqt-2}.
\end{align}

Then, the operators $\Delta_{p_m}$ preserve the relations for $\Bqt$.
\end{lemma}

\begin{remark}
Alternatively, we can say that $[\Delta_{p_m},-]$ are commuting derivations on $\Bqt$ satisfying \eqref{eq:ext-Bqt-1} and \eqref{eq:ext-Bqt-2}.

\end{remark}

\begin{proof}

Clearly, $\Delta_{p_m}$ commutes with all relations involving $z_i,T_i$, and $d_{-}$ for any value of $m$. Since $z_1$ commutes with $T_{i+1}$, we get
\begin{align*}
[\Delta_{p_m},d_{+}T_i]&=z_1^md_{+}T_i=T_{i+1}z_1^{m}d_{+}=[\Delta_{p_m},T_{i+1}d_{+}].
\end{align*}
Now, $T_1$ commutes with $z_1^{m}+z_2^{m}$, hence
$$
[\Delta_{p_m},T_1d_{+}^2]=T_1z_1^{m}d_+^2+T_1d_+z_1^{m}d_+=T_1(z_1^m+z_2^m)d_+^2=(z_1^m+z_2^m)T_1d_+^2=[\Delta_{p_m},d_{+}^2], 
$$
proving \eqref{eq:T d+}.

Next, we have the relations \eqref{eq:phi} involving $\varphi$. It is easy to see that $[\Delta_{p_m},\varphi]=z_1^{m}\varphi$, so
\begin{align*}
[\Delta_{p_m},q\varphi d_{-}]&=z_1^{m}q\varphi d_{-}=z_1^{m}d_{-}\varphi T_{k-1}=d_{-}z_1^{m}\varphi T_{k-1}=[\Delta_{p_m},d_{-}\varphi T_{k-1}] \\
[\Delta_{p_m},T_1\varphi d_{+}]&=T_1z_1^{m}\varphi d_{+}+T_1\varphi z_1^{m}d_{+}= T_1(z_1^{m}+z_2^{m})\varphi d_{+}=
(z_1^{m}+z_2^{m})T_1\varphi d_{+} \\
&=(z_1^{m}+z_2^{m})qd_{+}\varphi=qz_1^{m}d_{+}\varphi+qd_{+}z_1^{m}\varphi =[\Delta_{p_m},qd_{+}\varphi].
\end{align*}

Since $z_1$ commutes with $z_k$, the remaining relations are straightforward.
\end{proof}

\begin{definition}\label{def:BqtExt}
The \newword{extended $\Bqt$ algebra} $\Bqt^{\ext}$ is the algebra generated by $\Bqt$ and the commuting operators $\Delta_{p_m}$ with $m \in \Z_{> 0}$ modulo relations \eqref{eq:ext-Bqt-1} and \eqref{eq:ext-Bqt-2}.
\end{definition}

\begin{corollary} \label{cor:mac-operator}
 We can uniquely define the operator $\Delta_{p_m}$ on the polynomial representation $V$ by setting $\Delta_{p_m}1=0$ on $V_0$ and extending via Lemma \ref{lem:polrep}. In particular, this defines the Macdonald operators on $\Sym_{q,t}=V_0$.
\end{corollary}

\begin{remark}
A similar extension has been considered by Bechtloff-Weising in \cite{milo}, with a few differences. First, Bechtloff-Weising works with the polynomial representation of the \emph{stable limit DAHA} which is closely related to, but different from, the $\Aqt$-algebra, cf. \cite{IonWu}. Second, \cite{milo} extends the stable limit DAHA by a single operator, whose joint eigenspaces in the polynomial representation with a natural polynomial subalgebra of the stable limit DAHA are one-dimensional. It would be interesting to explore the relation between the construction in \cite{milo} and the operators $\Delta_{p_{m}}$. 
\end{remark}

\begin{remark} A natural question is whether an extended $\Aqt$ algebra exists as well. Given the explicit homomorphism $\beta: \Aqt \to \Bqt$ given in \cite[Theorem 3.2.7]{CGM}, one might expect such an extension to follow directly. Unfortunately, the image of $\Aqt$ is not closed under the operators $[\Delta_{p_{m}}, -]$, so defining $\Aqt^{\ext}$, if it exists, requires a different approach. 
\end{remark}

\section{Constructing Calibrated Representations}\label{sec: construction}

We start with the following definition.

\begin{definition}\label{def: calibrated}
A representation $V = \bigoplus_{k \geq 0}V_{k}$ of $\Bqt^{\ext}$ is called \newword{calibrated} if the following conditions are satisfied:
\begin{enumerate}
    \item[(1)] For each $k$, $V_{k}$ admits a basis of simultaneous eigenvectors for $z_1, \dots, z_k$ and the Delta operators $\Delta_{p_m}$ for all $m\geq 1$.
    \item[(2)] For each collection $(\zeta, \delta) = (\zeta_{i}, \delta_{m})_{1 \leq i \leq k, 1 \leq m}$, the space
    \[
    V_{k}^{(\zeta, \delta)} := \{v \in V_{k} \mid z_{i}v = \zeta_{i}v, \Delta_{p_m}(v) = \delta_{m}v\}
    \]
    is at most $1$-dimensional. 
    \item[(3)] If $V_{k}^{(\zeta, \delta)} \neq \{0\}$, then $\zeta_1, \dots, \zeta_{k} \neq 0$.
\end{enumerate}
\end{definition}

The goal of this section is to combinatorially construct a family of calibrated representations of $\Bqt^{\ext}$. Roughly speaking, we will construct a calibrated representation $V = \bigoplus_{k \geq 0} V_{k}$ starting from an eigenbasis of $V_0$. In order to have controllable combinatorics, we impose extra structure on the eigenbasis of $V_0$. Namely, that of a poset satisfying certain conditions which we now define. 

\subsection{Weighted Posets}
\label{sec: def poset new}

Let $E$ be a partially ordered set satisfying the following two conditions: 
\smallskip

\noindent{\bf 1.} \underline{Local Finiteness:} For any $\lambda\in E$ there exist finitely many $\mu\in E$ covering $\lambda$ and finitely many $\mu$ covered by $\lambda$.
\medskip

\noindent{\bf 2.} \underline{Grading:} Let $|\cdot|: E \to \Z$ be any map preserving the covering relations (i.e.  if $\mu$ covers $\lambda$, then $|\mu| = |\lambda| + 1$). Then for $E_n := \lbrace \lambda \in E \; ; \; |\lambda| = n \rbrace$ we have $E=\bigsqcup_{n\in \Z}E_n$, so that $E$ is graded.
\medskip

A  \newword{weighting} on $E$ is a family of algebra homomorphisms $\Sym_{q,t} \to \C(q,t)$, $f \mapsto f(\lambda)$, indexed by elements $\lambda \in E$.
A  \newword{weighted poset} is a poset $E$ with a weighting that satisfies the following two additional conditions:
\medskip

\noindent {(a)} \underline{Simple spectrum:} The maps $f \mapsto f(\lambda)$ have \newword{joint simple spectrum}. Since these are algebra homomorphisms, the images are determined by the family of scalars $p_m(\lambda)$ with $m>0$, so the simple spectrum condition means that for any $\mu\neq \lambda$ there exists $m$ such that $p_m(\mu)\neq p_m(\lambda)$.  
\medskip

\noindent{(b)} \underline{Edge Condition:} If $\mu$ covers $\lambda$ then we require that there exists $x\in \C(q,t)$ such that 
\begin{equation}
\label{eq: add box p}
p_m(\mu)=p_m(\lambda)+x^m\ \mathrm{for\ all}\ m.
\end{equation}
If such $x$ exists then it is clearly unique since $x=p_1(\mu)-p_1(\lambda)$.
Note that the simple spectrum condition implies that $x\neq 0$ and $\mu$ is determined by $\lambda$ and $x$.
Hence, in such cases we will often write $\mu=\lambda\cup x$ and call $x$ an \newword{addable weight} for $\lambda$.
 
\begin{remark}
\label{rem: shift}
Given a weighted poset $E$ and a collection of arbitrary scalars $a_m$, we can define a new weighted poset $E'$. The underlying poset of $E'$ is the same as of $E$, but the weightings are modified:
$$
p_{m;E'}(\lambda)=p_{m;E}(\lambda)+a_m.
$$
If $\lambda\cup x=\mu$ in $E$ then the same is true for $E'$ since $p_{m;E'}(\mu)-p_{m;E'}(\lambda)=p_{m;E}(\mu)-p_{m;E}(\lambda)$.
\end{remark}

\subsection{Calibrated Sequences and Affine Hecke Algebra Representations} The \newword{affine Hecke algebra} $\AH_{k}$ is the $\C(q,t)$-algebra\footnote{While the affine Hecke algebra is usually defined over $\C(q)$, we choose to define it over $\C(q,t)$ so that it has the same base field as $\Bqt^{\ext}$. In particular, while $t$ plays no role in the relations of the AHA, its representations are $\C(q,t)$-vector spaces.} with generators $T_i$ and $z_j^{\pm 1}$ with $1\leq i\leq k-1$ and $1\leq j\leq k$ and relations \eqref{eq:hecke relns} and \eqref{eq:T and z}. In particular, $\AH_{k}$ contains a polynomial subalgebra $\C[z_1^{\pm 1}, \dots, z_k^{\pm 1}] \subseteq \AH_{k}$.

\begin{definition}[\cite{Ram}]
A representation $M$ of the affine Hecke algebra (AHA) $\AH_{k}$ is called \newword{calibrated} if the algebra $\C[z_1^{\pm 1}, \dots, z_k^{\pm 1}]$ acts semisimply on $M$, that is, if the operators $z_1, \dots, z_k$ on $M$ are simultaneously diagonalizable with nonzero eigenvalues.
\end{definition}

By definition, if $V = \bigoplus_{k} V_{k}$ is a calibrated $\Bqt^{\ext}$-representation, then $V_{k}$ is a calibrated $\AH_{k}$-representation for each $k$. This motivates the terminology \lq\lq calibrated'' for these $\Bqt^{\ext}$-representations. 

In \cite{Ram}, Ram constructed all irreducible calibrated representations of the affine Hecke algebra $\AH_{k}$. We recall the classification, following \emph{loc. cit.} The main combinatorial input is that of a calibrated sequence, which for convenient notation later on we henceforth read from right to left.

\begin{definition}
A sequence $[w_k,\ldots,w_1] \in (\C(q,t)^{*})^{k}$ is called \newword{calibrated} if whenever $w_i=w_j\ (i<j)$ there exist $i<i',j'<j$ such that $w_{i'}=qw_i,w_{j'}=q^{-1}w_i$.   
\end{definition}

Given a sequence $[\underline{w}]=[w_k,\ldots,w_1]$, we call $s_i=(i\ i+1)$ an \newword{admissible transposition} if $w_{i+1}\neq q^{\pm 1}w_i$. It is easy to check that 
\begin{align}
[s_i(\underline{w})]:=[w_k,\ldots,w_{i},w_{i+1},\ldots,w_1]
\end{align}
is also a calibrated sequence. It is evident that admissible transpositions induce an equivalence relation on the set of calibrated sequences.

\begin{theorem}[\cite{Ram}]
\label{thm: Ram new}
Let $[\underline{w}]=[w_k,\ldots,w_1]$ be a calibrated sequence,
define $V_{\underline{w}}$ as the $\C(q,t)$-span of all sequences equivalent to $\underline{w}$. Then, the following operators on $V_{\underline{w}}$ 
\begin{align}
z_i[\underline{w}']&=w'_i[\underline{w}'], \label{eq: z Ram new}
\\
T_i[\underline{w}']&=\begin{cases}
\frac{(q-1)w'_{i+1}}{w'_i-w'_{i+1}}[\underline{w}']+
\frac{w'_i-qw'_{i+1}}{w'_i-w'_{i+1}}[s_i(\underline{w'})] &; \mathrm{if}\ s_i\ \mathrm{is\ admissible},\\
\frac{(q-1)w'_{i+1}}{w'_i-w'_{i+1}}[\underline{w}'] &; \mathrm{otherwise}.
\end{cases} \label{eq: T Ram new}
\end{align}
define a finite-dimensional irreducible calibrated representation of $\AH_k$. 

Conversely, in a calibrated representation of $\AH_k$ all the joint eigenvalues of $z_1, \dots, z_k$ are calibrated sequences. Moreover, every irreducible calibrated representation of $\AH_k$ is of the form $V_{\underline{w}}$ for some calibrated sequence $\underline{w}$. 
\end{theorem}

Note that Theorem \ref{thm: Ram new} implies that if there exists an eigenvector with $z$-eigenvalues $[\underline{w}]=[w_k,\ldots,w_1]$ and $s_i$ is an admissible transposition for $[\underline{w}]$, then there will also exist another eigenvector with $z$-eigenvalues $s_i[\underline{w}]$. We now observe that given an arbitrary calibrated representation $V$, it is in general \emph{not} possible to obtain a set of calibrated sequences and a basis of $V$ where the action is given by the formulas \eqref{eq: z Ram new}, \eqref{eq: T Ram new}.

\begin{example}\label{ex: calibrated non semisimple}
    Let us consider the affine Hecke algebra $\AH_{2}$. Take any nonzero elements $x, \alpha \in \C(q,t)$, and define a representation of $\AH_{2}$ on $\C(q,t)^{2}$ via the following matrices 
    \[
    T \mapsto \left(\begin{matrix} 1 & \alpha \\ 0 & -q \end{matrix}\right), \qquad z_1 \mapsto \left(\begin{matrix}qx & 0 \\ 0 & x \end{matrix}\right), \qquad z_2 \mapsto \left(\begin{matrix}x & 0 \\ 0 & qx \end{matrix}\right).
    \]
    Note that
    \[
    Tz_{1}T \mapsto \left(\begin{matrix} qx & \alpha x \\ 0 & -qx\end{matrix}\right)\left(\begin{matrix} 1 & \alpha \\ 0 & -q\end{matrix}\right)  = \left(\begin{matrix} qx & 0 \\ 0 & q^{2}x \end{matrix}\right) = q\left(\begin{matrix} x & 0 \\ 0 & qx\end{matrix}\right)
    \]
    so that we indeed have a representation of $\AH_{2}$. Clearly, $V$ is calibrated. However, it is not completely reducible and in fact there is a non-split exact sequence
    \[
    0 \to V_{[x, qx]} \to V \to V_{[qx, x]} \to 0. 
    \]
\end{example}

\subsection{Chains and Excellent Posets}
Let $E$ be a poset with a weighting satisfying all the conditions described in Section \ref{sec: def poset new}. We will consider maximal chains in this poset, and often simply call them chains. We write a covering relation $\lambda < \mu$ by $\lambda \to \mu$. If, moreover, $x = p_1(\mu) - p_1(\lambda)$ then we write $\lambda \xrightarrow{x} \mu$. By the edge condition, any chain has the form
\begin{equation}\label{eq: chain}
\lambda\xrightarrow{w_k}(\lambda\cup w_k)
\xrightarrow{w_{k-1}}(\lambda\cup w_k\cup w_{k-1})\xrightarrow{} \cdots \xrightarrow{w_1} \mu=\lambda\cup w_k\cup \ldots w_1,
\end{equation}
where $w_k$ is addable for $\lambda$, $w_{k-1}$ is addable for $\lambda\cup w_k$ and so on. 

We will denote the chain \eqref{eq: chain} by $[\lambda;w_k,\ldots,w_1]$ and often abbreviate by setting $[\lambda;\underline{w}]:=[\lambda;w_k,\ldots,w_1]$ and $ \lambda\cup \underline{w}:= \lambda\cup w_k\cup\cdots\cup w_1$.

\begin{lemma}\label{lem: uniqueness of weights up to permutation}
a) For any $\lambda < \mu$, the weights $w_1,\ldots,w_k$ are uniquely determined by $\lambda$ and $\mu=\lambda\cup w_k\cup \ldots \cup w_1$, up to permutation. 

b) Conversely, suppose that $[\lambda;\underline{w}]$ and $[\lambda;\underline{w'}]$ are two chains in $E$ with $[\underline{w'}]$ a permutation of $[\underline{w}]$. Then 
$$
\lambda\cup w_k\cup \ldots w_1=\lambda\cup w'_k\cup \ldots w'_1.
$$
\end{lemma}

\begin{proof}
Since $E$ is graded, we have $k=|\mu|-|\lambda|$ is fixed. Now the statement follows from the fact that, for all $s$, the equation
$$
p_s(\mu)-p_s(\lambda)=w_k^{s}+\ldots+w_1^s
$$
is invariant under permutation of the $w_j$'s, together with the simple spectrum condition.
\end{proof}

\begin{definition}
\label{def: excellent new}
A weighted poset $E$ is called \newword{excellent} if for all 2-step chains $[\lambda;x,y]$ in $E$ we have the following:

1) $y\neq x$ and $y\neq (qt)^{\pm 1}x$.

2) Exactly one of the following conditions hold:
\begin{itemize}
\item $y=qx$,
\item $y=tx$, or
\item $[\lambda;y,x]$ is also a chain.
\end{itemize}
\end{definition}

\begin{example} \label{ex:Epartition}
Consider the set $\mathcal{P}$ of all partitions, with partial order given by inclusion of Young diagrams. Recall that the $(q,t)$-content of a cell $\sq \in \lambda$ is given by $q^{r-1}t^{c-1}$ where $r$ and $c$ denote the row and column of $\sq$, respectively. Abusing notation by using $\sq$ to denote both the cells in $\lambda$ and their corresponding $(q,t)$-content, we can define a weighting by $p_m(\lambda) = \sum_{\sq \in \lambda} \sq^{m} \in \C(q,t)$. With this choice, we claim that the resulting weighted poset $\mathcal{P}$ is excellent. Indeed, if $x$ is addable for $\lambda$ and $y$ is addable for $\lambda\cup x$ (so that $[\lambda; x,y]$ is a chain in $\mathcal{P}$) we have three possibilities: $(i)$ $x$ and $y$ are in the same row and $y=qx$, $(ii)$ $x$ and $y$ are in the same column and $y=tx$, or $(iii)$ $y$ is addable for $\lambda$ and $x$ is addable for $\lambda\cup y$, so that $[\lambda; y,x]$ is also a chain in $\mathcal{P}$.
\end{example}

\begin{lemma}
\label{lem: no inverse new}
If $[\lambda;x,y]$ is a chain in an excellent poset then $y\neq q^{-1}x,y\neq t^{-1}x$.
\end{lemma}

\begin{proof}
Suppose that $y=q^{-1}x$. Then by definition $[\lambda;y,x]$ is a chain and $x=qy$. Therefore $[\lambda;x,y]$ is not a chain, yielding a contradiction. The case $y=t^{-1}x$ is similar.
\end{proof}

\begin{lemma}\label{lem: far quotient}
Suppose $[\lambda;\underline{w}]$ is a chain in an excellent poset $E$. Then for all $i>j$ we have $w_i\notin\{qw_j,tw_j,w_j\}$. 
\end{lemma}

\begin{proof}
Since for any $i>1$, $[\lambda\cup w_{k}\cup\dots\cup w_{i+1};w_{i},\ldots,w_j]$ is a chain in $E$ then, without loss of generality, we can assume $i=k$ and $j=1$. 

We will show that for any $k$ there exists a chain $[\lambda; w'_k \dots, w'_2,w_1]$ where $[w'_k,\ldots,w'_2]$ is a permutation of $[w_k,\ldots,w_2]$ with $w'_2=q^{a}t^{b}w_k$ and $a,b\ge 0$. When $k=2$ this is clear. 

For $k>2$, by inducing on $k$, we can first permute  $[w_k,\ldots,w_2]$ to 
$[w'_k,\ldots,w'_3,w_2]$ with $w'_3=q^{a}t^{b}w_k$. Then, since $E$ is excellent, either $w_2=qw'_3=q^{a+1}t^{b}w_k$, or $w_2=tw'_3=q^{a}t^{b+1}w_k$, or we can exchange $w_2$ and $w'_3$. Either way, the result follows.

Now assume that $w_k\in \{qw_1,tw_1,w_1\}$ and we  permuted the chain to $[\lambda;w'_k\ldots,w'_2,w_1]$ as above. If $w'_2=w_k$, we get a contradiction by Lemma \ref{lem: no inverse new} and Definition \ref{def: excellent new}. Otherwise $w'_2=q^{a}t^{b}w_k$ with $a+b\ge 1$ and 
$$
qw'_2=q^{a+1}t^{b}w_k\in 
\{q^{a+2}t^bw_1,q^{a+1}t^{b+1}w_1,q^{a+1}t^{b}w_1\}
$$
so $qw'_2\neq w_1$. Similarly, $tw'_2\neq w_1$, hence we can swap $w_1$ and $w'_2$ and induct on $k$.
\end{proof}

The next result now follows immediately from the previous lemma and the definitions. 

\begin{corollary}\label{cor: excellent chains}
Let $E$ be an excellent poset, and $[\lambda; \underline{w}]$ a chain in $E$. Then:
\begin{enumerate}[leftmargin=*]
\item The sequence $[\underline{w}]$ is calibrated.
\item A transposition $s_{i}$ is admissible for $[\underline{w}]$ if and only if $w_{i} \neq qw_{i+1}$.
\item If $s_{i}$ is an admissible transposition for $[\underline{w}]$, then $[\lambda; s_i(\underline{w})]$ is a chain in $E$ if and only if $w_{i} \neq tw_{i+1}$.
\end{enumerate}
\end{corollary}

\begin{definition}
Let $[\lambda; \underline{w}]$ be a chain in an excellent poset $E$ and $s_i=(i \ i+1)$ an admissible transposition for $[\underline{w}]$. We say that $s_{i}$ is \newword{\exc} for $[\lambda; \underline{w}]$ if $[\lambda; s_i(\underline{w})]$ is a chain in $E$.
\end{definition}

By Corollary \ref{cor: excellent chains}, the transposition $s_{i}$ is  {\exc} for $[\lambda; \underline{w}]$ if and only if $w_{i} \not\in \{qw_{i+1}, tw_{i+1}\}$. In particular, we note that the set of {\exc} transpositions for $[\lambda; \underline{w}]$ depends only on $[\underline{w}]$ and not on $\lambda \in E$.

\begin{lemma}
\label{lem: no qt}
Suppose $[\lambda;\underline{w}]$ is a chain in an excellent poset $E$. If $w_1=qtw_k$ then there exists an $1<i<k$ for which $w_i=qw_k$ or $w_i=tw_k$.
\end{lemma}

\begin{proof}
Suppose that $w_i\neq qw_k$ and $w_i\neq tw_k$ for all $1<i<k$. Then $w_k$ can be moved towards $w_1$ by a sequence of {\exc} transpositions to obtain the chain $[\lambda;w_{k-1}, \dots, w_2,w_k,w_1]$ in $E$. However, this implies that the two-step chain $[\lambda\cup w_{k-1}\cup\dots \cup w_2;w_k,w_1]$, with $w_1=qtw_k$, is also in $E$ which contradicts Definition \ref{def: excellent new} since $E$ is excellent.
\end{proof}

\begin{lemma}
\label{lem: cycles}
Suppose $[\lambda;\underline{w}]$ is a chain in an excellent poset $E$, and $[\underline{w'}]$ is a permutation of $[\underline{w}]$. Then 
$[\lambda;\underline{w'}]$ is a chain in $E$ if and only if for all $i>j$ we have $w'_i\notin\{qw'_j,tw'_j\}$. Moreover, any two such chains can be connected by a sequence of {\exc} transpositions.
\end{lemma}

\begin{proof}
The condition on $[\underline{w}']:=[w'_1, \dots, w'_k]$ is necessary by Lemma \ref{lem: far quotient}. We will prove it is also sufficient by induction on $k$. For $k=2$ this follows from Definition \ref{def: excellent new}. Suppose $k>2$ and $w'_j=w_k$. By assumption, if $w'_i=qw_k$ or $w'_i=tw_k$, then $i<j$.  We construct the sequence of {\exc} transpositions as follows.

First, note that $[w'_{k-1},\ldots,w'_{j+1},w'_{j-1},\ldots,w'_1]$ is a permutation of $[w_{k-1},\ldots,w_{1}]$ satisfying the conditions of the lemma. Second, by induction, we can find a sequence of {\exc} transpositions changing $[w_k,\ldots,w_1]$ to $[w_k,w'_{k-1},\ldots,w'_{j+1},w'_{j-1},\ldots,w'_1]$. Lastly, since none of the $w'_{k-1},\ldots,w'_{j+1}$  equal $qw_k$ or $tw_k$, using {\exc} transpositions $w_k$ can be moved past them to the $j^{th}$ position, as desired.
\end{proof}

We can think of $E$ as a graph $\Gamma_{E}$ with edges given by the covering relations. By Lemma \ref{lem: cycles} all the cycles in this graph are generated by the 4-cycles of the form ($y\notin \{qx,tx\}$): 
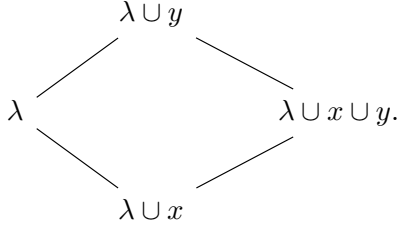
\begin{center}
\begin{tikzcd}
     & \lambda\cup y \arrow[dash]{dl}& \\
   \lambda  \arrow[dash]{dr}& & \lambda\cup x\cup y. \arrow[dash]{ul}\\
   & \lambda\cup x \arrow[dash]{ur}&
\end{tikzcd}
\end{center}

We will denote by $Q(\lambda;y,x)$ the following orientation of this cycle:
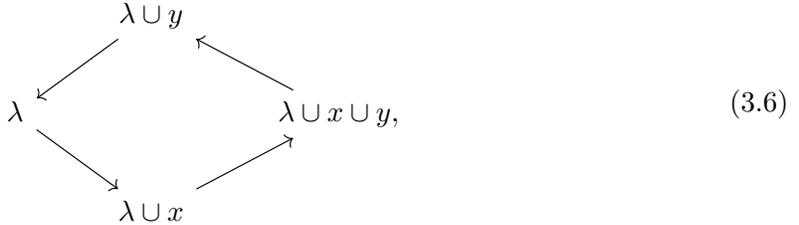
\begin{equation} \label{eq:Qcycle}
\begin{tikzcd}
     & \lambda\cup y \arrow{dl}& \\
   \lambda  \arrow{dr}& & \lambda\cup x\cup y, \arrow{ul}\\
   & \lambda\cup x \arrow{ur}&
\end{tikzcd}
\end{equation}
so that the first homology group $H^{1}(\Gamma_{E}, \Z)$ is generated by the elements $Q(\lambda; y,x)$. The following result gives a complete set of relations between these cycles.
\begin{lemma}
\label{lem: cycles relations}
All relations between the cycles in $\Gamma_E$ are given by 
$Q(\lambda;y,x)=-Q(\lambda;x,y), $
and
\begin{equation}
\label{eq: YBE}
Q(\lambda\cup x;z,y)+Q(\lambda;z,x)+Q(\lambda\cup z;y,x)=Q(\lambda;y,x)+Q(\lambda\cup y;z,x)+ Q(\lambda;z,y).
\end{equation}
\end{lemma}

\begin{proof}
As above, all chains from $\lambda$ to $\mu$ in $E$ are given by permutations of a single sequence $w_1,\ldots,w_k$
satisfying the conditions of Lemma \ref{lem: cycles}.

Any two such chains are related by a sequence of simple transpositions which correspond to $Q(\lambda;y,x)$. The relations between cycles correspond to the relations between simple transpositions, which are generated by braid relations. Visually, these correspond to the diagram:

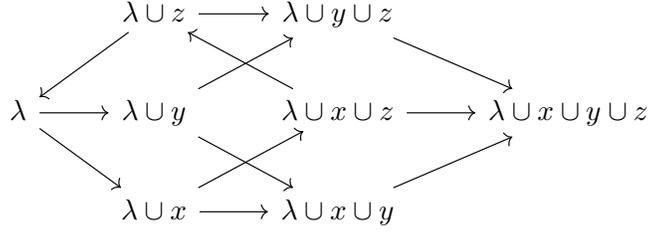
\begin{center}
\begin{tikzcd}
     & \lambda\cup z \arrow{dl} \arrow{r}& \lambda\cup y \cup z \arrow{dr} & \\
   \lambda  \arrow{dr} \arrow{r} & \lambda\cup y \arrow{ur}\arrow{dr} & \lambda\cup x\cup z \arrow{ul} \arrow{r}& \lambda\cup x\cup y\cup z\\
   & \lambda\cup x \arrow{ur} \arrow{r}& \lambda\cup x\cup y \arrow{ur} & \\
\end{tikzcd}
\end{center}
where there are two ways to go from the chain
$$
\lambda \to \lambda\cup x\to \lambda\cup x\cup y\to \lambda\cup x\cup y\cup z
$$
to the chain
$$
\lambda \to \lambda\cup z\to \lambda\cup y\cup z\to \lambda\cup x\cup y\cup z
$$
which indeed corresponds to \eqref{eq: YBE}.
\end{proof}

\begin{definition}
\label{def: good new}
We say that a chain $[\lambda; \underline{w}]$ is \newword{good} if $w_i\neq tw_j$ for all $i,j$.  
\end{definition}

Note that, if $[\lambda; \underline{w}]$ is a good chain and $E$ is an excellent poset, then the set of admissible transpositions for $[\underline{w}]$ coincides with that of {\exc} transpositions for $[\lambda; \underline{w}]$ (see Corollary \ref{cor: excellent chains}).

\begin{proposition}
\label{prop: good chains}
Suppose that $E$ is an excellent poset and $[\underline{w}] = [w_k,\dots,w_1]$. 

\noindent a) If $[\lambda; \underline{w}]$ is good, then both $[\lambda\cup w_k; w_{k-1},\ldots,w_{1}]$ and $[\lambda;w_k,w_{k-1},\ldots,w_2]$ are good.

\noindent b) Conversely, assume that both $[\lambda\cup w_k; w_{k-1},\ldots,w_{1}]$ and $[\lambda;w_k, w_{k-1},\ldots,w_2]$ are good. Then $[\lambda;\underline{w}]$ is good if and only if $w_1\neq tw_k$.

\noindent c) If $[\lambda;\underline{w}]$ is good then $w_1\neq qtw_k$.

\noindent d) If $[\lambda;\underline{w}]$ is good and $[\lambda;\underline{w'}]$ is a permuted chain, then 
they are related by a sequence of admissible transpositions through good chains.
\end{proposition}

\begin{proof}
Parts (a) and (b) are clear. Part (c) follows from Lemma \ref{lem: no qt}: indeed, if $w_1=qtw_k$ then $w_i=qw_k$ (and hence $w_1=tw_i$) or $w_i=tw_k$, so this is not a good chain. Finally, part (d) follows from Lemma \ref{lem: cycles}.
\end{proof}

\begin{lemma}
\label{lem: not good}
Given any chain $[\lambda;\underline{w}]$ that is not good, there exists a sequence of {\exc} transpositions transforming $[\lambda;\underline{w}]$ to a chain $[\lambda;\underline{w'}]$ for which $w'_i=tw'_{i+1}$ for some $i$.
\end{lemma}

\begin{proof}
If $[\lambda;\underline{w}]$ is not good then $w_i=tw_j$ for some $i<j$. By Lemma \ref{lem: cycles} we can use {\exc} transpositions to move $w_i$ and $w_j$ next to each other unless there exists $i<i'<j$ such that both $w_i\in \{qw_{i'},tw_{i'}\}$ and $w_{i'}\in \{qw_{j},tw_{j}\}$. However, this is impossible when $w_i=tw_j$.
\end{proof}

\subsection{Constructing Representations From Posets}

In this subsection we define a calibrated representation $V(E)$ starting from an excellent weighted poset $E$.
To do this, we will need the auxiliary data of \newword{edge functions}   $c(\lambda;x)\in \C(q,t)^{*}$, defined for any $\lambda\in E$ and any addable weight $x$ for $\lambda$. As always, we set $[\underline{w}]=[w_k,\dots, w_1]$ and write $[\underline{w},x] = [w_k,\dots, w_1,x]$ for the concatenation.

\begin{definition}
\label{def: calibrated from excellent}
Given an excellent weighted poset $E$ with edge functions $c(\lambda;x)\in \C(q,t)^{*}$, we define the space $V(E,c)=\bigoplus_{k=0}^{\infty}V_k(E,c)$ as follows:
\begin{itemize}
\item The collection of all good chains $[\lambda;\underline{w}]$ in $E$ forms a $\C(q,t)$-basis of $V_k(E,c)$.
\item The operators $z_i$ and $\Delta_{p_m}$ act diagonally:
\begin{align}
z_i[\lambda;\underline{w}]&=w_i[\lambda;\underline{w}], \\
\Delta_{p_m}[\lambda;\underline{w}]&=\left(p_m(\lambda)+w_1^m+\ldots+w_k^m\right)[\lambda;\underline{w}].
\end{align}
\item The operators $T_i$ are given by \eqref{eq: T Ram new}. 
\item The operator $d_{-}$ is given by
\begin{equation}
d_{-}[\lambda;\underline{w}]=[\lambda\cup w_k;w_{k-1}\ldots,w_{1}].
\end{equation}
\item The operator $d_{+}$ is given by
\begin{equation}\label{eq:d+ rep}
d_{+}[\lambda;\underline{w}]=\sum_{x}q^{k}c(\lambda\cup \underline{w};x)\prod_{i=1}^{k}\frac{x-tw_i}{x-qtw_i}[\lambda;\underline{w},x],
\end{equation}
where the sum is over all $x$ such that $[\lambda;\underline{w},x]$ is a good chain. 
\end{itemize}
\end{definition}

\begin{remark}
The operators in Definition \ref{def: calibrated from excellent} are well-defined. In particular, by Corollary \ref{cor: excellent chains} we have that
$[\lambda;s_i(\underline{w})]$ is a good chain whenever $[\lambda;\underline{w}]$ is a good chain and $s_i$ is an admissible transposition, hence $T_i$ is well-defined. Likewise, for $d_-$ Proposition \ref{prop: good chains}(a) ensures $[\lambda\cup w_k; w_{k-1}\ldots,w_{1}]$ is a good chain, and for $d_+$ Proposition \ref{prop: good chains}(c) guarantees $x\neq qtw_i$ for any good chain.    
\end{remark}

\begin{lemma}\label{lem: qphi}
For any choice of edge functions $c(\lambda;x)$, equation \eqref{eq:qphi}
holds in $V(E,c)$.
\end{lemma}

\begin{proof}
Assume $[\lambda;\underline{w}]$ is a good chain, then by Proposition \ref{prop: good chains}(a)
$[\lambda\cup w_k; w_{k-1},\ldots,w_{1}]$ is also good. If $[\lambda\cup w_k; w_{k-1},\ldots,w_{1},x]$ is good then by Proposition \ref{prop: good chains}(b) either $[\lambda;\underline{w},x]$ is good or $x=tw_k$. Let 
\begin{equation}
\label{eq: c from edges new}
c(\lambda;\underline{w},x):=q^{k}c(\lambda\cup \underline{w};x)\prod_{i=1}^{k}\frac{x-tw_i}{x-qtw_i}
\end{equation}
so that
$$d_{+}
[\lambda;\underline{w}]=\sum_{x}c(\lambda;\underline{w},x)[\lambda;\underline{w},x],$$
where summands with $x = tw_{k}$ are allowed, that is, summands where $[\lambda;\underline{w},x]$ is not good. Evidently, if $x\neq tw_k$ this recovers definition \eqref{eq:d+ rep} and if $x=tw_k$ then the coefficient $c(\lambda;\underline{w},x)$ vanishes.

Now the equation $z_1(qd_{+}d_{-}-d_{-}d_{+})=qt(d_{+}d_{-}-d_{-}d_{+})z_k $ is equivalent to 
\[
x\Bigl (qc(\lambda\cup w_k; w_{k-1}\ldots,w_{1},x)-c(\lambda;\underline{w},x)\Bigr)= qt\Bigl(c(\lambda\cup w_k;w_{k-1}\ldots,w_{1},x)-c(\lambda;\underline{w},x)\Bigr)w_k.
\]
Hence,  
$
q(x-tw_k)c(\lambda\cup w_k;w_{k-1}\ldots,w_{1},x)=(x-qtw_k)c(\lambda;\underline{w},x)
$
which clearly holds by \eqref{eq: c from edges new}. 
\end{proof}

An immediate consequence is an explicit formula for the operator $\varphi = \frac{1}{q-1}[d_{+}, d_{-}]$. 

\begin{corollary}
The action of the operator $\varphi$ is given by, 
\begin{equation}
\varphi[\lambda;\underline{w}]=-q^{k-1}\sum_{x}c(\lambda\cup \underline{w};x)\prod_{i=1}^{k-1}\frac{x-tw_i}{x-qtw_i}\cdot \frac{x}{x-qtw_k} [\lambda\cup w_k; w_{k-1}\ldots,w_{1},x]
\end{equation}
\end{corollary}
\begin{proof}
We have 
\begin{align*}
(d_{+}d_{-}-d_{-}d_{+})[\lambda;\underline{w}]&=q^{k-1}\sum_{x}\left(1-q\frac{x-tw_k}{x-qtw_k}\right) c(\lambda\cup w_k;w_{k-1}\ldots,w_{1},x)[\lambda\cup w_k; w_{k-1}\ldots,w_{1},x]
\\
&=q^{k-1}\sum_{x} \left( \frac{x(1-q)}{x-qtw_k} \right) c(\lambda\cup w_k; w_{k-1}\ldots,w_{1},x)[\lambda\cup w_k;w_{k-1}\ldots,w_{1},x].
\end{align*}
Now we apply \eqref{eq: c from edges new}.
\end{proof}

Until now, the the edge functions $c(\lambda;x)$ were completely arbitrary. However, in order to obtain representations of $\Bqt^{\ext}$ we need to impose additional conditions.

\begin{lemma}\label{lem: monodromy}
Suppose Definition \ref{def: calibrated from excellent} yields a representation of $\Bqt^{\ext}$. Then 
\begin{equation}
\label{eq: monodromy}
\frac{c(\lambda;x)c(\lambda\cup x;y)}{c(\lambda;y)c(\lambda\cup y;x)}=-\frac{(x-ty)(x-qy)(y-qtx)}{(y-tx)(y-qx)(x-qty)}
\end{equation}
whenever $[\lambda;x,y]$ and $[\lambda;y,x]$ are both chains. Furthermore, \eqref{eq: monodromy} implies $T_1d_+^2=d_{+}^2$.
\end{lemma}

\begin{proof}
 The coefficients of $[\lambda;\underline{w},x,y]$ and $[\lambda;\underline{w},y,x]$ in $d_+^2[\lambda; \underline{w}]$ are proportional to
\[
c(\lambda\cup \underline{w};x)c(\lambda\cup \underline{w}\cup x;y)\frac{y-tx}{y-qtx} \text{   and   } \ c(\lambda\cup \underline{w};y)c(\lambda\cup \underline{w}\cup y;x)\frac{x-ty}{x-qty}
\]
respectively, up to a factor of
\[
q^{2k+1}\prod \frac{x-tw_i}{x-qtw_i}\prod \frac{y-tw_i}{y-qtw_i}
\]
which is symmetric in $x$ and $y$. Setting $\lambda\cup \underline{w}=\mu$, the equation $T_1d_+^2=d_+^2$ can be rewritten as
\[
c(\mu;x)c(\mu\cup x;y)\frac{y-tx}{y-qtx}\frac{(q-1)x}{y-x}+c(\mu;y)c(\mu\cup y;x)\frac{x-ty}{x-qty}\frac{x-qy}{x-y}=c(\mu;x)c(\mu\cup x;y)\frac{y-tx}{y-qtx},
\]
so that
\[
c(\mu;y)c(\mu\cup y;x)\frac{x-ty}{x-qty}\frac{x-qy}{x-y}=c(\mu;x)c(\mu\cup x;y)\frac{y-tx}{y-qtx}\frac{y-qx}{y-x}
\]
and the result follows.

To show that \eqref{eq: monodromy} also implies $T_1d_{+}^{2} = d_{+}^{2}$, we need to consider the special case when $[\lambda;y,x]$ is not a chain in $E$. If $y=tx$ then $[\lambda;\underline{w},x,y]$ is not good and does not appear in $d_{+}^2$. If $y=qx$ then $[\lambda;\underline{w},x,y]$ is an eigenvector for $T_1$ with eigenvalue 1, so $T_1d_+^2=d_{+}^2$.
\end{proof}

\begin{theorem}
\label{thm: relations hold}
The edge functions $c(\lambda;x)$  define a calibrated representation $V(E,c)$ of $\Bqt^{\ext}$ if and only if they satisfy  \eqref{eq: monodromy}. 
\end{theorem}

\begin{proof}
If $V(E,c)$ is a representation of $\Bqt^{\ext}$ then the edge functions $c(\lambda;x)$ satisfy \eqref{eq: monodromy} by Lemma \ref{lem: monodromy}. Assume now that the edge functions satisfy \eqref{eq: monodromy}. Then, by Lemmas \ref{lem: qphi} and \ref{lem: monodromy}, equations \eqref{eq:T d+} and \eqref{eq:qphi} given by 
$
z_1(qd_{+}d_{-}-d_{-}d_{+})=qt(d_{+}d_{-}-d_{-}d_{+})z_k
$
and $T_1d_{+}^2=d_{+}^2$ are satisfied, with $d_{+}T_i=T_{i+1}d_{+}$ for $2\le i\le k-1$ following from \eqref{eq: c from edges new} since 
$c(\lambda;\underline{w},x)$ is symmetric in $w_1,\ldots,w_k$.

We check the remaining relations directly. 

The $\AH_k$ relations \eqref{eq:hecke relns}-\eqref{eq:T and z}  between $T_i$ and $z_j$ are satisfied by Theorem \ref{thm: Ram new}. The commutation relations \eqref{eq:ext-Bqt-1} between $\Delta_{p_m}$ and $T_i$ are clear since the eigenvalue of $\Delta_{p_m}$ is symmetric in $w_1,\ldots,w_k$. Likewise, the commutation relations \eqref{eq:ext-Bqt-2} between $\Delta_{p_m}, z_i$ and $d_{\pm}$ are clear.

The relation $d_{-}T_i=T_{i}d_{-}$ from \eqref{eq:T d-} is clear by construction. Let us verify $d_{-}^{2} = d_{-}^{2}T_{k-1}$. If $s_{k-1}$ is not admissible for $[\underline{w}]$ then $T_{k-1}[\lambda;  \underline{w}] = [\lambda;  \underline{w}]$ and we are done. On the other hand, if $s_{k-1}$ is admissible for $[\underline{w}]$ then:
\begin{align*}
d_{-}^2T_{k-1}[\lambda;\underline{w}]&=\frac{w_{k-1} - qw_{k} + (q-1)w_{k}}{w_{k-1} - w_{k}}[\lambda\cup w_k\cup w_{k-1};w_{k-2},\ldots,w_{1}] \\
&= [\lambda\cup w_k\cup w_{k-1};w_{k-2},\ldots,w_{1}] =d_{-}^2[\lambda;\underline{w}].
\end{align*}

It remains to verify \eqref{eq:phi}. For the first relation, $q\varphi d_{-}=d_{-}\varphi T_{k-1}$, we have:
\[
q\varphi d_{-}[\lambda;\underline{w}]=-q^{k-1}\sum_{x}c(\lambda\cup \underline{w};x )\prod_{i=1}^{k-2}\frac{x-tw_i}{x-qtw_i}\cdot \frac{x}{x-qtw_{k-1}}[\lambda\cup w_k\cup w_{k-1};w_{k-2},\ldots,w_{1},x]
\]
whereas,
\begin{align*}
d_{-}\varphi T_{k-1}=&-q^{k-1}\sum_{x}c(\lambda\cup \underline{w};x )\prod_{i=1}^{k-2}\frac{x-tw_i}{x-qtw_i}[\lambda\cup w_k\cup w_{k-1};w_{k-2},\ldots,w_{1},x]\times \\
&\left(
\frac{(q-1)w_k}{w_{k-1}-w_{k}}\frac{x}{x-qtw_k}\frac{x-tw_{k-1}}{x-qtw_{k-1}}+\frac{w_{k-1}-qw_{k}}{w_{k-1}-w_{k}}\frac{x}{x-qtw_{k-1}}\frac{x-tw_{k}}{x-qtw_{k}}\right). 
\end{align*}
These equations agree since
\[
\frac{(q-1)w_k(x-tw_{k-1})}{(w_{k-1}-w_{k})(x-qtw_k)}+\frac{(w_{k-1}-qw_{k})(x-tw_k)}{(w_{k-1}-w_{k})(x-qtw_k)}
=1.
\]

Now we check the remaining relation $T_{1}\varphi d_{+} = qd_{+}\varphi$. Assume first that $y\neq qx$. Then the coefficient of $[\lambda \cup w_{k}; w_{k-1},\ldots,w_{1},x,y]$ in $qd_{+}\varphi[\lambda; w]$ is
\[
-q^{2k}\prod_{i=1}^{k-1}\frac{(x-tw_i)(y-tw_i)}{(x-qtw_i)(y-qtw_i)}\times
c(\lambda\cup \underline{w},x)c(\lambda\cup \underline{w}\cup x;y)\frac{x}{x-qtw_k}\frac{y-tx}{y-qtx}
\]
while its coefficient in $T_{1}\varphi d_{+}$ equals
\begin{multline*}
-q^{2k}\prod_{i=1}^{k-1}\frac{(x-tw_i)(y-tw_i)}{(x-qtw_i)(y-qtw_i)}\times \\
\biggl(c(\lambda\cup \underline{w};x)c(\lambda\cup \underline{w}\cup x;y)\frac{(q-1)x}{y-x}\frac{(x-tw_k)(y-tx)y}{(x-qtw_k)(y-qtx)(y-qtw_k)} \\
+c(\lambda\cup \underline{w};y)c(\lambda\cup \underline{w}\cup y;x)\frac{(y-tw_k)(x-ty)x}{(y-qtw_k)(x-qty)(x-qtw_k)}\frac{x-qy}{x-y}\biggr)
\end{multline*}
By \eqref{eq: monodromy}, we have
\begin{align*}
c(\lambda\cup \underline{w};y)&c(\lambda\cup \underline{w}\cup y;x)\frac{(y-tw_k)(x-ty)x}{(y-qtw_k)(x-qty)(x-qtw_k)}\frac{x-qy}{x-y}\\
&=
-c(\lambda\cup \underline{w};x)c(\lambda\cup \underline{w}\cup x;y)\frac{(y-tx)(y-qx)(y-tw_k)x}{(y-qtx)(y-qtw_k)(x-qtw_k)(x-y)},
\end{align*}
so we can rewrite the coefficient of $[\lambda \cup w_k; w_{k-1},\ldots,w_{1},y]$ in $T_{1}\varphi d_{+}$ as follows:
\begin{multline*}
-q^{2k}\prod_{i=1}^{k-1}\frac{(x-tw_i)(y-tw_i)}{(x-qtw_i)(y-qtw_i)}c(\lambda\cup \underline{w};x)c(\lambda\cup \underline{w}\cup x;y) \times \\
\left(\frac{(q-1)x}{y-x}\frac{(x-tw_k)(y-tx)y}{(x-qtw_k)(y-qtx)(y-qtw_k)}-\frac{(y-tx)(y-qx)(y-tw_k)x}{(y-qtx)(y-qtw_k)(x-qtw_k)(x-y)}\right).
\end{multline*}

Now we can apply the identity of rational functions
$$
\frac{(q-1)(x-tw_k)y}{(y-x)(y-qtw_k)}-\frac{(y-qx)(y-tw_k)}{(y-qtw_k)(x-y)}=1
$$
and the result follows.

In the remaining case $y=qx$ the coefficient in $qd_{+}\varphi[\lambda;\underline{w}]$ is  proportional to 
$$
\frac{x}{x-qtw_k}\frac{y-tx}{y-qtx}=\frac{x(q-t)}{(x-qtw_k)(q-qt)}
$$
while the coefficient in $T_1\varphi d_{+}[\lambda;\underline{w}]$ is proportional  to
$$
\frac{(x-tw_k)(y-tx)y}{(x-qtw_k)(y-qtx)(y-qtw_k)}=
\frac{(x-tw_k)(q-t)qx}{(x-qtw_k)(q-qt)(qx-qtw_k)}=
\frac{x(q-t)}{(x-qtw_k)(q-qt)}.
$$
\end{proof}

Due to the previous theorem, any excellent poset together with the choice of edge functions satisfying \eqref{eq: monodromy} gives rise to a representation of the algebra $\Bqt^{\ext}$ via Definition \ref{def: calibrated from excellent}. Moreover, as the next result shows, we can always choose edge functions $c(\lambda;x)$ satisfying \eqref{eq: monodromy}. 

\begin{theorem}
\label{thm: existence and uniqueness}
For any given excellent poset $E$ there exist nonzero coefficients $c(\lambda;x)$ satisfying \eqref{eq: monodromy}. Moreover, any two choices of edge functions $c(\lambda;x)$ satisfying  \eqref{eq: monodromy} yield isomorphic representations of $\Bqt^{\ext}$.
\end{theorem}

As an immediate consequence of this theorem we have the following.  

\begin{corollary}
For any excellent poset $E$, there exists a calibrated representation $V(E,c)$ of $\Bqt^{\ext}$ constructed as in Definition \ref{def: calibrated from excellent}. In particular, any two such representations are isomorphic so that $V(E,c) \cong V(E,c')$ for any family of edge functions $c,c'$. 
\end{corollary}

\begin{proof}[Proof of Theorem \ref{thm: existence and uniqueness}]

As in Lemma \ref{lem: cycles relations} we visualize $E$ as the graph $\Gamma_{E}$ with edges given by the covering relations. The homology $H_1(\Gamma_E;\Z)$ is generated by the cycles $Q(\lambda;y,x)$ (see \eqref{eq:Qcycle}) modulo relations from Lemma \ref{lem: cycles relations}.

We claim that there exists a unique homomorphism
\[
\Upsilon:H_1(\Gamma_E,\Z)\to \C(q,t)^*
\]
such that 
\[
\Upsilon[Q(\lambda;y,x)]=\Upsilon[\lambda;y,x]:=-\frac{(x-ty)(x-qy)(y-qtx)}{(y-tx)(y-qx)(x-qty)}.
\]
Indeed, $\Upsilon[\lambda;x,y]=\Upsilon[\lambda;y,x]^{-1}$ and 
\begin{align*}
\Upsilon[\lambda\cup x;z,y]&\Upsilon[\lambda;z,x]\Upsilon[\lambda\cup z;y,x] \\
&=
-\frac{(x-ty)(x-tz)(y-tz)(x-qy)(x-qz)(y-qz)(y-qtx)(z-qtx)(z-qty)}{(y-tx)(z-tx)(z-ty)(y-qx)(z-qx)(z-qy)(x-qty)(x-qtz)(y-qtz)} \\
&=
\Upsilon[\lambda;y,x]\Upsilon[\lambda\cup y;z,x]\Upsilon[\lambda;z,y]
\end{align*}
so the equation \eqref{eq: YBE} is satisfied. 

Equation \eqref{eq: monodromy} can then be understood as follows. The coefficients $c(\lambda;x)$ should be interpreted as a cocycle $c:C_1(\Gamma_E,\Z)\to \C(q,t)^*$ which assigns $c(\lambda;x)$ to the oriented edge from $\lambda$ to $\lambda\cup x$ and $c(\lambda;x)^{-1}$ to the oppositely oriented edge. Equation \eqref{eq: monodromy} also requires that the value of $c$ on the cycle $Q(\lambda;y,x)$ equals $\Upsilon[\lambda;y,x]$, hence the value of $c$ on any cycle equals the value of $\Upsilon$ on that cycle. 

Such $c(\lambda;x)$ can be constructed in the following manner: Pick a spanning tree for each connected component of $\Gamma_E$ and assign the values of $c(\lambda;x)\in \C(q,t)^*$ to its edges arbitrarily. In this way, any edge not in the spanning tree forms a cycle, and its value is determined by $\Upsilon$. By the above discussion, this assignment is well defined. 

Now, assume that $c(\lambda;x)$ and $c'(\lambda;x)$ are two such choices. We claim there exist scalars $a(\lambda)\in \C(q,t)^*$ for which 
$$
c(\lambda;x)=c'(\lambda;x)\frac{a(\lambda)}{a(\lambda\cup x)}.
$$
Indeed, denoting $\phi(\lambda;x):=c(\lambda;x)/c'(\lambda;x)$, then \eqref{eq: monodromy} implies
$$
\phi(\lambda;x)\phi(\lambda\cup x;y)=\phi(\lambda;y)\phi(\lambda\cup y;x).
$$

Once again, thinking of $\phi(\lambda,x)$ as a 1-cocycle on the edges of $E$, so that its value on any cycle equals 1, we obtain that $\phi$ is a coboundary of some 0-cochain $a$ on the vertices of $E$.

It remains to note that the basis can be rescaled by 
$$
[\lambda;\underline{w}]\to a(\lambda\cup \underline{w})[\lambda;\underline{w}].
$$
This is symmetric in $w_i$, hence the action of $T_i$ and $z_i$ remains unaltered. Clearly, it does not modify the coefficients of $d_{-}$ as well. However, the coefficients of $d_{+}$ are changed from
$$
c(\lambda;\underline{w},x)=c(\lambda\cup \underline{w};x)\prod\frac{x-tw_i}{x-qtw_i}
$$
to
$$
c'(\lambda\cup \underline{w};x)\prod\frac{x-tw_i}{x-qtw_i}
$$
as desired.
\end{proof}

\begin{remark}
\label{rem: shift rep}
Given an excellent weighted poset $E$ and a collection of scalars $a_m$, we can define a new weighted poset $E'$ as in Remark \ref{rem: shift}, which will be excellent as well. Since the addable weights do not change under shifting, all the constructions will transfer onto $E'$ verbatim and yield the same formulas for operators in $\Bqt$. The only modification occurs with the action of $\Delta_{p_m}$ which is changed to $\Delta_{p_m}+a_m$. 
\end{remark}

\subsection{Rescaling the Basis}

In the above construction, the basis $[\lambda;\underline{w}]$ was adjusted with respect to $d_{-}$ so that all the coefficients of $d_{-}$ equal 1. Sometimes it is useful to rescale the basis in such a way that instead all coefficients of $d_{+}$ are equal to 1. This is achieved as follows.

\begin{lemma}
\label{lem: gamma}
Let
\begin{equation}
\gamma(\lambda;\underline{w}):=c(\lambda;w_k)c(\lambda;w_k,w_{k-1})\cdots c(\lambda;\underline{w}).
\end{equation}
\noindent a) The following identities hold, 
\begin{align}
\frac{\gamma(\lambda;\underline{w})}{\gamma(\lambda\cup w_k;w_{k-1},\ldots,w_{1})}&=q^{k-1}c(\lambda;w_k)\prod_{i=1}^{k-1}\frac{w_i-tw_k}{w_i-qtw_k},
\\
\frac{\gamma(\lambda;\underline{w})}{\gamma(\lambda;w_{k-1},\ldots,w_{1},x)}&=c(\lambda;\underline{w},x).
\end{align}

\noindent b) Assume that $s_i$ is admissible for $\underline{w}$. Then,
\begin{equation}
\frac{\gamma(\lambda;\underline{w})}{\gamma(\lambda;s_i(\underline{w}))}=\frac{qw_i-w_{i+1}}{w_i-qw_{i+1}}.
\end{equation}

\end{lemma}

\begin{proof}
a) Observe that by \eqref{eq: c from edges new},
\[
\gamma(\lambda;\underline{w})=q^{\frac{k(k-1)}{2}}c(\lambda;w_k)c(\lambda\cup w_k;w_{k-1})\cdots c(\lambda\cup w_k\cup\ldots w_{2};w_1)\prod_{i<j}\frac{w_i-tw_j}{w_i-qtw_j}.
\]
This implies the first equation in (a), the second is clear.

b) Let $\mu=\lambda\cup w_k\cdots \cup w_{i+2}$. Then we have
$$
\frac{\gamma(\lambda;\underline{w})}{\gamma(\lambda;s_i(\underline{w}))}=\frac{c(\mu;w_{i+1})c(\mu\cup w_{i+1};w_{i})}{c(\mu;w_{i})c(\mu\cup w_{i};w_{i+1})}\times \frac{(w_i-tw_{i+1})(w_{i+1}-qtw_{i})}{(w_{i+1}-tw_{i})(w_{i}-qtw_{i+1})}.
$$
By \eqref{eq: monodromy} we can rewrite this as
$$
-\frac{(w_{i+1}-tw_{i})(w_{i+1}-qw_{i})(w_{i}-qtw_{i+1})}{(w_i-tw_{i+1})(w_i-qw_{i+1})(w_{i+1}-qtw_i)}\times \frac{(w_i-tw_{i+1})(w_{i+1}-qtw_{i})}{(w_{i+1}-tw_{i})(w_{i}-qtw_{i+1})}=
-\frac{w_{i+1}-qw_{i}}{w_i-qw_{i+1}}.
$$
\end{proof}

We can now define the \newword{rescaled basis}:
\begin{equation}\label{eq:resbasis}
[\lambda;\underline{w}]^{\resc}:=\gamma(\lambda;\underline{w})[\lambda;\underline{w}]
\end{equation}

\begin{theorem}
\label{thm: rescale}
The $\Bqt^{\ext}$ action on the rescaled basis \eqref{eq:resbasis} is as follows.
\begin{itemize}
\item The operators $z_i$ and $\Delta_{p_m}$ act diagonally:
\begin{align} 
z_i[\lambda;\underline{w}]^{\resc}&=w_i[\lambda;\underline{w}]^{\resc},\\
 \Delta_{p_m}[\lambda;\underline{w}]^{\resc}&=\left(p_m(\lambda)+w_1^m+\ldots+w_k^m\right)[\lambda;\underline{w}]^{\resc}.
\end{align}
\item The operators $T_i$ are given by  
\begin{align}
T_i[\lambda;\underline{w}]^{\resc}=
\begin{cases}
\frac{(q-1)w_{i+1}}{w_i-w_{i+1}}[\lambda;\underline{w}]^{\resc}+\frac{qw_i-w_{i+1}}{w_i-w_{i+1}}[\lambda;s_i(\underline{w})]^{\resc} & \text{; if $s_i$ is admissible} 
\\
\frac{(q-1)w_{i+1}}{w_i-w_{i+1}}[\lambda;\underline{w}]^{\resc} & \text{; else.}
\end{cases}
\end{align}
\item The operator $d_{-}$ is given by
\begin{equation}
d_{-}[\lambda;\underline{w}]^{\resc}=q^{k-1}c(\lambda;w_k)\prod_{i=1}^{k-1}\frac{w_i-tw_k}{w_i-qtw_k}[\lambda\cup w_k;w_{k-1}\ldots,w_{1}]^{\resc}.
\end{equation}
\item The operator $d_{+}$ is given by
\begin{equation}
d_{+}[\lambda;\underline{w}]^{\resc}=\sum_{x}[\lambda;\underline{w},x]^{\resc}
\end{equation}
where the sum is over all $x$ such that $[\lambda;\underline{w},x]$ is a good chain.  
\end{itemize}
\end{theorem}

\begin{proof}
The action of $z_i$ and $\Delta_{p_m}$ does not change after rescaling their eigenvectors. The rest follow from direct computation. 

For $T_i$ by Lemma \ref{lem: gamma}(b) we have, 
\begin{align*}
T_i[\lambda;\underline{w}]^{\resc}&=\frac{(q-1)w_{i+1}}{w_i-w_{i+1}}[\lambda;\underline{w}]^{\resc}+\frac{w_i-qw_{i+1}}{w_i-w_{i+1}}\frac{\gamma(\lambda;\underline{w})}{\gamma(\lambda;s_i(\underline{w}))}[\lambda;s_i(\underline{w})]^{\resc}
\\
&=\frac{(q-1)w_{i+1}}{w_i-w_{i+1}}[\lambda;\underline{w}]^{\resc}+\frac{qw_i-w_{i+1}}{w_i-w_{i+1}}[\lambda;s_i(\underline{w})]^{\resc}.
\end{align*}
Likewise, by Lemma \ref{lem: gamma}(a) we get, 
\begin{align*}
d_{-}[\lambda;\underline{w}]^{\resc}&=\frac{\gamma(\lambda;\underline{w})}{\gamma(\lambda\cup w_k;w_{k-1},\ldots,w_{1})}[\lambda\cup w_k;w_{k-1},\ldots,w_{1}]^{\resc}
\\
&=q^{k-1}c(\lambda;w_k)\prod_{i=1}^{k-1}\frac{w_i-tw_k}{w_i-qtw_k}[\lambda\cup w_k;w_{k-1},\ldots,w_{1}]^{\resc}
\end{align*}
and
\[
d_{+}[\lambda;\underline{w}]^{\resc}=\sum_{x}\frac{\gamma(\lambda;\underline{w})}{\gamma(\lambda;\underline{w},x)}c(\lambda;\underline{w},x)[\lambda;\underline{w},x]^{\resc}=
\sum_{x}[\lambda;\underline{w},x]^{\resc}.
\]
\end{proof}

\subsection{Examples of $\Bqt^{ext}$ Representations} We now exemplify the various constructions above. 
\subsubsection{Trivial Representations}
\label{sec: trivial}

Let $E=\{\bullet\}$ be a one-element set. Any arbitrary choice of scalars $a_k=p_k(\bullet)$ will endow $E$ with the structure of an excellent weighted poset. The corresponding representation has $V_0=\langle [\bullet]\rangle$ and $V_k=0$ for $k>0$. Thus, the operators in $\Bqt$ vanish (and relations hold tautologically) with $\Delta_{p_m}[\bullet]=a_m[\bullet]$.

\subsubsection{Linearly Ordered Sets}  
\label{sec: linear posets}
Take $n \in \Z_{>0}\cup\{\infty\}$ and consider the poset $E = \{m \in \Z_{\geq 0} \mid m < n\}$, so that $E$ is the interval $[0, n-1]$ if $n \in \Z_{>0}$, or $\Z_{\geq 0}$ if $n = \infty$. We endow $E$ with the linear order inherited from $\Z$.

Let $f \mapsto f(m)$ be a weighting on the poset $E$, and for each $i \in E$ let $x_{i}$ be such that
\[
p_{m}(i+1) = p_{m}(i) + x_{i}^{m} \; \text{for all} \; m.
\]
Notice that the poset $E$ is excellent if and only if for every $i$, $x_{i+1} \in \{qx_{i}, tx_{i}\}$. So assume that $x_{i} = q^{i}x_{0}$ for $i \in E$. In particular, $E$ is excellent and every chain in $E$ is a good chain.

Suppose, additionally, that $E$ is finite with cardinality $\# E = n$. For $i = 0, \dots, n-1-k$ we denote by $[i, i+k]$ the chain
\[
i \rightarrow i+1 \rightarrow \dots \rightarrow i+k.
\]
Thus, for $k \geq n$ we have $V_{k}(E) = 0$, and for $k<n$ the space $V_{k}(E)$ has a $\C(q,t)$-basis $\{[i, i+k] \mid i = 0, \dots, n-1-k\}$.
Namely, 
\begin{equation}\label{eqn: operators linearly ordered}
T_{j}[i, i+k] = [i, i+k], \quad z_{p}[i, i+k] = x_{0}q^{i+k-p}[i, i+k], \quad d_{-}[i, i+k] = [i+1, i+k].
\end{equation}

Now, because the graph $\Gamma_{E}$ is a tree, according to (the proof of) Theorem \ref{thm: existence and uniqueness}, we can arbitrarily choose nonzero elements $c(i) \in \C(q,t)$, $i = 0, \dots, n-2$, such that \eqref{eqn: operators linearly ordered} together with
\[
d_{+}[i, i+k] = \begin{cases} \frac{q^{k}-t}{1-t}c(i+k)[i, i+k+1] &; i+k+1 \in E, \\0 & \text{; else}\end{cases}
\]
defines a representation of $\Bqt$. We verify this directly. 

Since all $T$'s act by $1$ and it is clear that $z_{p}[i, i+k] = q^{-1}z_{p+1}[i, i+k]$, we need only check relations \eqref{eq:phi}, \eqref{eq: d z} and \eqref{eq:qphi}. Relations \eqref{eq: d z} are clear. For \eqref{eq:phi}, since $T$'s act by $1$ we need to verify that $q\varphi d_{-} = d_{-}\varphi$ and that $\varphi d_{+} = qd_{+}\varphi$. We first note that, 
\[
(d_{+}d_{-} - d_{-}d_{+})[i, i+k] = d_{+}[i+1, i+k] - c(i+k)\frac{q^{k}-t}{1-t}d_{-}[i, i+k+1] = c(i+k)q^{k-1}\frac{1-q}{1-t}[i+1, i+k+1]
\]
if $i+k+1 \in E$ and zero else. From here, both relations in \eqref{eq:phi} follow easily. Now, notice that,
\[
qt(d_{+}d_{-} - d_{-}d_{+})z_{k}[i, i+k] = qtx_{0}q^{i}q^{k-1}c(i+k)\frac{1-q}{1-t}[i+1, i+k+1]
\]
whereas
\[
z_{1}(qd_{+}d_{-} - d_{-}d_{+})[i,i+k] =  x_{0}q^{i+k}\left(c(i+k)\frac{q^{k} - qt}{1-t} - \frac{q^{k} - t}{1-t}\right)[i+1, i+k+1]
\]
if $i+k+1 \in E$, in which case both expressions are equal. If $i+k+1 \not \in E$, then both 
\[qt(d_{+}d_{-} - d_{-}d_{+})z_{k}[i, i+k] \text{  and  }z_{1}(qd_{+}d_{-} - d_{-}d_{+})[i, i+k]\]
 are zero, and the result follows. 

\begin{remark}
   In particular, if $E$ is a singleton, then $V(E)$ is a trivial representation as considered in the previous subsection. 
\end{remark}

It is interesting to note that, \emph{up to constants}, we can picture the action of $d_{+}$ and $d_{-}$ on $V(E)$ as in Figure \ref{fig: interval}. From here, it is easy to obtain the socle filtration of $V(E)$ -- the corresponding subquotients are given by horizontal slices in the figure.  As we will see in Section \ref{sec: duality} below, 
the horizontal symmetry of Figure \ref{fig: interval} is explained by the fact that the poset $E$ is isomorphic to its own opposite poset.

\begin{figure}[ht]
\begin{tikzpicture}[scale=0.5]
    \node at (0,0) {$[0]$}; 
    \draw[->] (0.5,0) to (4.5, -0.8);
    \node at (0, -2) {$[1]$};
    \draw[->] (0.5, -2) to (4.5, -2.8);
    \node at (0, -4) {$[2]$}; 
    \draw[->] (0.5, -4) to (4.5, -4.8);
    \node at (0, -6) {$\vdots$};
    \draw[->] (0.5, -6) to (3.25, -6.8);
    \node at (0, -8) {$[n-3]$};
    \draw[->] (1.25, -8.1) to (3.25, -8.8);
    \node at (0, -10) {$[n-2]$};
    \draw[->] (1.25, -10.1) to (3.25, -10.8);
    \node at (0, -12) {$[n-1]$};
    \node at (0, -15) {$V_0(E)$};

    \node at (5.5, -1) {$[0, 1]$};
    \draw[->] (4.5, -1) to (0.5, -1.8);
    \draw[->] (6.5, -1) to (10, -1.8);
    \node at (5.5, -3) {$[1, 2]$};
    \draw[->] (4.5, -3) to (0.5, -3.8);
    \draw[->] (6.3, -3) to (10, -3.8);
    \node at (5.5, -5) {$\vdots$};
    \draw[->] (6, -5) to (8.8, -5.8);
    \node at (5.5, -7) {$[n-4,n-3]$};
    \draw[->] (3.25, -7) to (1.25, -8);
    \draw[->] (7.8, -7) to (8.8, -7.8);
    \node at (5.5, -9) {$[n-3, n-2]$};
    \draw[->] (3.25, -9) to (1.25, -10);
    \draw[->] (7.8, -9) to (8.8, -9.8);
    \node at (5.5, -11) {$[n-2, n-1]$};
    \draw[->] (3.25, -11) to (1.25, -12);
    \node at (5.5, -15) {$V_1(E)$};

    \node at (11, -2) {$[0, 2]$};
    \draw[->] (10, -2.1) to (6.5, -2.9);
    \node at (11, -4) {$\vdots$};
    \node at (11, -6) {$[n-5, n-3]$};
    \draw[->] (8.8, -6) to (7.8, -6.8);
    \node at (11, -8) {$[n-4, n-2]$};
    \draw[->] (8.8, -8) to (7.8, -8.8);
    \node at (11, -10) {$[n-3, n-1]$};
    \draw[->] (8.8, -10) to (7.8, -10.8);
    \node at (11, -15) {$V_2(E)$};

    \node at (15, -6) {$\dots$};
    \node at (15, -15) {$\dots$};

    \node at (18, -5) {$[0,n-2]$};
    \draw[->] (19.5, -5) to (21.5, -5.8);
    \node at (18, -7) {$[1, n-1]$};
    \node at (18, -15) {$V_{n-2}(E)$};

    \node at (23, -6) {$[0, n-1]$};
    \draw[->] (21.5, -6) to (19.5, -7);
    \node at (23, -15) {$V_{n-1}(E)$};
    \end{tikzpicture}
\caption{The structure of the representation $V(E)$ when $E$ is the interval $0 < 1 < \dots < n-1$. Up to scalars, $d_{+}$ is indicated by the arrows going right, and $d_{-}$ is indicated by the arrows going left.} \label{fig: interval}
\end{figure}
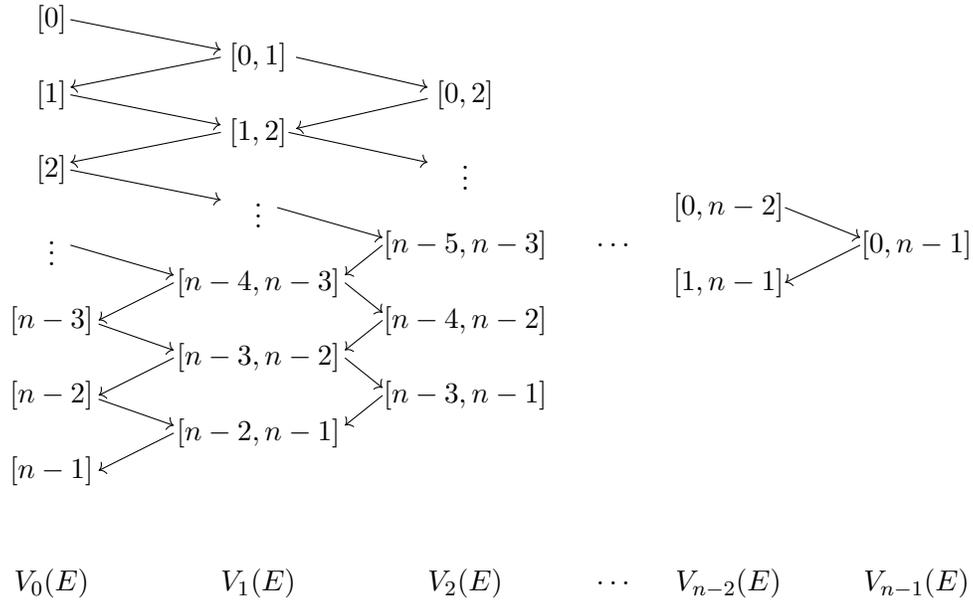

\subsubsection{The Polynomial Representation}\label{sec: poly} Recall from Section \ref{sec: polynomial} the polynomial representation of $\Bqt$ originally considered by Carlsson-Gorsky-Mellit \cite{CGM, Shuffle} that arose in relation to the shuffle theorem. It turns out, this representation can be recovered from a particular choice of poset. As earlier, we abuse notation and denote by $\sq$ both the cell in a partition $\lambda$ and its $(q,t)$-content. 

\begin{theorem}
\label{thm: polynomial is calibrated}
The polynomial representation of $\Bqt$ constructed in \cite{CGM, Shuffle} is isomorphic to $V(\mathcal{P})$ where $\mathcal{P}$ is the poset of all partitions with weighting $p_m(\lambda)=\sum_{\sq\in \lambda}\sq^m$. 
\end{theorem}

\begin{proof}
Let $\mathcal{P}$ be the poset of all partitions and note that addable weights correspond to addable boxes for a partition (which we identify with their $(q,t)$-content). By Example \ref{ex:Epartition} $\mathcal{P}$ is excellent. 

In particular, this means that chains from $\lambda$ to $\mu$ in $\mathcal{P}$ are in bijection with standard skew tableaux of shape $\mu\setminus \lambda$. Hence, a chain is good if and only if $\mu\setminus \lambda$ is a horizontal strip. Note that this recovers the description in \cite[\S 4.2]{CGM}. Thus, it suffices to prove the $\Bqt$ action in \cite{CGM} coincides with that in Definition \ref{def: calibrated from excellent}.

The actions of $z_i$ and $T_i$ in \cite[Lemma 5.2.1]{CGM} directly agree with \eqref{eq: T Ram new}. Likewise, in \cite[Lemma 5.3.1]{CGM} the coefficients of $d_{-}$ are all equal to $1$ in accord with Definition \ref{def: calibrated from excellent}. Now, by \cite[Lemma 5.3.1, eq. (5.3.2)]{CGM} the action of $d_+$ is given by
\[
d_{+}[\lambda;\underline{w}]=-q^k\sum_{x}xD(\lambda\cup \underline{w};x)\prod_{i=1}^{k}\frac{x-tw_i}{x-qtw_i}[\lambda; \underline{w},x],
\]
where for a partition $\lambda$ with addable box $x$ we define
\begin{equation}
D(\lambda;x):=x^{-1}\Lambda(-x^{-1}+(1-q)(1-t)B_{\lambda}x^{-1}+1),
\end{equation}
where $B_{\lambda}=\sum_{\sq\in \lambda}\sq$ and $\Lambda\left(\sum \phi_{i,j}q^it^j\right)=\prod (1-q^it^j)^{\phi_{i,j}}$, see Definition \ref{def:Lambda-operator} below.

Note that $D(\lambda;x)$ is well defined and nonzero. Indeed, if $\lambda$ is empty then $B_{\lambda}=0$ and $x=1$, so $D(\lambda,x)=\Lambda(0)=1$. Otherwise, observe that the constant term in $(1-q)(1-t)B_{\lambda}x^{-1}$ equals $(-1)$ for any addable box $x$ since either: 
$\lambda$ contain three boxes with $(q,t)$-weights $q^{-1}x,t^{-1}x,(qt)^{-1}x$, or one box with weight $q^{-1}x$, or one box with weight $t^{-1}x$. In particular, the constant term in $(-x^{-1}+(1-q)(1-t)B_{\lambda}x^{-1}+1)$ vanishes.

As a consequence, the coefficients of $d_{+}$ satisfy \eqref{eq: c from edges new} with edge functions given by
\begin{equation}
\label{eq: c for partitions}
c(\lambda;x)=-\Lambda(-x^{-1}+(1-q)(1-t)B_{\lambda}x^{-1}+1).
\end{equation}
By Theorem \ref{thm: relations hold} it remains to check that these edge functions satisfy \eqref{eq: monodromy}. Computing directly we see that, 
$$
c(\lambda;x)c(\lambda\cup x;y)=C_{xy}\Lambda((1-q)(1-t)xy^{-1})=C_{xy}\frac{(1-xy^{-1})(1-qtxy^{-1})}{(1-qxy^{-1})(1-txy^{-1})}=
C_{xy}\frac{(y-x)(y-qtx)}{(y-qx)(y-tx)}
$$
where $C_{xy}= \Lambda((x^{-1}y^{-1})(1+(1-q)(1-t)B_\lambda)+2)$ is symmetric in $x$ and $y$, which implies
$$
\frac{c(\lambda;x)c(\lambda\cup x;y)}{c(\lambda;y)c(\lambda\cup y;x)}=\frac{(y-x)(y-qtx)(x-qy)(x-ty)}{(y-qx)(y-tx)(x-y)(x-qty)}=-\frac{(y-qtx)(x-qy)(x-ty)}{(y-qx)(y-tx)(x-qty)}.
$$
\end{proof}

\section{The Structure of $V(E)$}\label{sec: structure}

\subsection{Ideals and Coideals} Recall that a nonempty subset $I$ of a poset $E$ is said to be an \newword{ideal} (resp. \newword{coideal}) if for any $\lambda \in I$ and $\mu \leq \lambda$ (resp. $\mu \geq \lambda$) one has $\mu \in I$. If $E$ is a weighted poset, then so is $I$ as it inherits the weights from $E$. For the remainder of this section assume $E$ is an excellent weighted poset. 

\begin{proposition}\label{prop: subs and quotients}
Let $E$ be an excellent weighted poset and $I \subseteq E$.
\begin{enumerate}
    \item[(1)] If $I$ is an ideal, then $I$ is excellent and $V(I)$ is a quotient of $V(E)$.
    \item[(2)] If $I$ is a coideal, then $I$ is excellent and $V(I)$ is a submodule of $V(E)$. 
\end{enumerate}
\end{proposition}
\begin{proof}
(1) Let $[\lambda; \underline{w}]$ be a chain in $E$. Then, every element in the chain $[\lambda; \underline{w}]$ belongs to $I$ if and only if $\lambda \cup \underline{w} \in I$. From here, it follows easily that $I$ is excellent. Now, notice that
\[
V^{I} := \C(q,t)\Bigl\{ [\lambda;\underline{w}] \mid [\lambda; \underline{w}] \; \text{is a good chain in} \; E \; \text{and} \; \lambda \cup \underline{w}\not\in I \Bigr\}
\] 
 is a $\Bqt^{ext}$-submodule of $V(E)$. Indeed, the operators $d_{-}, z_{i}$ and $T_{i}$ do not change the endpoint of a chain, and the fact that $V^{I}$ is closed under $d_{+}$ follows easily from the fact that $I$ is an ideal. It is clear that $V(I) \cong V(E)/V^{I}$.

(2) Let $[\lambda; \underline{w}]$ be a chain in $E$. Then, every element in the chain $[\lambda; \underline{w}]$ belongs to $I$ if and only if $\lambda \in I$. From this, it is easy to see that the 
\[
V(I)=\C(q,t)\Bigl\{[\lambda; \underline{w}] \mid [\lambda; \underline{w}] \; \text{is a good chain in} \; E \; \text{and} \; \lambda \in I \Bigr\}
\]
is a $\Bqt^{ext}$-submodule of $V(E)$.
\end{proof}

\begin{remark}
If $I \subseteq E$ is an ideal, then $E \setminus I$ is a coideal. In particular, its associated complex
\[
0 \to V(E \setminus I) \to V(E) \to V(I) \to 0
\]
is not exact but has homology $H = \bigoplus_{k \geq 0}H_{k}$ with $H_{0} = 0$. Hence, $H$ is a calibrated representation of $\Bqt^{\ext}$ that does not come from a poset. Note also that, if $K = \ker(V(E) \to V(I))$ then $V(E \setminus I) \subsetneq K$, but $V(E\setminus I)_{0} = K_{0}$.
See also Remark \ref{rem: U as kernel } below.
\end{remark}

The previous remark implies, in particular, that \emph{not} every $\Bqt^{\ext}$-submodule of $V(E)$ is of the form $V(I)$ for a coideal $I \subseteq E$. We will see, however, that it is possible to extract a coideal from an arbitrary $\Bqt^{\ext}$-submodule of $V(E)$ which in turn gives bounds for the submodule. We start with the following standard lemma. 

\begin{lemma}\label{lem: vandermonde}
Let $W = \bigoplus W_{k}$ be a $\Bqt^{\ext}$-submodule of $V(E)$. Suppose that there exist pairwise distinct good chains $[\lambda_1; \underline{w}^{(1)}], \dots, [\lambda_s; \underline{w}^{(s)}] \in V_{k}(E)$ and nonzero scalars $a_1, \dots, a_s$ such that
\[
\sum_{i = 1}^{s}a_{i}\left[\lambda_{i}; \underline{w}^{(i)}\right] \in W_{k}.
\]
Then, $[\lambda_{i}; \underline{w}^{(i)}] \in W_{k}$ for every $i = 1, \dots, s$.
\end{lemma}
\begin{proof}
  Since the $\Delta$ operators together with the $z$-operators have joint simple spectrum, this follows from a standard Vandermonde argument. 
\end{proof}

This implies, in particular, that if $W \subseteq V(E)$ is a $\Bqt^{\ext}$-submodule, then $W_{0} = \bigoplus_{[\lambda] \in W_0}\C(q,t)[\lambda]$.

\begin{lemma}\label{lem: producing a coideal}
Let $W \subseteq V(E)$ be a nonzero $\Bqt^{\ext}$-submodule, and $I_{W} := \bigr\{\lambda \mid [\lambda] \in W_0\bigr\} \subseteq E$. Then, $I_{W}$ is a nonempty coideal of $E$.
\end{lemma}
\begin{proof}
The fact that $I_{W}$ is nonempty, provided $W$ is nonzero, follows from Lemma \ref{lem: vandermonde}.
Indeed, $W$ must contain a basis vector $[\lambda;\underline{w}]\in W_k$ for some $k$ and hence it contains $d_-^k[\lambda;\underline{w}]=[\lambda\cup \underline{w}]\in W_0.$

Now, if $\lambda \in I_{W}$ and $x$ is addable for $\lambda$, then $[\lambda; x]$ appears in $d_{+}\lambda$, so $[\lambda; x] \in W_{1}$ and therefore $[\lambda \cup x] = d_{-}[\lambda; x] \in W_{0}$. It follows that $I_{W}$ is a coideal of $E$. 
\end{proof}

\begin{corollary}
    Assume that for every $\lambda, \mu \in E$ there exists $\nu$ such that $\lambda, \mu \leq \nu$. Then, $V(E)$ is indecomposable as a $\Bqt^{\ext}$-module. 
\end{corollary}
\begin{proof}
We show the stronger property that any two nonzero submodules $W, W' \subseteq V(E)$ intersect nontrivially. Indeed, it suffices to show that $I_{W} \cap I_{W'} \neq \emptyset$, which follows since $I_{W}$ and $I_{W'}$ are both nonempty coideals in $E$.  
\end{proof}

\begin{remark}
    By using duality, see Section \ref{sec: duality} below, we arrive at the same conclusion by assuming that for every $\lambda, \mu \in E$ there exists $\nu$ such that $\nu \leq \lambda, \mu$. 
\end{remark}

\begin{proposition}\label{prop: bounds}
Let $I \subseteq E$ be a coideal. Define $V(I)$ and $U(I)$ by:
\[
V_{k}(I) = \Bigl\langle[\lambda; \underline{w}] \in V_{k}(E) \mid \lambda \in I\Bigr\rangle, \qquad U_{k}(I) = \Bigl\langle[\lambda; \underline{w}] \in V_k(E) \mid \lambda \cup \underline{w} \in I\Bigr\rangle.
\]
Then:
\begin{enumerate}
\item[(1)] Both $V(I)$ and $U(I)$ are $\Bqt^{\ext}$-submodules of $V(E)$.
\item[(2)] If $W \subseteq V(E)$ is a $\Bqt^{\ext}$-submodule, then $V(I_{W}) \subseteq W \subseteq U(I_{W})$. 
\end{enumerate}
\end{proposition}
\begin{proof}
Statement (1) is easy. For (2), the fact that $V(I_{W}) \subseteq W$ follows from Lemma \ref{lem: vandermonde} and the definition of $d_{+}$. Indeed, if $\lambda\in I_W$ then $[\lambda]\in W_0$ and any good chain $[\lambda;\underline{w}]$ appears in $d_+^{k}[\lambda]$ with a nonzero coefficient.

On the other hand, the containment $W \subseteq U(I_{W})$ follows from $W$ being closed under the operator $d_{-}$ since
$d_{-}^k[\lambda;\underline{w}]=[\lambda\cup \underline{w}].$
\end{proof}

\begin{remark}
\label{rem: U as kernel }
Note that, if $I \subseteq E$ is an ideal, then the kernel of $V(E) \to V(I)$ is precisely $U(E \setminus I)$. 
\end{remark}

In fact, we can describe all the submodules of $V(E)$ in terms of coideals in a poset structure on the set $\gCh(E)$ of (maximal) good chains on $E$, that we now describe.

\begin{lemma}
Let $\Ch(E)$ denote the set of all maximal chains on the poset $E$. The following are the covering relations of a partial order on $\Ch(E)$:
\begin{equation}\label{eq: partial order increase right}
(\lambda_1 \to \lambda_2 \to \cdots \to \lambda_{k}) \prec (\lambda_1 \to \lambda_2 \to \cdots \to \lambda_k \to \lambda_{k+1})
\end{equation}
and
\begin{equation}\label{eq: partial order remove left}
(\lambda_1 \to \lambda_2 \to \cdots \to \lambda_{k}) \prec (\lambda_2 \to \cdots \to \lambda_k).
\end{equation}
\end{lemma}
\begin{proof}
We need to show that no cycles appear in the transitive closure of $\prec$. Note that \eqref{eq: partial order increase right} increases the endpoint of a chain, while \eqref{eq: partial order remove left} increases the starting point of a chain. This implies the result. 
\end{proof}

By the previous lemma, we have a partial order $\prec$ on $\Ch(E)$ that we can restrict to $\gCh(E)$, which is a basis of the representation $V(E)$. Note that if $[\lambda; \underline{w}]$ is a good chain and $s_i$ is an admissible transposition for $\underline{w}$, then $[\lambda; s_i(\underline{w})]$ is a good chain as well. So we have an action of admissible transpositions on the poset $\gCh(E)$.

\begin{theorem}\label{thm:submodules}
Let $W \subseteq V(E)$ be a nonzero submodule, and let $\mathcal{I}_{W} := \bigl\{[\lambda; \underline{w}] \mid [\lambda; \underline{w}] \in W\bigr\} \subseteq \gCh(E)$. Then, $\mathcal{I}_W$ is a coideal in $\gCh(E)$ that is stable under the action of admissible transpositions. 

Conversely, if $\mathcal{I} \subseteq \gCh(E)$ is a coideal that is stable under the action of admissible transpositions, then $\mathrm{span}\bigr\{[\lambda; \underline{w}] \mid [\lambda; \underline{w}] \in \mathcal{I}\bigr\}$ is a submodule of $V(E)$. 
\end{theorem}
\begin{proof}
Since $W$ is closed under $d_{+}$, it follows from Lemma \ref{lem: vandermonde} that $\mathcal{I}_{W}$ is closed under \eqref{eq: partial order increase right} and, since $W$ is closed under $d_{-}$ we also have that $\mathcal{I}_{W}$ is closed under \eqref{eq: partial order remove left}. Finally, the fact that $\mathcal{I}_{W}$ is stable under admissible transpositions follows from the fact that $W$ is closed under the action of the $T$-operators, see \eqref{eq: T Ram new}, together with Lemma \ref{lem: vandermonde}. The converse statement is clear. 
\end{proof}

\subsection{Application to the Polynomial Representation}

As an example of the above constructions, we can consider the poset $\mathcal{P}$ of partitions as in Section \ref{sec: poly}. 

Given an integer $N\ge 1$, let $I_N$ (resp. $J_N$) be the subset of $\mathcal{P}$ consisting of partitions with more than $N$ rows (resp. columns). Clearly, $I_N$ and $J_N$ are both coideals in $\mathcal{P}$, while $\mathcal{P}\setminus I_N$ and $\mathcal{P}\setminus J_N$ are ideals in $\mathcal{P}$. 

For $N=1$ the ideal $\mathcal{P}\setminus I_N$ is a linear poset of partitions with one row, and $V(\mathcal{P}\setminus I_N)$ recovers the construction in Section \ref{sec: linear posets}. More generally, we have the following.

\begin{lemma}
\label{lem: restricted partitions}
a) For all $N\ge 1$ the polynomial representation of $\Bqt^{\ext}$ admits quotients $V(\mathcal{P}\setminus I_N)$ and $V(\mathcal{P}\setminus J_N)$ with bases given by standard tableaux in horizontal strips with at most $N$ rows (resp. at most $N$ columns).

b) In particular, $V_k(\mathcal{P}\setminus I_N)$ is nontrivial for all $k\ge 0$, while $V_k(\mathcal{P}\setminus J_N)$ is nontrivial only if $0\le k\le N$.

c) The representation $V(\mathcal{P}\setminus J_N)$ is isomorphic (as a graded vector space) to 
$$
\bigoplus_{k=0}^{N}\C[x_1,\ldots,x_{N}]^{S_{N-k}}
$$
where $S_{N-k}$ acts on the first $N-k$ variables $x_1,\ldots,x_{N-k}$.
\end{lemma}

\begin{proof}
Part (a) follows from Proposition \ref{prop: subs and quotients}. For part (b), note that if $[\lambda; \underline{w}]$ is a good chain, then $w_k, \dots, w_1$ belong to different columns. Part (c) is similar to \cite[Section 7]{CGM}, we construct a bijection between the bases. Indeed, $V_k(\mathcal{P}\setminus J_N)$ has a basis labeled by standard tableaux of skew shape $\mu\setminus \lambda$ such that
\begin{itemize}
\item[(1)] $|\mu\setminus \lambda|=k$. 
\item[(2)] $\mu\setminus \lambda$ is a horizontal strip.
\item[(3)] $\mu$ has at most $N$ columns.
\end{itemize}
Given such a tableau $T$, we look at its columns which can be either unlabeled, or contain exactly one label $i$. Let $a_i, 1\le i\le k$ denote the height of the column labeled by $i$, and $\nu_1,\ldots,\nu_{N-k}$ denote the heights of unlabeled columns. Here $\nu_1\ge \ldots\ge \nu_{N-k}\ge 0$, so that $\nu$ is a partition. If $\mu$ has less than $N$ columns, we add unlabeled empty columns and set the corresponding $\nu_i=0$.

Conversely, given a partition $\nu=(\nu_1,\ldots,\nu_{N-k})$, and a sequence $(a_1,\ldots,a_k)$, we can sort the sequence $(\nu_1,\ldots,\nu_{N-k},a_1,\ldots,a_k)$ and consider a partition $\mu$ with the corresponding column length. If $a_i=a_j$ and $i<j$ we can ensure that $a_i$ is to the left of $a_j$. Then we get a standard tableau in a horizontal strip $\mu\setminus \lambda$ satisfying (1)-(3) above.

It remains to notice that pairs $(\nu;a_1,\ldots,a_k)$ are in bijection with polynomials 
$$
\Sym\left(x_1^{\nu_1}\cdots x_{N-k}^{\nu_{N-k}}\right)x_{N-k+1}^{a_1}\cdots x_{N}^{a_k}
$$
which form a basis in $\C[x_1,\ldots,x_{N}]^{S_{N-k}}$.
\end{proof}

It would be interesting to construct a representation of $\Bqt^{\ext}$ on $\bigoplus_{k=0}^{N}\C[x_1,\ldots,x_{N}]^{S_{N-k}}$
directly. We conjecture that this can be done by utilizing partially symmetrized DAHA or its analogues as in \cite{milo, Goodberry, IonWu}.

\subsection{Semisimplification of $V(E)$} Let us recall that if $M$ is a finite-dimensional representation of an algebra $A$, its \newword{semisimplification} is
\[
\ssn(M) := \gr\filt_{\rad}(M) \simeq \gr \filt_{\soc}(M),
\]
where $\filt_{\rad}$ is the (descending) radical filtration and $\filt_{\soc}$ is the (ascending) socle filtration. From a geometric point of view, the $A$-modules of dimension $\dim(M)$ form an algebraic variety that admits a $\operatorname{GL}_{\dim(M)}$-action by base-change, and $\ssn(M)$ is the representative of the unique closed orbit contained in the closure of the orbit of $M$, see e.g. \cite[Section 2.3]{kirillov}. The following proposition can be proved by induction on the number of elements of $E$.

\begin{proposition}
    Let $E$ be a finite excellent weighted poset. Then, the semisimplification of $V := V(E)$ is described as follows:
    \begin{itemize}
    \item $\ssn(V)_{k} = V_{k}(E)$ as $\AH_{k}$-modules and also as modules over the algebra generated by $\Delta$-operators.
    \item $d_{+}$ and $d_{-}$ are identically zero on $\ssn(V)$. 
    \end{itemize}
\end{proposition}
\begin{proof}
Let $\mu \in E$ be a maximal element so that $I := \{\mu\}$ is a coideal and $U(I)$ is a submodule of $V(E)$ with quotient $V(E \setminus \{\mu\})$. By induction on the number of elements of $E$, it suffices to show that the semisimplification of $U(I)$ is given by simply setting $d_{+}, d_{-}$ to be identically zero. Note that by Theorem \ref{thm: Ram new} $U(I)_k$ is already semisimple as $\AH_k$ representation.

    To this end, note that by the maximality of $\mu$, $d_{+}$ is already identically zero on $U(I)$. This implies that we have a filtration by $\Bqt^{\ext}$-submodules
    \[
    U_0(I) \subseteq U_0(I) \oplus U_1(I) \subseteq \dots \subseteq U_0(I) \oplus \dots \oplus U_{m}(I) = U(I)
    \]
    where $m$ is the length of the longest chain in $E$ terminating at $\mu$. Taking the associated graded with respect to this filtration, we obtain the desired result. 
\end{proof}

Let us remark that a similar statement holds if $E$ is either upper or lower finite. More precisely, in these cases we can find a (infinite) filtration whose associated graded simply sets both $d_+$ and $d_-$ to zero. See Section \ref{sec: limits} below. 

\subsection{Direct and Inverse Limits}\label{sec: limits} One can use the representations $V(I)$ when $I$ is a (co)ideal of the excellent poset $E$ in order to give filtrations of $V(E)$ with nice properties. For an element $\lambda \in E$, define
\[
I_{\leq \lambda} := \{\mu \in E \mid \mu \leq \lambda\}, \qquad I_{\geq \lambda} := \{\mu \in E \mid \mu \geq \lambda\}.
\]
It is clear that $I_{\leq \lambda}$ is an ideal of $E$, while $I_{\geq \lambda}$ is a coideal. By Proposition \ref{prop: subs and quotients} we have maps
\[
\pi_{\lambda}: V(E) \twoheadrightarrow V(I_{\leq \lambda}), \qquad \iota_{\lambda}: V(I_{\geq\lambda}) \hookrightarrow V(E). 
\]

Moreover, if $\lambda \leq \mu$ then we have  that $I_{\leq \lambda}$ is an ideal in $I_{\leq \mu}$ while $I_{\geq \mu}$ is a coideal in $I_{\geq \lambda}$, so we have maps
\[
\pi_{\lambda, \mu}: V(I_{\leq \mu}) \twoheadrightarrow V(I_{\leq \lambda}), \qquad \iota_{\lambda, \mu}: V(I_{\geq \mu}) \hookrightarrow V(I_{\geq \lambda}).
\]
These maps are compatible in the sense that $\pi_{\lambda} = \pi_{\lambda, \mu}\circ\pi_{\mu}$ and $\iota_{\mu} = \iota_{\lambda}\circ\iota_{\lambda, \mu}$. Moreover, if $\lambda \leq \mu \leq \nu$, then we have $\pi_{\lambda, \nu} = \pi_{\lambda, \mu}\circ\pi_{\mu, \nu}$ and $\iota_{\lambda, \nu} = \iota_{\lambda, \mu}\circ \iota_{\mu, \nu}$. 

\begin{proposition}\label{prop: limits}
We have:
\[
V(E) \simeq\varprojlim_{\lambda \in E}V(I_{\leq \lambda}) \qquad \text{and} \qquad V(E) \simeq \varinjlim_{\lambda \in E}V(I_{\geq \lambda}). 
\]
\end{proposition}
\begin{proof}
We only show the direct limit assertion. The other assertion follows from this and the duality in Section \ref{sec: duality} (which is independent of the intervening material). Since all the maps $\iota_{\lambda, \mu}$ and $\iota_{\lambda}$ are inclusions, the direct limit assertion is equivalent to the statement that
\[
V(E) = \bigcup_{\lambda \in E}V(I_{\geq \lambda}).
\]
But this is clear since $[\lambda; \underline{w}] \in V(I_{\geq \lambda})$. 
\end{proof}

\begin{remark}
The statement of Proposition \ref{prop: limits} can be strengthened, with essentially the same proof, as follows. Let $P \subseteq E$ be a subset with the property that for every $\lambda \in E$ there exists $\mu \in P$ with $\mu \leq \lambda$. Then,
\[
V(E) \simeq \varprojlim_{\lambda \in P} V(I_{\leq \lambda}), \qquad V(E) \simeq \varinjlim_{\lambda \in P}V(I_{\geq \lambda}).
\]

\end{remark}

Recall that a poset $E$ is said to be \emph{lower finite} (resp. \emph{upper finite}) if, for every $\lambda \in E$, the ideal $I_{\leq \lambda} := \{\mu \in E \mid \mu \leq \lambda\}$ (resp. the coideal $I_{\geq \lambda} := \{\mu \in E \mid \mu \geq \lambda\}$) is finite. 

\begin{corollary}
The representation $V(E)$ is the direct (resp. inverse) limit of finite-dimensional representations if and only if the poset $E$ is upper (resp. lower) finite. 
\end{corollary}
\begin{proof}
Again, we only show the direct limit statement. If $E$ is upper finite then $V(I_{\geq \lambda})$ is finite-dimensional for every $\lambda \in E$, so the result follows immediately from Proposition \ref{prop: limits}.
Now assume that $E$ is not upper finite, so there exists $\lambda \in E$ such that $I_{\geq \lambda}$ is infinite. This implies that every coideal of $E$ containing $\lambda$ must be infinite. Hence by Lemma \ref{lem: producing a coideal}, $[\lambda] \in V_0(E)$ is not contained in any finite-dimensional subrepresentation of $V(E)$ and thus $V(E)$ cannot be a direct limit of finite-dimensional representations. 
\end{proof}

\subsection{Gradings}
Notice that the representation $V(E)$ is automatically bigraded
\[
V(E) = \bigoplus_{(r,s) \in \Z^{2}} V_{r,s}(E),
\]
with the bi-degree of a chain $(\lambda \to \lambda \cup w_k \to \cdots \to \lambda \cup w_k \cup \dots \cup w_1 = \mu)$ given by $(|\mu|,|\lambda|)$. In particular, we have
\[
V_{k}(E) = \bigoplus_{r-s = k}V_{r,s}(E).
\]
It is straightforward to see that $d_{+}$ has degree $(1,0)$ and $d_{-}$ has degree $(0,1)$. The generators $z_i, T_i$ and $\Delta_{p}$ all have degree $(0,0)$, whereas $\varphi = \frac{1}{q-1}[d_{+}, d_{-}]$ has degree $(1,1)$. 

On the other hand, the algebra $\Bqt$ is triply graded, see \cite[Section 3.3]{CGM}, with $d_{+}$ in degree $(1,0,0)$, $d_{-}$ in degree $(0,1,0)$, $T_{i}$ in degree $(0,0,0)$ and $z_i$ in degree $(0,1,1)$. The two gradings are related by
\[
(a,b,c) \mapsto (a, b-c).
\]

\subsection{Homomorphisms Between $\Bqt$ Representations}\label{subsec:Homomorphisms} We use the results of the previous sections to study the Hom space $\Hom_{\Bqt^{\ext}}(V(E), V(E'))$ where $E, E'$ are excellent weighted posets.

Let us assume that we have a $\Bqt^{\ext}$-homomorphism $f: V(E) \to V(E')$ such that $$f = \bigoplus_{k \geq 0}\left(f_{k}: V_{k}(E) \to V_{k}(E')\right).$$ In particular, we have $f_0: V_{0}(E) \to V_{0}(E')$. Since $f$ is a $\Bqt^{\ext}$-homomorphism, it must send $\Delta_{p_{m}}$-eigenvectors to $\Delta_{p_{m}}$-eigenvectors. Thus, for each $\lambda \in E$, we have $f([\lambda]) = \alpha_{\lambda}[\lambda']$ for some $\lambda' \in E'$ and $0 \neq \alpha_{\lambda} \in \C(q,t)$, or $f([\lambda]) = 0$. Moreover, if $f([\lambda]) = [\lambda']$ then $p_m(\lambda) = p_m(\lambda')$ for every $m > 0$.

For any poset $E$, let us define a new poset $E_{\mathbf{0}} := E\sqcup\{\mathbf{0}\}$, with $\lambda < \mathbf{0}$ for all $\lambda \in E$.

\begin{theorem}\label{thm: hom}
Let $f: V(E) \to V(E')$ be a homomorphism of $\Bqt^{\ext}$-modules. Define a map
\[
F: E_{\mathbf{0}} \to E'_{\mathbf{0}}, \qquad F(\lambda) = \begin{cases} \lambda', & \lambda \in E \; \text{and} \; f([\lambda]) = \alpha_{\lambda}[\lambda'] \neq 0, \\ \mathbf{0}, & \lambda \in E \; \text{and} \; f([\lambda]) = 0, \\
\mathbf{0}, & \lambda = \mathbf{0}.
\end{cases}
\]
Then,
\begin{enumerate}
    \item $F^{-1}(\mathbf{0}) \subseteq E_{\mathbf{0}}$ is a coideal in $E_{\mathbf{0}}$.
    \item $F(E_{\mathbf{0}}) \subseteq E'_{\mathbf{0}}$ is a coideal in $E'_{\mathbf{0}}$.
    \item If $\lambda \in E$, $\lambda \not\in F^{-1}(\mathbf{0})$, then $p_{m}(\lambda) = p_{m}(F(\lambda))$ for every $m > 0$.
    \item If $\lambda, \mu \in E$ are such that $\mu$ covers $\lambda$, then either $F(\mu) = \mathbf{0}$ or $F(\mu)$ covers $F(\lambda)$. 
    \item If $\lambda, \mu \in E$ are such that $F(\mu)$ covers $F(\lambda)$ and $F(\mu) \neq \mathbf{0}$, then $\mu$ covers $\lambda$.
\end{enumerate}
Conversely, if $F: E_{\mathbf{0}} \to E'_{\mathbf{0}}$ is a map such that $F(\mathbf{0}) = \mathbf{0}$ and satisfying (1), (2), (3), (4) and (5), then it induces a $\Bqt^{\ext}$-homomorphism $f: V(E) \to V(E')$. 
\end{theorem}
\begin{proof}
For (1), note that the kernel $\ker(f_{0})$ is spanned by $\{[\lambda] \mid \lambda \in E\cap F^{-1}(\mathbf{0})\}$. Since $\mathbf{0} \in E_{\mathbf{0}}$ is declared to be the maximum element, the result now follows from Lemma \ref{lem: producing a coideal}. The proof of (2) is similar, now observing that the image of $f$ is a $\Bqt^{\ext}$-submodule of $V(E')$ and that the image of $f_0$ is spanned by $\{f([\lambda]) \mid \lambda \in E\}$. Statement (3) follows since $f$ intertwines the actions of the $\Delta$-operators.

Let us now show Statement (4). Assume $\mu$ covers $\lambda$. If $F(\mu) \neq \mathbf{0}$ then by (2) we obtain $F(\lambda) \neq \mathbf{0}$. Now, a scalar multiple of $f([\mu])$ is a summand in $f(d_{-}d_{+}[\lambda])$. Since $f$ is a homomorphism, $f(d_{-}d_{+}[\lambda]) = d_{-}d_{+}(f([\lambda]))$. However, all the summands appearing in $d_{-}d_{+}(f([\lambda]))$ are scalar multiples of $[\nu]$ where $\nu$ covers $F(\lambda)$, hence $F(\mu)$ covers $F(\lambda)$.

To show Statement (5), assume $\mathbf{0} \neq F(\mu)$ covers $F(\lambda)$, in particular, $F(\lambda) \neq \mathbf{0}$. Then, $F(\mu)$ appears as a summand in $d_{-}d_{+}[F(\lambda)]$, which means that an element of weight equal to that of $F(\mu)$ must appear as a summand in $d_{-}d_{+}(\lambda)$. That is, there is an element of weight equal to that of $\mu$ that covers $\lambda$. By the simple spectrum condition, we must have that $\mu$ covers $\lambda$. 

Now, assume that $F: E_{\mathbf{0}} \to E'_{\mathbf{0}}$ is a map satisfying (1), (2), (3), (4) and (5). By Lemma \ref{lem: compatible edge} below, we can choose compatible edge functions on $E$ and $E'$, so we assume we have chosen these. If $F \equiv \mathbf{0}$, then the map $f: V(E) \to V(E')$ is simply the zero map. So let us assume this is not the case. Let
\[
\lambda \to \lambda \cup w_k \to \cdots \to \lambda \cup w_k \cdots \cup w_1 = \mu
\]
be a chain in $E$. Assume $F(\mu) \neq \mathbf{0}$. Then, by (1) and (4):
\[
F(\lambda) \to F(\lambda \cup w_k) \to \cdots \to F(\lambda \cup w_k \cdots \cup w_1) = F(\mu)
\]
is a chain in $E'$. Moreover, by (3) this chain is
\[
F(\lambda) \to F(\lambda) \cup w_k \to \cdots \to F(\lambda) \cup w_k \cdots \cup w_1 = F(\mu)
\]
thus, we define a map $f_{k}: V_{k}(E) \to V_{k}(E')$ by 
\begin{equation}\label{eqn: def psi}
f_{k}\Bigl([\lambda; \underline{w}]\Bigr) = \begin{cases} [F(\lambda); \underline{w}], & F(\lambda \cup \underline{w}) \neq \mathbf{0} \\
0, & \text{else}\end{cases}.
\end{equation}
We claim that $f = \bigoplus f_{k}$ is a $\Bqt^{\ext}$-homomorphism. By (3), $f$ intertwines the action of the $\Delta$ operators, and it is clear from the formula \eqref{eqn: def psi} that $f$ also intertwines the action of the affine Hecke algebras and the operator $d_{-}$. It only remains to show that it intertwines the action of $d_{+}$.

For this, assume that $f([\lambda; \underline{w}]) = 0$. Then $F(\lambda \cup \underline{w}) = \mathbf{0}$. By (1), $F(\mu) = \mathbf{0}$ for every $\mu > \lambda \cup \underline{w}$. It follows that $f(d_{+}[\lambda; \underline{w}]) = 0$, as needed. If, on the other hand, $F(\lambda \cup \underline{w}) \neq \mathbf{0}$, then by (2) every element in $E'$ that covers $F(\lambda \cup \underline{w})$ also belongs to the image of $F$. Let $\nu \in E$ be such that $F(\nu)$ covers $F(\lambda \cup \underline{w})$. By (5), $\nu$ covers $\lambda \cup \underline{w}$. From here and the explicit formulas for $d_{+}$, together with the fact that we have chosen compatible edge functions, the result follows. 
\end{proof}

\begin{lemma}\label{lem: compatible edge}
Let $E, E'$ be excellent posets, and let $F: E_{\mathbf{0}} \to E'_{\mathbf{0}}$ be a function satisfying (1)--(5) of Theorem \ref{thm: hom}. Let $c$ be an edge function on $E$ satisfying \eqref{eq: monodromy}, and define a partial edge function $c'$ on $E'$ by
\begin{equation}\label{eq: partial c}
c'\bigl(F(\mu); x \bigr) = c(\mu; x) \; \text{if} \; x \; \text{is addable for} \; \mu \; \text{and} \; F(\mu\cup x) \neq \mathbf{0}. 
\end{equation}
Then, $c'$ can be extended to an edge function on $E'$ satisfying \eqref{eq: monodromy}. 
\end{lemma}
\begin{proof}
As in the proof of Theorem \ref{thm: existence and uniqueness} let $\Gamma_{E'}$ be the associated graph to the poset $E'$, where edges are given by covering relations. Let $\Gamma \subseteq \Gamma_{E'}$ be the subgraph whose edges are those that connect elements in the image of $F$. Pick a spanning forest for $\Gamma$, and extend to a spanning forest for $\Gamma_{E'}$. Then choose any edge function on this spanning forest that coincides with $c'$ on those edges belonging to $\Gamma$. By (the proof of) Theorem \ref{thm: existence and uniqueness}, this edge function extends uniquely to an edge function satisfying \eqref{eq: monodromy}, and we need to verify that this extension coincides with $c'$ on $\Gamma$. But this follows because $c$ satisfied \eqref{eq: monodromy}. 
\end{proof}

In view of Theorem \ref{thm: hom} it is worth spelling out an explicit description of both the kernel and the image of a map $f: V(E) \to V(E')$. To simplify the argument and statement of the result, in the next proposition we assume that the edge functions on $E$ and $E'$ are chosen compatibly following Lemma \ref{lem: compatible edge}.

\begin{proposition}\label{prop:hom}
Let $E, E'$ be excellent weighted posets, $f: V(E) \to V(E')$ a $\Bqt^{\ext}$-morphism and $F: E_{\mathbf{0}} \to E'_{\mathbf{0}}$ its associated function as in Theorem \ref{thm: hom}. Assume that the edge functions on $E$ and $E'$ are chosen compatibly following Lemma \ref{lem: compatible edge}. Then, for every good chain $[\lambda; \underline{w}] \in V(E)$:
\begin{equation}\label{eq: hom}
f\Bigr([\lambda; \underline{w}]\Bigr) = \begin{cases} \alpha_{\lambda}[F(\lambda); \underline{w}] & ; \text{ if }F(\lambda \cup \underline{w}) \neq \mathbf{0} \\ 0 &; \text{ else}. \end{cases}
\end{equation}
where $\alpha_{\lambda}$ is a nonzero scalar such that $f([\lambda]) = \alpha_{\lambda}[F(\lambda)]$ (note that $F(\lambda \cup \underline{w}) \neq \mathbf{0}$ implies that $F(\lambda) \neq \mathbf{0}$). In particular we obtain that:
\begin{enumerate}
    \item The kernel of $f$ is spanned by good chains $[\lambda; \underline{w}] \in V(E)$ such that $\lambda \cup \underline{w} \in F^{-1}(\mathbf{0})\cap E$. In particular, the kernel of $f$ is $U(F^{-1}(\mathbf{0})\cap E)$.
    \item The image of $f$ is spanned by all the good chains $[\lambda'; \underline{w}] \in V(E')$ where $\lambda' \in F(E)\cap E'$. In particular, the image of $f$ is isomorphic to $V(F(E)\cap E')$.   
\end{enumerate}
 \end{proposition}
\begin{remark}
    Note that $F(E)\cap E'$ is a coideal in $E'$ by Theorem \ref{thm: hom} (2), and thus inherits the structure of an excellent weighted poset from $E'$. 
\end{remark}
\begin{proof}
 We prove \eqref{eq: hom} by induction on $k$, the case $k = 0$ follows from the definition of $F$. Now,
 \begin{equation*}
 q^{k-1}c(\lambda \cup w_k \cup \dots \cup w_2; w_1)\prod_{i = 1}^{k-1}\frac{w_k - tw_i}{w_k - qtw_i}[\lambda; \underline{w}] = \mathsf{P}d_{+}[\lambda; w_k, \dots, w_2],
 \end{equation*}
 where $\mathsf{P}$ is the projection to a specific simultaneous weight space for the $\Delta$-operators and $z_1, \dots, z_k$. Since $f$ is a $\Bqt^{\ext}$-homomorphism it commutes both with $d_{+}$ and $\mathsf{P}$, so that
 \begin{equation}\label{eq: d+ and P and f}
 q^{k-1}c(\lambda \cup w_k \cup \dots \cup w_2; w_1)\prod_{i = 1}^{k-1}\frac{w_k - tw_i}{w_k - qtw_i}f([\lambda; \underline{w}]) = \mathsf{P}d_{+}f\bigl([\lambda; w_k, \dots, w_2]\bigr).
 \end{equation}
If $F(\lambda \cup  \underline{w}) \neq \mathbf{0}$, then $F(\lambda \cup w_k \cup \cdots \cup w_2) \neq \mathbf{0}$ as well, and we have a two step chain
\[
F(\lambda \cup w_k \cup \dots \cup w_2) \to F(\lambda \cup  \underline{w}) = F(\lambda \cup w_k \cup \dots \cup w_2) \cup w_1.
\]
By induction hypothesis, $f([\lambda; w_k, \dots, w_2]) = \alpha_\lambda[F(\lambda); w_k, \dots, w_2])$. Now we have:
\begin{align}
\mathsf{P}d_{+}f([\lambda; w_k, \dots, w_2]) &= 
\alpha_{\lambda}\mathsf{P}d_{+}[F(\lambda); w_k, \dots, w_2] \nonumber \\
&= \alpha_{\lambda}q^{k-1}c(F(\lambda) \cup w_k \cup \cdots \cup w_2; w_1)\prod_{i = 1}^{k-1}\frac{w_k - tw_i}{w_k - qtw_i}[F(\lambda); \underline{w}]. \label{eq: d+ and P in E'}
\end{align}
By the assumption of compatibility of the edge functions, 
\[c(F(\lambda) \cup w_k \cup \cdots \cup w_2; w_1) = c(\lambda \cup w_k \cup \cdots \cup w_2; w_1).\]
 Comparing \eqref{eq: d+ and P and f} with \eqref{eq: d+ and P in E'} we arrive to $f([\lambda; \underline{w}]) = \alpha_\lambda[F(\lambda); \underline{w}]$, as needed.\\ 
Now assume that $F(\lambda \cup  \underline{w}) = \mathbf{0}$. Then 
\begin{equation}\label{eq: d- for a contradiction}
d_{-}^{k}f([\lambda; \underline{w}]) = f(d_{-}^{k}[\lambda; \underline{w}]) = f([\lambda \cup \underline{w}]) = 0.
\end{equation}
Now, if $f([\lambda; \underline{w}]) \neq 0$ then it is a weight vector for the action of $z$'s and $\Delta$, so it has to be of the form $\beta[\mu; \underline{v}]$ for some $\mu \in E'$ and $\beta \in \C(q,t)^{*}$. But then $d_{-}^{k}f([\lambda; \underline{w}]) = \beta d_{-}^{k}[\mu; \underline{v}] = \mu\cup\underline{v} \neq 0$, a contradiction with \eqref{eq: d- for a contradiction}. Thus, $f([\lambda; \underline{w}]) = 0$. 

From \eqref{eq: hom} (1) is clear. It is also clear that the image of $f$ is contained in the span of all good chains $[\lambda';  \underline{w}] \in V(E')$ such that $\lambda' \in F(E) \cap E'$. To finish, note that if $[\lambda';  \underline{w}]$ is such a chain, then by Theorem \ref{thm: hom} (2), (3) and (4) we can find a chain $[\lambda;  \underline{w}] \in V(E)$ such that $F(\lambda) = \lambda'$, and the result follows. 
\end{proof}

\begin{corollary}
Let $E = \{\bullet\}$ be a one-element set with weighting $a_k = p_{k}(\bullet)$. Then, for any other poset $E'$, $\Hom_{\Bqt^{\ext}}(V(E), V(E'))$ is at most one-dimensional, and it is nonzero if and only if there exists a maximal element $\mathsf{M} \in E'$ such that $p_{k}(\mathsf{M}) = a_{k}$ for all $k$.

Similarly, $\Hom_{\Bqt^{\ext}}(V(E'), V(E))$ is at most one-dimensional, and it is nonzero if and only if there exists a minimal element $\mathsf{m} \in E'$ such that $p_k(\mathsf{m}) = a_k$ for all $k$. 
\end{corollary}
\begin{proof}
Assume $\Hom_{\Bqt^{\ext}}(V(E), V(E'))$ is nonzero. Then, there must exist a function $F: \{\bullet, \mathbf{0}\} \to E'_{\mathbf{0}}$ satisfying  (1)--(5) of Theorem \ref{thm: hom}. Setting $\mathsf{M} := F(\bullet)$, we have that $\mathsf{M} \neq \mathbf{0}$, and since $\{\mathsf{M}, \mathbf{0}\} \subseteq E'_{\mathbf{0}}$ is a coideal, $\mathsf{M} \in E'$ must be a maximal element. That $p_{k}(\mathsf{M}) = a_k$ is simply Condition (3) of Theorem \ref{thm: hom}. 

The proof of the second statement is similar. Alternatively, it follows from the first statement using the results of Section \ref{sec: duality}. 
\end{proof}

\begin{remark}
While this describes the homomorphisms in $\Hom_{\Bqt^{\ext}}(V(E), V(E'))$, it is still an interesting question to describe the homomorphisms in $\Hom_{\Bqt}(V(E), V(E'))$. For example, if we have two singletons $E = \{\bullet\}$, $E' = \{\mathsf{M}\}$, then $\Hom_{\Bqt^{\ext}}(V(E), V(E'))$ is nonzero if and only if $p_m(\bullet) = p_m(\mathsf{M})$ for all $m > 0$, while the spaces $V(E)$ and $V(E')$ are always isomorphic as $\Bqt$-representations. 
\end{remark}

\section{Parabolic Gieseker Moduli Spaces}
\label{sec: Gieseker}

In this section, we construct an action of the $\Bqt^{\ext}$ algebra on the $K$-theory of parabolic Gieseker moduli spaces. As we will see, we obtain a calibrated representation of the form $V(E)$, where $E$ is the poset of $r$-multipartitions. The case $r = 1$ was treated in work of the second author with Carlsson and Mellit \cite{CGM} and we follow that work closely. 

\subsection{The Gieseker Moduli Space}\label{subsec: Gieseker}

We will be working with the \newword{Gieseker moduli space} $\M(r,n)$ of framed sheaves on $\PP^2$, that is, torsion free rank $r$ sheaves  $\F$ on $\PP^2$ that are locally free in a neighborhood of $\ell_{\infty}$, with a choice of framing $\Phi:\F|_{\ell_{\infty}}\simeq \CO_{\ell_{\infty}}^{\oplus r}$. Here $\ell_{\infty}$ is a line at infinity on $\PP^2$. We also require that $c_2(\F)=n$. The space $\M(r,n)$ is known to be smooth of dimension $2rn.$

Below we will often use the description of $\M(r,n)$ as a quiver variety, see e.g. \cite[Chapter 2]{nakajimabook}. Consider two vector spaces $V$ and $W$ of dimensions $n$ and $r$ respectively, with four operators $X,Y:V\to V$, $i:W\to V$ and $j:V\to W$. The group $GL(V)$ acts on the space of quadruples $(X,Y,i,j)$, and the corresponding moment map is 
$$
\mu(X,Y,i,j)=[X,Y]-ij.
$$
Now we can write $\M(r,n)$ as a symplectic quotient
$$
\M(r,n)\simeq \mu^{-1}(0)^{ss}/GL(V)
$$
where the stability condition means that the only $X,Y$-invariant subset of $V$ containing $\operatorname{im}(i)$ is $V$ itself.

Let $T$ be the maximal torus of $GL(r)$, and $\widetilde{T}=\C^*\times \C^*\times T$. The torus $T$ acts on $\M(r,n)$ by changing the framing, and $\C^*\times \C^*$ acts by scaling the coordinates on $\PP^2$: $[z_0:z_1:z_2]\mapsto [z_0,qz_1,tz_2]$.
The fixed points of $T$ on $\M(r,n)$ are direct sums of ideal sheaves $\F=\I_1\oplus \I_2\oplus \cdots \oplus \I_r$, so that
$$
\M(r,n)^T=\bigsqcup_{n_1+\ldots+n_r=n} \Hilb^{n_1}(\C^2)\times \cdots \times \Hilb^{n_r}(\C^2)
$$
while the fixed points of $\widetilde{T}$ are isolated and correspond to all $\I_j$ being monomial ideals. Thus, $\M(r,n)^{\widetilde{T}}$ is parametrized by $r$-tuples of Young diagrams $(\lambda_1,\ldots,\lambda_r)$ with total size $n = |\lambda_1| + \cdots + |\lambda_r|$.

In the quiver description, the $\widetilde{T}$-action on $\M(r,n)$ is given by:
\[
(q,t,a_1,\ldots,a_r).(X,Y,i,j)=(qX,tY,iA^{-1},qtAj),\ A=\diag(a_1,\ldots,a_r).
\]

Next, we describe the characters of the tangent spaces at $\widetilde{T}$-fixed points of $\M(r,n)$. We will do it in two ways. As before, if $\sq$ is a box in the $i$-th row and $j$-th column of the partition $\lambda_{\alpha}$ ($\alpha = 1, \dots, r$), we identify the box $\sq$ with its $(q,t)$-content given by the monomial $q^{i-1}t^{j-1}$. 

\begin{theorem}[\cite{NY}, Theorem 3.2]
\label{thm: arm leg}
The $\widetilde{T}$-character of the tangent space at the fixed point $\lambda_{\bullet}=(\lambda_1,\ldots,\lambda_r)$  is given by
\begin{equation}
T_{ \lambda_{\bullet}}\M(r,n)=\sum_{\alpha,\beta=1}^{r}a_{\beta}a_{\alpha}^{-1}\left(
\sum_{\sq\in \lambda_{\alpha}}q^{-a_{\lambda_{\beta}}(\sq)}t^{l_{\lambda_{\alpha}}(\sq)+1}+\sum_{\sq\in \lambda_{\beta}}q^{a_{\lambda_{\alpha}}(\sq)+1}t^{-l_{\lambda_{\beta}}(\sq)}
\right).
\end{equation}
\end{theorem}
Here, the \newword{arm} $a_{\mu}(\sq)$  and \newword{leg} $l_{\mu}(\sq)$ of a box $\sq=(i,j)$ are defined as:
\begin{equation}
a_{\mu}(\sq)=\lambda_j-i,\qquad l_{\mu}(\sq)=\lambda'_i-j, 
\end{equation}
where we allow these values to be negative whenever $\sq$ is outside the Young diagram for $\mu$. \footnote{In other words, $a_{\mu}(\sq)$ (resp. $l_{\mu}(\sq)$) is the \emph{signed} horizontal (resp. vertical) distance from the box $\sq$ to the boundary of $\mu$.}

Notice that any box $\sq\in \lambda_{\alpha}$ contributes to $2r$ different terms in the sum, so we get
$
2r\sum_{\alpha}|\lambda_{\alpha}|=2rn
$
terms in total, which agrees with the fact that $\dim T_{\lambda_{\bullet}}(\M(r,n))=2rn$.

The second description of the tangent space is given in the proof of \cite[Theorem 2.11]{NY2} and uses the quiver formulation. Let us fix $(X, Y, i, j) \in \mu^{-1}(0)$ and consider the complex
\begin{equation}
\label{eq: complex Gieseker}
\Hom(V,V)\xrightarrow{\sigma}\Hom(V,V)^{\oplus 2}\oplus \Hom(V,W)\oplus \Hom(W,V)\xrightarrow{\tau} \Hom(V,V).
\end{equation}
Here $\sigma(M)=([X,M],[Y,M],Mi,-jM)$ is the differential of the action of $\operatorname{GL}(V)$ and 
$\tau(B,C,I,J)=[X,C]+[B,Y]+iJ+Ij$ is the differential of the moment map. Note that 
$$
\tau\circ\sigma=[X,[Y,M]]+[[X,M],Y]-ijM+Mij=[[X,Y]-ij,M]=0.
$$
One can check that the stability conditions imply that $\sigma$ is injective and $\tau$ is surjective, and $T\M(r,n)$ is isomorphic to the middle cohomology of \eqref{eq: complex Gieseker} which can be computed as the Euler characteristic, which in turn decomposes as:
$$
\sum_{\alpha,\beta=1}^{r}a_{\beta}a_{\alpha}^{-1}\chi_{q,t}\left[
\Hom(V_{\alpha},V_{\beta})\to(q+t)\Hom(V_{\alpha},V_{\beta})\oplus qtV_{\alpha}^{*}\oplus V_{\beta}\to qt\Hom(V_{\alpha},V_{\beta})
\right].
$$ 
We are using that $W_{\beta}$ is one-dimensional, so that
$\Hom(V_{\alpha},W_{\beta})=V_{\alpha}^*$ and 
$\Hom(W_{\alpha},V_{\beta})=V_{\beta}$.
Now \eqref{eq: complex Gieseker} has Euler characteristic 
\begin{equation}
\label{eq: character alpha beta}
\sum_{\alpha,\beta=1}^{r}a_{\beta}a_{\alpha}^{-1}
\sum_{\sq\in \lambda_{\alpha},\sq'\in \lambda_{\beta}}\left[qt\sq+(\sq')^{-1}-(1-q)(1-t)\sq(\sq')^{-1}\right]
\end{equation}

Now Theorem \ref{thm: arm leg} follows from \eqref{eq: character alpha beta} by involved combinatorial manipulations, see \cite[Theorem 2.11]{NY2} for more details.

Finally, defining $B_{\mu}=\sum_{\sq\in \mu}\sq$ and
$B^*_{\mu}=\sum_{\sq\in \mu}\sq^{-1}$ as before, equation \eqref{eq: character alpha beta} can be compactly written as
$$
\sum_{\alpha,\beta=1}^{r}a_{\beta}a_{\alpha}^{-1}\left[
qtB_{\lambda_{\alpha}}+B^{*}_{\lambda_{\beta}}-(1-q)(1-t)B_{\lambda_{\alpha}}B^{*}_{\lambda_{\beta}}
\right].
$$

\subsection{The Parabolic Gieseker Space}

Define the \newword{{ parabolic} Gieseker moduli space} 
$\M^{par}(r,n;n+k)$ as the space of flags
$$
\F_n\supset \F_{n+1}\cdots \supset \F_{n+k}
$$
such that all $\F_i$ are framed rank $r$ sheaves, $c_2(\F_i)=i$, the framings at infinity agree with the filtration and $y\F_n\subset \F_{n+k}$. For $r=1$ this agrees with the parabolic flag Hilbert scheme defined in \cite{CGM}.

\begin{theorem}
\label{thm: smooth}
The space $\M^{par}(r,n;n+k)$ is smooth.
\end{theorem}

\begin{proof}
We follow the proof of \cite[Theorem 4.1.6]{CGM}. Consider a map $\phi:\C^2\to \C^2$, $\phi:(x,y)\to (x,y^{k+1})$ (in homogeneous coordinates on $\PP^2$ we have $[x:y:z]\mapsto [xz^k:y^{k+1}:z^{k+1}]$). Given a flag of sheaves
$\F_n\supset \F_{n+1}\cdots \supset \F_{n+k}$, define 
$$
\F:=\phi^*\F_{n+k}\oplus y\phi^*\F_{n+k-1}\oplus \cdots \oplus y^{k}\F_n.
$$
We claim that $\F$ is also a rank $r$ sheaf on $\PP^2$. On $\C^2$ we need to check that it is simply a module over $\C[x,y]$, it is indeed invariant under multiplication by $x$ and 
$$
y\cdot y^{k-i}\phi^*\F_{n+i}\subset y^{k-i+1}\phi^*\F_{n+i-1},\ i>0
$$
and
$$
y\cdot y^k\phi^*\F_n=\phi^*(y\F_n)\subset \phi^*(\F_{n+k}),
$$
so it is invariant under $y$ as well. Note that 
\[c_2(\F)=\sum c_2(\F_i)=n+(n+1)+\ldots+(n+k)=(k+1)n+\binom{k+1}{2}.\]

Next, we claim that the resulting sheaf $\F$ is invariant under the action of cyclic group $\Gamma=\Z_{k+1}$ which acts on $\C^2$ by $(x,y)\mapsto (x,\zeta y), \zeta^{k+1}=1$. Furthermore, any $\Gamma$-invariant sheaf $\F$ has this form, this follows from the $\Gamma$-eigenspace decomposition for $\F$.  

To sum up, we identified $\M^{par}(r,n;n+k)$ with the connected component of the $\Gamma$-fixed point locus in $\M(r,c_2(\F))$. Since $\M(r,c_2(\F))$ is smooth, the fixed point locus for a cyclic group action on it is smooth as well.
\end{proof}

The torus $\widetilde{T}$ acts on $\M^{par}(r,n;n+k)$, and we can identify the fixed points as follows. Each sheaf $\F_i$ corresponds to an $r$-tuple of Young diagrams $\lambda_{\bullet}^{(i)}=\left(\lambda_{\alpha}^{(i)}\right)$ and the above conditions imply that 
$$
\lambda_{\alpha}^{(n)}\subset \cdots \lambda_{\alpha}^{(n+k)}
$$
and $\lambda_{\alpha}^{(n+k)}\setminus \lambda_{\alpha}^{(n)}$ is a \textit{horizontal strip}, that is, it does not contain a pair of boxes in the same column. Furthermore, $\sum_{\alpha=1}^{r}|\lambda_{\alpha}^{(i)}|=i$.
We define $w^{(i)}_{\alpha}=B_{\lambda^{(n+i+1)}_{\alpha}}-B_{\lambda^{(n+i)}_{\alpha}}$ to be the $(q,t)$-content of the box $\lambda^{(n+i+1)}_{\alpha}\setminus \lambda^{(n+i)}_{\alpha}$ if the box labeled by $i$ is in part $\alpha$, and zero otherwise. We also define $w_i=\sum a_{\alpha}w^{(n+i)}_{\alpha}$ to be the $\widetilde{T}$-weight of the box labeled by $i$ (there's only one nonzero term in the sum), and 
\[B_{\lambda_{\bullet}^{(n+i)}}=\sum_{\alpha} a_{\alpha}B_{\lambda^{(n+i)}_{\alpha}}.\]

\begin{remark}
By Theorem \ref{thm: arm leg} the character of the cotangent space to $\M(r,n)$ at a point $\lambda_{\bullet}$ is given by:
\[
\Omega_{\lambda_{\bullet}} = \sum_{\alpha, \beta = 1}^{r} a_{\alpha}a_{\beta}^{-1}\left(q^{-1}t^{-1}B^{*}_{\lambda_{\alpha}} + B_{\lambda_{\beta}} - (1-q^{-1})(1-t^{-1})B^{*}_{\lambda_{\alpha}}B_{\lambda_{\beta}}\right).
\]

To not carry the $q^{-1}, t^{-1}$ factors we will \emph{twist} the action of $\C^{*} \times \C^{*}$ by $q \leftrightarrow q^{-1}, t \leftrightarrow t^{-1}$. Obviously, the set of fixed points does not change and we now obtain:
\begin{equation}\label{eq: cotangent weight}
\Omega_{\lambda^{\bullet}}\M(r,n) = \sum_{\alpha, \beta = 1}^{r} a_{\alpha}a_{\beta}^{-1}\left(qt B_{\lambda_{\alpha}} + B^{*}_{\lambda_{\beta}} - (1-q)(1-t)B_{\lambda_{\alpha}}B^{*}_{\lambda_{\beta}}\right).
\end{equation}
Note that this formula coincides with \cite[(4.2.1)]{CGM} when $r = 1$. 
\end{remark}

\begin{theorem}
The character of the cotangent space to $\M^{par}(r,n;n+k)$ at a fixed point  is given by the formula\footnote{We note that our $(q,t)$-content differs from that of \cite{NY, NY2} by $q \leftrightarrow q^{-1}, t \leftrightarrow t^{-1}$}:
\begin{align}
\Omega_{\lambda_{\bullet}^{(\star)}} &=   \displaystyle{\sum_{\alpha, \beta = 1}^{r}a_{\alpha}a_{\beta}^{-1}\left(qtB_{\lambda_{\alpha}^{(n)}} + B^{*}_{\lambda_{\beta}^{(n)}} - (1-q)(1-t)B_{\lambda_{\alpha}^{(n)}}B^{*}_{\lambda_{\beta}^{(n)}}\right)  
} \nonumber \\
&+  \displaystyle{\sum_{\alpha, \beta = 1}^{r}a_{\alpha}a_{\beta}^{-1}\left(\sum_{i = 1}^{k}\left(w_{\beta}^{(i)}\right)^{-1} - (1-q)(1-t)B_{\lambda_{\alpha}^{(n)}}\left(w_{\beta}^{(i)}\right)^{-1} - (1-q)\sum_{i \leq j}w_{\alpha}^{(i)}\left(w_{\beta}^{(j)}\right)^{-1} \right)} \nonumber \\
&=  \displaystyle{\Omega_{\lambda^{(n)}} + \sum_{\alpha}a_{\alpha}\sum_{i = 1}^{k}w_{i}^{-1} - (1-q)(1-t)B_{\lambda^{(n)}_{\bullet}}\sum_{i = 1}^{k}w_{i}^{-1} - (1-q)\sum_{i \geq j}w_{i}^{-1}w_{j} \label{eq: character parabolic}
}
\end{align}
\end{theorem}

\begin{proof}
We follow the proof of Theorem \ref{thm: smooth}.
The sheaf $\F$ glued from $\F_i$ corresponds to an $r$-tuple of diagrams 
$$
\lambda_{\alpha}=\left(\lambda_{\alpha,1}^{(n+k)},\ldots ,\lambda_{\alpha,1}^{(n)},\lambda_{\alpha,2}^{(n+k)},\ldots ,\lambda_{\alpha,2}^{(n)},\ldots\right)
$$
Note that the horizontal strip condition implies 
$\lambda_{\alpha,2}^{(n+k)}\le \lambda_{\alpha,1}^{(n)}$, so $\lambda_{\alpha}$ is a well defined Young diagram.
We have: 
\[
B_{\lambda_{\alpha}} = B_{\lambda_{\alpha}^{(n+k)}}(q, t^{k+1}) + tB_{\lambda_{\alpha}^{(n+k-1)}}(q, t^{k+1}) + \cdots + t^{k}B_{\lambda_{\alpha}^{(n)}}(q, t^{k+1}).
\]

To compute the tangent weight for $(\F_{n} \supseteq \cdots \supseteq \F_{n+k}) \in \M^{par}(r,n;n+k)$, we have to compute the tangent weight for $\F \in \M(r,c_2(\F))$, pick up terms with $t$-degree divisible by $(k+1)$ and change $q^at^{b(k+1)}$ to $q^at^b$. We get:

\begin{align*}
\sum_{\alpha,\beta=1}^{r}a_{\beta}a_{\alpha}^{-1}\left[
qtB_{\lambda_{\alpha}^{(n)}}^{*}+B_{\lambda_{\beta}^{(n+k)}}-(1-q)\left(\sum_{i=0}^{k}B_{\lambda_{\alpha}^{(n+i)}}^*B_{\lambda_{\beta}^{(n+i)}}-tB_{\lambda_{\alpha}^{(n+k)}}^*B_{\lambda_{\beta}^{(n)}}-\sum_{i=0}^{k-1} B_{\lambda_{\alpha}^{(n+i)}}^*B_{\lambda_{\beta}^{(n+i+1)}}
\right)
\right]\\
=
\sum_{\alpha,\beta=1}^{r}a_{\beta}a_{\alpha}^{-1}\left[
qtB_{\lambda_{\alpha}^{(n)}}^{*}+B_{\lambda_{\beta}^{(n+k)}}-(1-q)\left(B_{\lambda_{\alpha}^{(n+k)}}^*B_{\lambda_{\beta}^{(n+k)}}-tB_{\lambda_{\alpha}^{(n+k)}}^*B_{\lambda_{\beta}^{(n)}}-\sum_{i=1}^{k} B_{\lambda_{\alpha}^{(n+i)}}^*w_{\beta}^{(i)}
\right)
\right].
\end{align*}
Observe that 
\begin{align*}
B_{\lambda_{\beta}^{(n+k)}}&=B_{\lambda_{\beta}^{(n)}}+\sum_{i=1}^{k}w^{(i)}_{\beta},
\\
B_{\lambda_{\alpha}^{(n+k)}}^*B_{\lambda_{\beta}^{(n+k)}}&=B_{\lambda_{\alpha}^{(n)}}^*B_{\lambda_{\beta}^{(n)}}+\sum_{i=1}^{k} \left(w^{(i)}_{\alpha}\right)^{-1}B_{\lambda_{\beta}^{(n)}}+\sum_{i=1}^{k} B_{\lambda_{\alpha}^{(n)}}^*(w^{(i)}_{\beta})+\sum_{i,j=1}^{k} \left(w^{(i)}_{\alpha}\right)^{-1}w^{(j)}_{\beta},
\\
B_{\lambda_{\alpha}^{(n+k)}}^*B_{\lambda_{\beta}^{(n)}}&=B_{\lambda_{\alpha}^{(n)}}^*B_{\lambda_{\beta}^{(n)}}+\sum_{i=1}^{k} \left(w^{(i)}_{\alpha}\right)^{-1}B_{\lambda_{\beta}}^{(n)},
\\
\sum_{i=1}^{k} B_{\lambda_{\alpha}^{(n+i)}}^*w_{\beta}^{(i)}&=\sum_{i=1}^{k} B_{\lambda_{\alpha}^{(n)}}^*(w^{(i)}_{\beta})+\sum_{i<j}\left(w^{(i)}_{\alpha}\right)^{-1}w^{(j)}_{\beta}.
\end{align*}

By collecting the terms, we get
\begin{multline*}
\sum_{\alpha,\beta=1}^{r}a_{\beta}a_{\alpha}^{-1}\left[
qtB_{\lambda_{\alpha}^{(n)}}^{*}+B_{\lambda_{\beta}^{(n)}}-(1-q)(1-t)B_{\lambda_{\alpha}^{(n)}}^*B_{\lambda_{\beta}^{(n)}}
\right]\\
+ \sum_{\alpha,\beta=1}^{r}a_{\beta}a_{\alpha}^{-1}\left[
\sum_{i=1}^{k}w^{(i)}_{\beta}-(1-q)(1-t)\left(w^{(i)}_{\alpha}\right)^{-1}B_{\lambda_{\beta}^{(n)}}-(1-q)\sum_{i\ge j}\left(w^{(i)}_{\alpha}\right)^{-1}w^{(j)}_{\beta}
\right],
\end{multline*}
which agrees with \eqref{eq: character parabolic}.
\end{proof}

\begin{corollary}
The dimension of $\M^{par}(r,n;n+k)$ equals $(2n+k)r$.
\end{corollary}

\begin{proof}
 Indeed, at $q=t=1$ the first sum in \eqref{eq: character parabolic} has $\dim \M(r,n)=2rn$ terms, the second sum has $kr$ terms and the rest vanishes.
\end{proof}

\subsection{Construction of $d_{+}$ and $d_{-}$ Operators for the Gieseker Space}

Let us fix $r > 0$. For $k \geq 0$, let us define the  $\C(q,t, a_{1}, \dots, a_{r})$-vector spaces:
\begin{equation}\label{eq:GiesekerRepn}
V_{k}^{(r)} := \bigoplus_{n \geq 0}K^{\widetilde{T}}_{loc}(\mathcal{M}^{par}(r;n, n+k))
\end{equation}
(in particular, $V^{(r)}_{0}$ can be identified with the $k$-Fock space). We will construct operators 
\[
d_{+}: V^{(r)}_{k} \to V^{(r)}_{k+1}, \qquad d_{-}: V^{(r)}_{k} \to V^{(r)}_{k-1}
\]
exactly in the same way as in \cite{CGM}.  Namely, we have forgetful maps
$$
\begin{array}{ccc}f:\M^{par}(r,n;n+k+1)\to \M^{par}(r,n;n+k) &&  g:\M^{par}(r,n;n+k)\to \M^{par}(r;n+1;n+k) \\
(\F_{n} \supset \cdots \supset \F_{n+k+1}) \mapsto (\F_{n} \supset \cdots \supset \F_{n+k}) && (\F_{n} \supset \cdots \supset \F_{n+k}) \mapsto (\F_{n+1} \supset \cdots \supset \F_{n+k})
\end{array}
$$
and we define:
$$
d_{-}=g_*,\ d_+=q^k(q-1)f^*.
$$
On the level of fixed points, $g$ sends each a tuple $$d_-:\left(\lambda_{\alpha}^{(n)},\ldots,\lambda_{\alpha}^{(n+k)}\right)\mapsto \left(\lambda_{\alpha}^{(n+1)},\ldots,\lambda_{\alpha}^{(n+k)}\right)
$$
(forgets the smallest diagram) with coefficient 1, while 
$$
d_+:\left(\lambda_{\alpha}^{(n)},\ldots,\lambda_{\alpha}^{(n+k)}\right)\mapsto \left(\lambda_{\alpha}^{(n)},\ldots,\lambda_{\alpha}^{(n+k)},\lambda_{\alpha}^{(n+k+1)}\right)
$$
(creates a bigger diagram) with some coefficient which we now compute.

\begin{definition} \label{def:Lambda-operator}
Let 
\[\phi(q,t,a_1,\ldots,a_r)=\sum \phi_{i,j,s_1,\ldots,s_r}q^it^ja_1^{s_1}\cdots a_r^{s_r}\]
 be any polynomial in $\Z[q^{\pm},t^{\pm},a_1^{\pm},\ldots,a_r^{\pm}]$ with $\phi_{0,\ldots,0}=0$. Define: 
\begin{equation}
\Lambda(\phi(q,t,a_1,\ldots,a_r)):=\prod (1-q^it^ja_1^{s_1}\cdots a_r^{s_r})^{\phi_{i,j,s_1,\ldots,s_r}}.
\end{equation}
\end{definition}

\begin{remark}
If $\phi(q,t,a_1,\ldots,a_r)$ is the character of a (virtual) representation $V$ of the torus $\widetilde{T}=(\C^{\times})^{2}\times T$, then $\Lambda(\phi(q,t,a_1,\ldots,a_r))$ is the $\widetilde{T}$-character of the exterior algebra $\Lambda^{\bullet}V$.
\end{remark}

\begin{lemma}\label{lem:d+ coeff}
The coefficient in 
$d_+\left(\lambda_{\alpha}^{(n)}, \dots, \lambda_{\alpha}^{(n+k)}\right)$ 
at the fixed point 
$$\left(\lambda_{\alpha}^{(n)}, \dots, \lambda_{\alpha}^{(n+k)}, \lambda_{\alpha}^{(n+k)}\cup\{x\}\right)$$
 equals
$
-q^kxD\left(\lambda^{(n+k)}_{\bullet},x\right)\prod_{i}\frac{(x-tw_i)}{(x-qtw_i)}
$
where 
$$
D\left(\lambda^{(n+k)}_{\bullet},x\right):=x^{-1}\Lambda\left(-Ax^{-1}+(1-q)(1-t)B_{{\lambda_{\bullet}}^{(n+k)}}x^{-1}+1\right)
$$
and $A=\sum a_i$.
\end{lemma}

\begin{remark}
For $r=1$ we have $A=1$ and $D\left(\lambda^{(n+k)}_{\bullet},x\right)$ was denoted by $d_{\lambda+x,x}$ in \cite{CGM}, see Section \ref{sec: poly}.
\end{remark}

\begin{remark}
As in Section \ref{sec: poly}, we can check that $D\left(\lambda^{(n+k)}_{\bullet},x\right)$ is well defined, that is, the constant term in 
$$
-Ax^{-1}+(1-q)(1-t)B_{{\lambda_{\bullet}}^{(n+k)}}x^{-1}+1
$$
vanishes whenever $x$ is addable for $\lambda_{\bullet}^{(n+k)}$. Assume it is addable for $\lambda_{\alpha}^{(n+k)}$, then the constant terms in $a_{\beta}x^{-1}$ and 
$(1-q)(1-t)B_{\lambda_{\beta}}^{(n+k)}x^{-1}$ vanish for all $\beta\neq \alpha$.

Furthermore, if $\lambda_{\alpha}^{(n+k)}$ is empty then $x=a_{\alpha}$ and $-a_{\alpha}x^{-1}+1=0$. If $\lambda_{\alpha}^{(n+k)}$ is nonempty then the constant term of $a_{\alpha}x^{-1}$ vanishes and similarly to Section \ref{sec: poly} the constant term of $(1-q)(1-t)B_{\lambda_{\alpha}}^{(n+k)}x^{-1}+1$ vanishes as well.
\end{remark}

\begin{proof}[Proof of Lem. \ref{lem:d+ coeff}]
We begin by computing the coefficient of $\left(\lambda^{(n)}_{\bullet}, \dots, \lambda^{(n+k)}_{\bullet}, \lambda^{(n+k)}_{\bullet}\cup\{x\}\right)$ in $f^{*}\left(\lambda^{(n)}_{\bullet}, \dots, \lambda^{(n+k)}_{\bullet}\right)$, after which we need only multiply by $q^{k}(q-1)$. The coefficient of $f^*$ is given by the exterior algebra of the relative cotangent space, hence using \eqref{eq: character parabolic} we compute: 
\begin{align}
\ch\Omega&\left(\lambda_{\bullet}^{(n)}, \dots, \lambda_{\bullet}^{(n+k)}\right) - \ch\Omega\left(\lambda_{\bullet}^{(n)}, \dots, \lambda_{\bullet}^{(n+k)}, \lambda_{\bullet}^{(n+k)}\cup\{x\}\right) \nonumber \\
&={ -Ax^{-1}+(1-q)(1-t)B_{{\lambda_{\bullet}}^{(n)}}x^{-1}+(1-q)+(1-q)\sum_{i = 1}^{k} x^{-1}w_i} \nonumber \\
&={-Ax^{-1}+(1-q)(1-t)B_{{\lambda_{\bullet}}^{(n+k)}}x^{-1}+(1-q)+t(1-q)\sum_{i = 1}^{k} x^{-1}w_i }. \label{eq:diff cotangent}
\end{align}

Taking the exterior algebra of \eqref{eq:diff cotangent},we find that the coefficient of $\left(\lambda^{(n)}_{\bullet}, \dots, \lambda^{(n+k)}_{\bullet}, \lambda^{(n+k)}_{\bullet}\cup\{x\}\right)$ in $f^{*}\left(\lambda^{(n)}_{\bullet}, \dots, \lambda^{(n+k)}_{\bullet}\right)$ is
$
xD(\lambda^{(n+k)}_{\bullet},x)\frac{1}{1-q}\prod_{i = 1}^{k}\frac{x-tw_{i}}{x - qtw_{i}}
$
and the result follows. 
\end{proof}

\subsection{Verification of the $\Bqt$ Relations for the Gieseker Space}

The computation of $d_{+}d_{-},d_{-}d_{+}$ and $d_{+}d_{-}-d_{-}d_{+}$,$qd_{+}d_{-}-d_{-}d_{+}$ follows 
\cite[Section 6.2]{CGM} verbatim, in particular, we get the analogues of \cite[eq. (6.1),(6.2)]{CGM} with
$D\left(\lambda_{\bullet}^{(n+k)},x\right)$ instead of $d_{\lambda+x,x}$. 

It is immediate from Lemma \ref{lem:d+ coeff} that we have:
\begin{align}
d_{-}d_{+}\Bigl(\lambda_{\bullet}^{(n)}, \dots, \lambda_{\bullet}^{(n+k)}\Bigr) 
= -q^{k}\sum_{x}D(\lambda_{\bullet}^{(n+k)}, x)\prod_{i = 1}^{k}\frac{x - tw_{i}}{x-qtw_{i}}\Bigl(\lambda^{(n+1)}_{\bullet}, \dots, \lambda^{(n+k)}_{\bullet}, \lambda^{(n+k)}_{\bullet}\cup\{x\}\Bigr)& \nonumber \\
= -q^{k}\sum_{x}D(\lambda_{\bullet}^{(n+k)}, x)\frac{x - tw_{1}}{x-qtw_{1}}\prod_{i = 2}^{k}\frac{x - tw_{i}}{x-qtw_{i}}\Bigl(\lambda^{(n+1)}_{\bullet}, \dots, \lambda^{(n+k)}_{\bullet}\cup\{x\}\Bigr)& \label{eq:d-d+}
\end{align}
and
\begin{equation}\label{eq:d+d-}
\begin{array}{rl}
 d_{+}d_{-}\Bigl(\lambda_{\bullet}^{(n)}, \dots, \lambda_{\bullet}^{(n+k)}\Bigr)  = & \displaystyle{-q^{k-1}\sum_{x}xD(\lambda_{\bullet}^{(n+k)}, x)\prod_{i = 2}^{k}\frac{x - tw_{i}}{x-qtw_{i}}\Bigl(\lambda_{\bullet}^{(n+1)}, \dots, \lambda_{\bullet}^{(n+k)}, \lambda_{\bullet}^{(n+k)}\cup\{x\}\Bigr).}
\end{array}
\end{equation}
From \eqref{eq:d+d-} and \eqref{eq:d-d+} we obtain:
\begin{align*}
\frac{[d_{+},d_{-}]}{q-1}&\left(\lambda_{\bullet}^{(n)}, \dots, \lambda^{(n+k)}_{\bullet}\right)
   \\
&= \frac{q^{k-1}}{1-q}\sum_{x}xD(\lambda^{(n+k)}_{\bullet}, x)\prod_{i = 2}^{k}\frac{x - tw_{i}}{x - qtw_{i}}\left(1 - q\frac{x-tw_{1}}{x - qtw_{i}}\right)\left(\lambda^{(n+1)}_{\bullet}, \dots, \lambda^{(n+k)}_{\bullet}\cup\{x\}\right)
 \\
&= q^{k-1}\sum_{x}xD(\lambda^{(n+k)}_{\bullet}, x)\frac{x}{x-qtw_{1}}\prod_{i = 2}^{k}\frac{x - tw_{i}}{x-qtw_{i}}\left(\lambda^{(n+1)}_{\bullet}, \dots, \lambda^{(n+k)}_{\bullet}\cup\{x\}\right).
\end{align*}
as well as,
\begin{align*}
\frac{qd_{+}d_{-} - d_{-}d_{+}}{q -1}&\left(\lambda_{\bullet}^{(n)}, \dots, \lambda_{\bullet}^{(n+k)}\right) 
\\
&=\frac{q^{k}}{1-q}\sum_{x}xD(\lambda^{(n+k)}_{\bullet}, x)\prod_{i = 2}^{k}\frac{x - tw_{i}}{x - qtw_{i}}\left(1 - \frac{x-tw_{1}}{x - qtw_{i}}\right)\left(\lambda^{(n+1)}_{\bullet}, \dots, \lambda^{(n+k)}_{\bullet}\cup\{x\}\right) 
 \\
&=q^{k}t\sum_{x}xD(\lambda^{(n+k)}_{\bullet}, x)\frac{w_1}{x-qtw_{1}}\prod_{i = 2}^{k}\frac{x - tw_{i}}{x-qtw_{i}}\left(\lambda^{(n+1)}_{\bullet}, \dots, \lambda^{(n+k)}_{\bullet}\cup\{x\}\right).
\end{align*}
Likewise, we compute:
\begin{equation*}\label{eqn:d+squared 1}
    \begin{array}{l}
d_{+}^{2}\left(\lambda^{(n)}_{\bullet}, \dots, \lambda^{(n+k)}_{\bullet}\right) = q^{2k+1}\times\\
\displaystyle{\sum_{x,y}xD(\lambda_{\bullet}^{(n+k)}, x)yD(\lambda_{\bullet}^{(n+k)}\cup\{x\}, y)\left(\prod_{i = 1}^{k}\frac{x-tw_{i}}{x-qtw_{i}}\frac{y - tw_{i}}{y - qtw_{i}}\right)\frac{y - tx}{y - qtx}(\lambda^{(n)}_{\bullet}, \dots, \lambda^{(n+k)}_{\bullet}\cup\{x,y\})
}
    \end{array}
\end{equation*}
so that,
\begin{align*}
D(\lambda^{(n+k)}_{\bullet},x)D(\lambda^{(n+k+1)}_{\bullet},y)&=(xy)^{-1}\Lambda \left(
-\left(A-(1-q)(1-t)B_{\lambda^{(k)}_{\bullet}}\right)(x^{-1}+y^{-1}\right)\\
&+(1-q)(1-t)y^{-1}x+2
 ),
\end{align*}
and conclude
\begin{equation}\label{eqn:d+squared 2}
d_{+}^{2}\left(\lambda_{\bullet}^{(n)}, \dots, \lambda_{\bullet}^{(n+k)}\right) = q^{2k+1}\sum_{x, y}C_{\lambda^{(n+k)}_{\bullet}, w}(x,y)\frac{y-x}{y-qx}\left(\lambda^{(n)}_{\bullet}, \dots, \lambda^{(n+k)}_{\bullet}\cup\{x\}, \lambda^{(n+k)}_{\bullet}\cup\{x,y\}\right)
\end{equation}
where $C_{\lambda^{(n+k)}_{\bullet},w}(x,y)$ is symmetric in $x,y$, compare \cite[eq. (6.2.4)]{CGM}. From here, it follows easily that $T_{1}d_{+}^{2} = d_{+}^{2}$.

\section{Tensor Products of Representations}\label{sec: tensor}

\subsection{General Position and Twisted Posets} Given two excellent weighted posets $E_1$ and $E_2$, we would like for the product poset $E_1 \times E_2$ to be excellent as well. In order for this to be the case, we need to impose the following condition on the weights on $E_1$ and $E_2$.

    \begin{definition}\label{def: general position}
    Let $E_1, E_2$ be weighted posets. The pair $E_1$ and $E_2$ are in \newword{general position} (with each other) if the following conditions are satisfied:
    \begin{enumerate}
        \item For every cover relation $\lambda_1 \xrightarrow {z_1} \mu_1$ in $E_1$, and every cover relation $\lambda_2 \xrightarrow {z_2} \mu_2$ in $E_2$, one has $z_1/z_2 \not\in \{1, q^{\pm 1}, t^{\pm 1}, (qt)^{\pm 1}\}$.
        \item If $\lambda_1, \mu_1 \in E_1$ and $\lambda_2, \mu_2 \in E_2$ are such that
        \[
        p_m(\lambda_1) + p_m(\mu_1) = p_m(\lambda_2) + p_m(\mu_2)
        \]
        for all $m > 0$; then $\lambda_1 = \lambda_2$ and $\mu_1 = \mu_2$. 
        \end{enumerate}
\end{definition}

Note that a weighted poset $E$ with at least one covering relation is never in general position with respect to itself. However, we can \newword{\lq\lq twist\rq\rq}  any poset as follows.

\begin{definition}\label{def: twist}
Let $E$ be a weighted poset, and $a\in \C(q,t)^*$ (equivalently, let $a$ be a formal variable). We define the \newword{twisted weighted poset} $E(a)$ as follows: the underlying poset is the same, but for each $a \in E$ we denote the corresponding element in $E(a)$ by $\lambda(a)$. Furthermore, the weightings are related via
$$
p_m(\lambda(a))=a^m\cdot p_m(\lambda).
$$
\end{definition}

It is easy to see that $\lambda\cup x=\mu$ if and only if
$\lambda(a)\cup (ax)=\mu(a)$.

\begin{remark}
Let $E_1$ and $E_2$ be weighted posets. By choosing sufficiently generic $a_1$ and $a_2$ (or independent variables $a_1$ and $ a_2$) we can always twist so that $E_1(a_1)$ and $E_2(a_2)$ are in general position. 
\end{remark}

For any sequence $[\underline{w}]$, let $[a \underline{w}] = [aw_k, \dots, aw_1]$. 

\begin{proposition}
a) The twisted weighted poset $E(a)$ is excellent if and only if $E$ is excellent.

\noindent b) A chain $[\lambda(a);\underline{w}]$ in $E(a)$ is good if and only if $[\lambda;a^{-1}\underline{w}]$ is a good chain in $E$.

\noindent c) Define $c(\lambda(a);x):=c(\lambda;a^{-1}x)$. Then \eqref{eq: monodromy} holds for $E(a)$ if and only if it holds for $E$.
\end{proposition}

\begin{proof}
Parts (a) and (b) are straightforward. To prove part (c), observe that the right hand side of \eqref{eq: monodromy} is homogeneous in $x$ and $y$, so
\begin{align*}
\frac{c(\lambda(a);x)c(\lambda(a)\cup x;y)}{c(\lambda(a);y)c(\lambda(a)\cup y;x)}&=\frac{c(\lambda;a^{-1}x)c(\lambda\cup a^{-1}x;a^{-1}y)}{c(\lambda;a^{-1}y)c(\lambda\cup a^{-1}y;a^{-1}x)}
\\
&= -\frac{(a^{-1}x-ta^{-1}y)(a^{-1}x-qa^{-1}y)(a^{-1}y-qta^{-1}x)}{(a^{-1}y-ta^{-1}x)(a^{-1}y-qa^{-1}x)(a^{-1}x-qta^{-1}y)}\\
&=-\frac{(x-ty)(x-qy)(y-qtx)}{(y-tx)(y-qx)(x-qty)}.
\end{align*}
\end{proof}

To conclude the discussion on twisting, we show that given an excellent poset $E$, the representations $V(E)$ and $V(E(a))$ are related via an automorphism of $\Bqt^{\ext}$.

\begin{lemma}
The algebra $\Bqt^{\ext}$ has an automorphism given by
$$
z_i\to az_i,\ T_i\to T_i,\ d_{+}\to d_{+},\ d_{-}\to d_{-},\ \Delta_{p_m}\to a^m\Delta_{p_m}.
$$
The calibrated representations associated to $V(E)$ and $V(E(a))$ are related by this automorphism.
\end{lemma}

\begin{proof}
It is easy to see from the relations that we indeed get an automorphism of $\Bqt^{\ext}$. To check that it agrees  with the representation $V(E(a))$, observe that the action of $T_i$ depends only on the ratio $w_i/w_{i+1}$ and hence is not modified by the twist.  By construction, the coefficients of $\Delta_{p_m}$ and $z_i$ are scaled by $a^m$ and by $a$, respectively, and the coefficients of $d_-$ remain the same. Finally, the coefficients of $d_{+}$ are equal to
\[
q^kc(\lambda(a)\cup \underline{w};x)\prod_{i=1}^{k}\frac{x-tw_i}{x-qtw_i}=q^kc(\lambda\cup a^{-1}\underline{w};a^{-1}x)\prod_{i=1}^{k}\frac{a^{-1}x-ta^{-1}w_i}{a^{-1}x-qta^{-1}w_i}
\]
so the action of $d_+$ is preserved for both $V(E)$ and $V(E(a))$.
\end{proof} 

\subsection{Products of Posets} We can now define the product of weighted posets.

\begin{definition}
    
Let $(E_{1}, \prec_{1}, p^{(1)}_{m}: E_{1} \to \C(q,t))$ and $(E_{2}, \prec_{2}, p^{(2)}_{m}: E_{2} \to \C(q,t))$ be excellent weighted posets. We define their product $(E, \prec, p_{m}: E \to \C(q,t))$ as follows. Set:
\begin{itemize}
\item $E := E_{1} \times E_{2}$.
\item $(\lambda^{(1)}, \lambda^{(2)}) \preceq (\mu^{(1)}, \mu^{(2)})$ if and only if $\lambda^{(1)} \preceq_{1} \mu^{(1)}$ and $\lambda^{(2)} \preceq_{2} \mu^{(2)}$.
\item $p_{m}(\lambda^{(1)}, \lambda^{(2)}) :=
p^{(1)}_{m}(\lambda^{(1)}) + p^{(2)}_{m}(\lambda^{(2)})$.
\end{itemize}
\end{definition}

It follows that the poset $E$ is graded, with $|(\lambda^{(1)}, \lambda^{(2)})| = |\lambda^{(1)}| + |\lambda^{(2)}|$.

\begin{remark}
    Note that if $E_1$ and $E_2$ are in general position, then by Definition \ref{def: general position}(2) the weighting on $E_1 \times E_2$ has simple spectrum. 
\end{remark}

{\bf For the rest of this section, we will always assume that the weighted posets $E_1$ and $E_2$ are in general position.}

\begin{proposition}
A weight $x$ is addable for $(\lambda^{(1)},\lambda^{(2)})$ if and only if either $x$ is addable for $\lambda^{(1)}$ or $x$ is addable for $\lambda^{(2)}$.
\end{proposition}

\begin{proof}
 It is obvious that $(\mu^{(1)}, \mu^{(2)})$ covers $(\lambda^{(1)}, \lambda^{(2)})$ if and only if $\mu^{(1)} = \lambda^{(1)}$ and $\mu^{(2)}$ covers $\lambda^{(2)}$, or $\mu^{(1)}$ covers $\lambda^{(1)}$ and $\mu^{(2)} = \lambda^{(2)}$. Thus, the degrees of elements in a covering pair differ by $1$, as required. Moreover, if say $\mu^{(1)} = \lambda^{(1)}$ and $\mu^{(2)}$ covers $\lambda^{(2)}$, say $\mu^{(2)} = \lambda^{(2)}\cup x$ then
\[
p_{m}(\mu^{(1)}, \mu^{(2)}) = p_{m}^{(1)}(\lambda_{1}) + p_{m}^{(2)}(\lambda_{2}) + x^{m}
\]
so that $(\mu^{(1)}, \mu^{(2)}) = (\lambda^{(1)}, \lambda^{(2)}) \cup x$. 
\end{proof}

The next natural direction is the construction of chains in $E_1\times E_2$. 
Recall that $\sigma\in S_{k+n}$ is a $(k,n)$ \emph{shuffle permutation} if
$\sigma(1)<\ldots<\sigma(k)$ and $\sigma(k+1)<\ldots<\sigma(k+n)$.
The following statement is clear. 

\begin{proposition}
Suppose that $[\lambda^{(1)};\underline{w}]$
and $[\lambda^{(2)};\underline{z}]$ are chains in $E_1$ and $E_2$ respectively, and $\sigma$ a $(k,n)$ shuffle permutation. Then
\begin{equation}
\label{eq: shuffle chains}
\left[(\lambda^{(1)},\lambda^{(2)})\;;\;\sigma(\underline{w},\underline{z})\right]
\end{equation}
is a chain in $E_1\times E_2$, and all chains in $E_1\times E_2$ are given by this construction.
\end{proposition}

\begin{lemma}\label{lem: product excellent}
Assume $E_1$ and $E_2$ are excellent and in general position. Then $E_1\times E_2$ is excellent. 
\end{lemma}

\begin{proof}
Suppose that $\left[(\lambda^{(1)},\lambda^{(2)});y,x\right]$ is a 2-chain in $E_1\times E_2$, $y$ is addable for 
$(\lambda^{(1)},\lambda^{(2)})$ and $y\notin 
\{qx,tx\}$. We have the following cases:
\begin{itemize}[leftmargin=5mm]
\item[1)] $[\lambda^{(1)};y,x]$  is a chain in $E_1$.  Since $E_1$ is excellent, $y\neq qx$ and $y\neq tx$, $y$ is addable for $\lambda^{(1)}$ and $x$ is addable for $\lambda^{(1)}\cup y$. Therefore $y$ is addable for $(\lambda^{(1)},\lambda^{(2)})$ and 
$x$ is addable for $(\lambda^{(1)}\cup y,\lambda^{(2)})$.

\item[2)] $[\lambda^{(2)};y,x]$  is a chain in $E_2$, this is similar to (1).

\item[3)] $y$ is addable for $\lambda^{(2)}$ and $x$ is addable for $\lambda^{(1)}$. Then $y$ is addable for $(\lambda^{(1)},\lambda^{(2)})$ and $x$ is addable for $(\lambda^{(1)},\lambda^{(2)}\cup y)$.

\item[4)] $y$ is addable for $\lambda^{(1)}$ and
$x$ is addable for $\lambda^{(2)}$. Then
$y$ is addable for $(\lambda^{(1)},\lambda^{(2)})$
and $x$ is addable for $(\lambda^{(1)}\cup y,\lambda^{(2)})$.
\end{itemize}
Finally, if $y=qx$ or $y=tx$ then we are in cases (1) or (2) since $E_1$ and $E_2$ are in general position, and the statement follows from $E_1$ and $E_2$ being excellent. In either case, $y\neq qtx$.
\end{proof}

\begin{corollary}
The chain $\left[(\lambda^{(1)},\lambda^{(2)});\sigma(\underline{w},\underline{z})\right]
$ is good if and only if both $[\lambda^{(1)};\underline{w}]$
and $[\lambda^{(2)};\underline{z}]$ are good.
\end{corollary}

\subsection{Tensor Products of Representations and Generically Monoidal Categories}

Assume $E_1$ and $E_2$ are excellent weighted posets in general position so that $E_1 \times E_2$ is also an excellent weighted poset. In this subsection we describe the representation of $\Bqt^{\ext}$ corresponding to the product $E_1\times E_2$. While Theorem \ref{thm: existence and uniqueness} ensures there always exist edge functions for $E_1\times E_2$ which determine the coefficients of $d_+$, in what follows we give explicit formulas for these coefficients in terms of the edge functions of $E_1$ and $E_2$.

\begin{lemma}
\label{lem: M psi}
Let $\psi(y,u)$ be a rational function in $\C(q,t)(y,u)$. Then, for any weighted poset $E$ there exists a collection of rational functions $M_{\psi}(\lambda;u) \in \C(q,t)(u)$ indexed by $\lambda\in E$ such that:

\noindent a) For any addable weight $y$ for $\lambda$ we have,
\begin{equation}
\label{eq: M psi}
M_{\psi}(\lambda\cup y;u)=M_{\psi}(\lambda;u)\psi(y,u).
\end{equation}

\noindent b) For the weighted poset $E(a)$, we can set $M_{\psi}(\lambda(a);u)=M_{\widetilde{\psi}}(\lambda;u)$ with $\widetilde{\psi}(y,u)=\psi(ay,u)$.

\noindent c) For the weighted poset $E_1\times E_2$, we can set $M_{\psi}((\lambda,\mu);u)=M_{\psi}(\lambda;u)M_{\psi}(\mu;u)$.  
\end{lemma}

\begin{proof}
First choose an element $\lambda$ in each connected component of $E$ and define $M_{\psi}(\lambda;u)$ arbitrarily for these. Now, extend it to the rest of $E$ uniquely using \eqref{eq: M psi}. This proves existence, provided $M_{\psi}(\lambda;u)$ is well defined.

\noindent a) If $[\lambda;\underline{w}]$ is a chain and $\mu=\lambda\cup \underline{w}$ then \eqref{eq: M psi} implies
\[
M_{\psi}(\mu;u)=M_{\psi}(\lambda;u)\psi(w_k,u)\cdots \psi(w_1,u),
\]
which is symmetric in $w_1,\ldots,w_k$. Since any chain from $\lambda$ to $\mu$ is given by a permutation of $\underline{w}$, cf. Lemma \ref{lem: uniqueness of weights up to permutation}, we conclude that $M_{\psi}$ is well defined and satisfies \eqref{eq: M psi}.

\noindent b) In $E(a)$ we have
$$
\frac{M_{\psi}(\lambda(a)\cup y;u)}{M_{\psi}(\lambda(a);u)}=\frac{M_{\widetilde{\psi}}(\lambda\cup a^{-1}y;u)}{M_{\widetilde{\psi}}(\lambda;u)}=\widetilde{\psi}(a^{-1}y,u)=\psi(y,u).
$$

\noindent c) In $E_1\times E_2$ we have two cases. If $y$ is addable for $\lambda$ then
$$
\frac{M_{\psi}((\lambda,\mu)\cup y;u)}{M_{\psi}((\lambda,\mu);u)}=\frac{M_{\psi}((\lambda\cup y,\mu);u)}{M_{\psi}((\lambda,\mu);u)}=\frac{M_{\psi}(\lambda\cup y; u)M_{\psi}(\mu;u)}{M_{\psi}(\lambda;u)M_{\psi}(\mu;u)}=\psi(y,u).
$$
Similarly, if $y$ is addable for $\mu$ then
$$
\frac{M_{\psi}((\lambda,\mu)\cup y;u)}{M_{\psi}((\lambda,\mu);u)}=
\frac{M_{\psi}(\lambda; u)M_{\psi}(\mu\cup y;u)}{M_{\psi}(\lambda;u)M_{\psi}(\mu;u)}=\psi(y,u).
$$
\end{proof}

We now give an explicit formula for the edge function on $E_1 \times E_2$ in terms of the edge functions on $E_1$ and $E_2$, and the $\psi$ functions from Lemma \ref{lem: M psi}. First, we make the following definition. 

\begin{definition}
Let $E_1$ and $E_2$ be weighted posets and let  $\psi(y,u)$ be a rational function in $y$ and $u$. The pair $E_1$ and $E_2$ are in \newword{very general position} \emph{with respect to $\psi$} if the following two conditions are satisfied:
\begin{enumerate}
    \item $E_1$ and $E_2$ are in general position. 
    \item For any covering relation $\lambda_1 \xrightarrow {z_1} \mu_1$ in $E_1$ and any covering relation $\lambda_2 \xrightarrow {z_2} \mu_2$ in $E_2$, the value $\psi(z_1, z_2) \in \C(q,t)$ is defined. 
\end{enumerate}
\end{definition}

\begin{theorem}
\label{thm: c for tensor product}
Suppose that $E_1$ and $E_2$ are two excellent posets with respective edge functions $c_1(\lambda;x)$ and $c_2(\mu;x)$ satisfying \eqref{eq: monodromy}. Let $\psi_1$ and $\psi_2$ be rational functions satisfying 
\begin{equation}
\label{eq: psi1 psi2}
\frac{\psi_1(x,y)}{\psi_2(y,x)}=-\frac{(x-ty)(x-qy)(y-qtx)}{(y-tx)(y-qx)(x-qty)}
\end{equation}
for all values $x,y$ addable for $\lambda$ or $\mu$, for some $\lambda \in E_1, \mu \in E_2$. 
Assume moreover that $E_1$ and $E_2$ are in very general position with respect to both $\psi_1$ and $\psi_2$. Then the edge function 
\begin{equation}
\label{c product}
c\Bigl((\lambda,\mu);x\Bigr)=\begin{cases}
c_1(\lambda;x)M_{\psi_2;E_2}(\mu;x) & \text{; if}\ x\ \text{is addable for}\ \lambda,\\
c_2(\mu;x)M_{\psi_1;E_1}(\lambda;x) & \text{; if}\ x\ \text{is addable for}\ \mu\
\end{cases}
\end{equation}
satisfies \eqref{eq: monodromy} and defines a representation of $\Bqt^{\ext}$ corresponding to $E_1\times E_2$.
\end{theorem}

Before proving the theorem, we note that by our assumption of being in very general position, we can ensure that $x$ is neither a zero nor a pole for $M_{\psi_2}(\mu;u)$ when $x$ is addable for $\lambda$, and that $x$ is neither a zero nor a pole of $M_{\psi_1}(\lambda;u)$ when $x$ is addable for $\mu$. Moreover, if $E_1$ and $E_2$ are in general position, we can always choose $\psi_1$ and $\psi_2$ so that $E_1$ and $E_2$ are in very general position with respect to both $\psi_1$ and $\psi_2$, see Example \ref{ex: psi easy} below. 

\begin{proof}
To verify \eqref{eq: monodromy}, we need to consider four cases. 

\begin{itemize}
\setlength{\itemindent}{.1cm}
\item[Case 1:] if $x$ and $y$ are addable for $\lambda$. Then 
$$
\frac{c((\lambda,\mu);x)c((\lambda\cup x,\mu);y)}{c((\lambda,\mu);y)c((\lambda\cup y,\mu);x)}=
\frac{c_1(\lambda;x)c_1(\lambda\cup x;y)}{c_1(\lambda;y)c_1(\lambda\cup y;x)}\cdot
\frac{M_{\psi_2}(\mu;x)M_{\psi_2}(\mu;y)}{M_{\psi_2}(\mu;y)M_{\psi_2}(\mu;x)}.
$$
The first fraction equals 
$$
-\frac{(x-ty)(x-qy)(y-qtx)}{(y-tx)(y-qx)(x-qty)},
$$
and the second fraction equals 1.

\item[Case 2:] $x$ and $y$ are addable for $\mu$, this is similar to (1).
\item[Case 3:] $x$ is addable for $\lambda$ and $y$ is addable for $\mu$. Then
$$
\frac{c((\lambda,\mu);x)c((\lambda\cup x,\mu);y)}{c((\lambda,\mu);y)c((\lambda,\mu\cup y);x)}=
\frac{c_1(\lambda;x)c_2(\mu;y)}{c_2(\mu;y)c_1(\lambda;x)}
\cdot \frac{M_{\psi_2}(\mu;x)M_{\psi_1}(\lambda\cup x;y)}{M_{\psi_1}(\lambda;y)M_{\psi_2}(\mu\cup y;x)}.
$$
The first fraction equals 1, and the second equals 
$$
\frac{\psi_1(x,y)}{\psi_2(y,x)}=
-\frac{(x-ty)(x-qy)(y-qtx)}{(y-tx)(y-qx)(x-qty)},
$$
by our assumption \eqref{eq: psi1 psi2} on $\psi_1$ and $\psi_2$. 
\item[Case 4:] $x$ is addable for $\mu$ and $y$ is addable for $\lambda$. This is similar to (3).
\end{itemize}

\end{proof}

\begin{example}
\label{ex: psi easy}
Choose 
$
\psi_1=\psi_2=\psi(y,u)=\frac{(1-yu^{-1})(1-qtyu^{-1})}{(1-qyu^{-1})(1-tyu^{-1})}.
$
Then, if two posets $E_1$ and $E_2$ are in general position, they will automatically be in very general position with respect to $\psi$.
\end{example}

\begin{remark}
It follows from \cite[Lemma 2.7]{GV} and the general position assumption that, as $\AH_{k}$-modules:
\[
V_k(E_1 \times E_2) = \bigoplus_{m+n = k}\operatorname{Ind}_{\AH_{m} \otimes \AH_{n}}^{\AH_{m+n}}V_m(E_1) \otimes V_n(E_2).
\]
\end{remark}

\begin{remark}\label{rmk: tensor morphisms}
Assume $f_1: V(E_1) \to V(E_1')$ and $f_2: V(E_2) \to V(E_2')$ are $\Bqt^{\ext}$-homomorphisms. We assume that the edge functions on $E_1$ and $E_1'$ are compatible, as well as the edge functions on $E_2$ and $E_2'$, cf. Lemma \ref{lem: compatible edge}. Fixing $\psi_1, \psi_2$ satisfying \eqref{eq: psi1 psi2} with respect to the pairs $(E_1, E_2)$ and $(E_1', E_2')$ we can also assume that the $M$ functions on $E_1$ and $E_1'$ are compatible, in the sense that
\[
M_{\psi_{i}}(\lambda; u) = M_{\psi_{i}}(F_1(\lambda); u)
\]
for every $\lambda \in E_1$, cf. Theorem \ref{thm: hom}. Similarly, we assume that the $M$ functions on $E_2$ and $E_2'$ are compatible. Then, we have a $\Bqt^{\ext}$-homomorphism
\[
f_1 \boxtimes f_2: V(E_1 \times E_2) \to V(E_1' \times E_2')
\]
explicitly given, in the notation of Proposition \ref{prop:hom} by
\[
(f_1 \boxtimes f_2)\left[(\lambda^{(1)}, \lambda^{(2)})\; ;\; \sigma(\underline{w}, \underline{z})\right] = \alpha_{\lambda^{(1)}}\alpha_{\lambda^{(2)}}\Bigl [\bigl(F_1(\lambda^{(1)}), F_2(\lambda^{(2)})\bigr)\;;\; \sigma(\underline{w}, \underline{z})\Bigr]
\]
whenever both $F_1(\lambda^{(1)} \cup \underline{w}) \neq {0}$ and $ F_2(\lambda^{(2)}\cup \underline{z}) \neq {0}$; otherwise $(f_1 \boxtimes f_2)\left[(\lambda^{(1)}, \lambda^{(2)}); \sigma(\underline{w}, \underline{z})\right] = 0$.
\end{remark}
 
We can reformulate the above results more abstractly. 

\begin{definition}
Given two functions $\psi_1(x,y)$ and $\psi_2(x,y)$ satisfying \eqref{eq: psi1 psi2}, let  
$\Cat_{\psi_1,\psi_2}$ be the category with,
\begin{enumerate}[leftmargin=*]
    \item[(1)] \underline{Objects} given by quadruples
$\left(E,c,M_{\psi_1},M_{\psi_2}\right)$ where:
\begin{itemize}
\item $E$ is an excellent weighted poset.
\item $c=c(\lambda;x)$ is a function on the edges of $E$ satisfying \eqref{eq: monodromy}.
\item $M_{\psi_i}=M_{\psi_i}(\lambda;u)$ are functions satisfying \eqref{eq: M psi}.
\end{itemize}
\item[(2)] \underline{Morphisms} $F \in \Hom_{\mathcal{C}} \left( (E, c, M_{\psi_1}, M_{\psi_2}),  (E', c', M'_{\psi_1}, M'_{\psi_2})\right)$ given by functions 
\[F: E_{\mathbf{0}} \to E'_{\mathbf{0}}\]
 satisfying conditions (1)--(5) of Theorem \ref{thm: hom} in addition to:
\begin{itemize}
    \item If $\mu = \lambda \cup x$ and $F(\mu) = F(\lambda) \cup x$, then the respective edge functions agree, 
    $$c(\lambda; x)=c(F(\lambda);x).$$ 
    \item For $i = 1, 2$ and any $\lambda \in E \setminus F^{-1}(\mathbf{0})$, the corresponding rational functions are equal,
     $$M_{\psi_{i}}(\lambda; x)=M'_{\psi_{i}}(F(\lambda); x).$$ 
\end{itemize}
\end{enumerate}
\end{definition}

It is clear that if $F: E_{\mathbf{0}} \to E'_{\mathbf{0}}$ and $F': E'_{\mathbf{0}} \to E''_{\mathbf{0}}$ are morphisms in $\mathcal{C}_{\psi_1, \psi_2}$ then $F' \circ F: E_{\mathbf{0}} \to E''_{\mathbf{0}}$ is also a morphism in $\mathcal{C}_{\psi_1, \psi_2}$, so the category is well defined. 

Note that the pair $(E,c)$ completely determines a representation $V(E,c)$ of $\Bqt^{\ext}$. Moreover, by Theorem \ref{thm: hom}, any map $F: (E, c, M_{\psi_1}, M_{\psi_2}) \to (E', c', M'_{\psi_1}, M'_{\psi_2})$ completely determines a $\Bqt^{\ext}$-homomorphism $V(F): V(E, c) \to V(E', c')$. Thus, there exists a functor
\begin{equation}\label{eq:Vfunctor}
V:\Cat_{\psi_1,\psi_2}\to \Bqt^{ext}-\mathrm{mod}\text{,    such that    } V\left(E,c,M_{\psi_1},M_{\psi_2}\right)=V(E,c).
\end{equation} 

On the one hand, this functor is faithful since if $F: E_{\mathbf{0}} \to E'_{\mathbf{0}}$ and $F': E_{\mathbf{0}} \to E'_{\mathbf{0}}$ are different maps, then there exists $\lambda \in E$ such that $F(\lambda) \neq F'(\lambda)$, so that $V(F)([\lambda]) \neq V(F')([\lambda])$. On the other hand, since there is no $\C(q,t)$-linear structure on $\mathcal{C}_{\psi_1, \psi_2}$, the functor $V$ is never full.
\medskip

The category $\Cat_{\psi_1,\psi_2}$ is \newword{generically monoidal}, meaning that the tensor product of two objects will be defined only when they are in very general position. 

\begin{theorem}
a) The following endows $\Cat_{\psi_1,\psi_2}$ with the structure of a generically monoidal category: if $E_1$ and $E_2$ are in very general position with respect to both rational functions $\psi_1$ and $\psi_2$, let
\begin{equation}\label{eq:generic-product}
\left(E_1,c_1,M_{\psi_1;E_1},M_{\psi_2;E_1}\right)\boxtimes \left(E_2,c_2,M_{\psi_1;E_2},M_{\psi_2;E_2})=(E,c,M_{\psi_1},M_{\psi_2}\right)
\end{equation}
where:
\begin{itemize}
\item $E=E_1\times E_2$
\item $c$ is given by \eqref{c product}
\item $M_{\psi_i}((\lambda,\mu),u)=M_{\psi_i;E_1}(\lambda,u)M_{\psi_i;E_2}(\mu,u)$.
\end{itemize}

If
\[
F_i: (E_i, c_i, M_{\psi_1; E_i}, M_{\psi_2, E_i}) \to (E'_1, c'_i, M_{\psi_1; E'_i}, M_{\psi_2, E'_i})
\]
are morphisms for $i = 1, 2$ the map $F_1 \boxtimes F_2$ is defined by
\[
(F_1 \boxtimes F_2)\left(\lambda^{(1)}, \lambda^{(2)}\right) = \begin{cases} \left(F_1(\lambda^{(1)}), F_2(\lambda^{(2)})\right) & ; \; \text{if both} \; F_1(\lambda^{(1)}), F_2(\lambda^{(2)}) \neq \mathbf{0}, \\
\mathbf{0} & ;  \;\text{else}. \end{cases}
\]

b) Given any three weighted posets $E_1, E_2$ and $E_3 \in \Cat_{\psi_1,\psi_2}$ in pairwise very general position with respect to both $\psi_1$ and $\psi_2$, then $E_1 \times E_2$ is in very general position with $E_3$ (with respect to $\psi_i$) if and only if $E_1$ is in very general position with $E_2 \times E_3$ (with respect to $\psi_i$). In particular, the generic monoidal product \eqref{eq:generic-product} on $\Cat_{\psi_1, \psi_2}$ is associative with trivial associator. 
\smallskip

c) Consider the weighted poset $I = \{\bullet\}$ with one element and $p_m(\bullet) = 0$. Then, $I$ is in very general position with respect to any other excellent weighted poset. Moreover, $\Cat_{\psi_1, \psi_2}$ has a monoidal unit given by $(I, c, M_{\psi_1}, M_{\psi_2})$ where $c$ is the empty function and $M_{\psi_i}(\bullet, u) \equiv 1$. 
\smallskip

d) There is an equivalence of categories $\beta: \Cat_{\psi_1,\psi_2}\to \Cat_{\psi_2,\psi_1}^{\rev}$
given by 
$$
\beta(E,c,M_{\psi_1},M_{\psi_2})=(E,c,M_{\psi_2},M_{\psi_1}).
$$
where $\Cat_{\psi_2,\psi_1}^{\rev}$ denotes $\Cat_{\psi_2,\psi_1}$ with reversed generic monoidal product.
\smallskip

e) If $\psi_1=\psi_2=\psi$, then $\Cat_{\psi,\psi}$ admits a symmetric generic monoidal subcategory with objects $(E,c,M_{\psi},M_{\psi})$.
\end{theorem}

\begin{proof}
a) The product $\boxtimes$ is well defined by Theorem \ref{thm: c for tensor product} and Lemma \ref{lem: M psi}(c). For the product of morphisms, see Remark \ref{rmk: tensor morphisms}.

(b) First, note that $(E_1 \times E_2)$ is in very general position with $E_3$ if and only if $E_1, E_2, E_3$ are pairwise in very general position, which in turn is equivalent to $E_1$ being in general position with respect to $E_2 \times E_3$ (all with respect to $\psi_i$). Now, the isomorphism of weighted posets
$$
(E_1\times E_2)\times E_3\simeq E_1\times (E_2\times E_3)
$$
is straightforward.

Furthermore, for 
$(E_1\times E_2)\times E_3$ we get
$$
c\bigl((\lambda,\mu,\nu);x\bigr)=\begin{cases}
c_1(\lambda;x)M_{\psi_2}(\mu;x)M_{\psi_2}(\nu;x)\ \text{if}\ x\ \text{is\ addable\ for}\ \lambda\\
c_2(\mu;x)M_{\psi_1}(\lambda;x)M_{\psi_2}(\nu;x)\ \text{if}\ x\ \text{is\ addable\ for}\ \mu\\
c_3(\nu;x)M_{\psi_1;E_1\times E_2}((\lambda,\mu;x))\ \text{if}\ x\ \text{is\ addable\ for}\ \nu.
\end{cases}
$$
By  construction, we get
$$
c_3(\nu;x)M_{\psi_1;E_1\times E_2}((\lambda,\mu);x)=c_3(\nu;x)M_{\psi_1}(\lambda;x)M_{\psi_1}(\mu;x)
$$
and we get the same expression for $
c((\lambda,\mu,\nu),x)$ in $E_1\times (E_2\times E_3)$.
Finally,
$$
M_{\psi_i;E_1\times E_2\times E_3}((\lambda,\mu,\nu);u)=M_{\psi_i;E_1}(u)M_{\psi_i;E_2}(u)M_{\psi_i;E_3}(u).
$$

c) Vacuously, the poset $I$ is in very general position with respect to any other poset $E$. Moreover, since $p_m(\bullet) = 0$ for every $m$, it is clear that $I \times E \simeq E \simeq E \times I$ as weighted posets.  
If $x$ is an addable box for $\lambda \in E$ then
$$
c((\lambda,\bullet);x)=c(\lambda;x)M_{\psi_2;I}(\bullet;u)=c(\lambda;x).
$$
Finally, 
$$
M_{\psi_i;E\times I}((\lambda,\bullet;u)=M_{\psi_i;E}(\lambda;u)
M_{\psi_i;I}(\bullet;u)=M_{\psi_i;E}(\lambda;u).
$$
The proof of the isomorphism $I\times E\simeq E$ is similar.

d) First observe that swapping $x$ and $y$ in \eqref{eq: psi1 psi2} yields
$$
\frac{\psi_1(y,x)}{\psi_2(x,y)}=\left[\frac{\psi_1(x,y)}{\psi_2(y,x)}\right]^{-1},\ \mathrm{so}\ 
\frac{\psi_2(x,y)}{\psi_1(y,x)}=\frac{\psi_1(x,y)}{\psi_2(y,x)}
$$
and $(\psi_2,\psi_1)$ satisfy \eqref{eq: psi1 psi2}. Now we simply use the fact that \eqref{c product} is symmetric when exchanging $E_1$ with $E_2$ and $\psi_1$ with $\psi_2$.

e) When $\psi_1=\psi_2$ and $M_{\psi_1}=M_{\psi_2}$, then \eqref{c product} is symmetric in $E_1$ and $E_2$.
\end{proof}

\subsection{The Boolean Poset}

As a concrete example, consider the boolean poset $E_n(a_1,\ldots,a_n)$. The elements of $E_n(a_1,\ldots,a_n)$ are subsets $S\subset \{1,\ldots,n\}$ with weighting $p_m(S)=\sum_{j\in S}a_j^{m}$. The covering relation is given by $S\cup a_i=S\cup \{i\}$ provided that $i\notin S$, and the grading $|S|$ is given by the number of elements in $S$. It is easy to see that $E_n(a_1,\ldots,a_n)$ is excellent if and only if $a_i/a_j\notin \{1,q,t,qt\}$ for all $i\neq j$. 

It is an interesting question to determine the coefficients $c(S;a_i)$ explicitly. For this, we observe that 
$$
E_n(a_1,\ldots,a_n)=E_1(a_1)\times E_1(a_2)\times \ldots\times E_1(a_n).
$$
The poset $E_1(a)$ has two elements $\emptyset$ and $\{1\}$ with 
$$
\Delta_{p_m}[\emptyset]=0,\ \Delta_{p_m}(a)[\{1\}]=a^m[\{1\}].
$$
Furthermore, we can define $c(\emptyset;a)=1$ and
$$
M(\emptyset;u)=1,\ M(\{1\},u)=\psi(a,u)=\frac{(1-au^{-1})(1-qtau^{-1})}{(1-qau^{-1})(1-tau^{-1})}.
$$
Now it is easy to check by induction in $n$ using \eqref{c product} that
$$
c(S;x)=\prod_{j\in S}\psi(a_j,x).
$$
This implies the following result.

\begin{proposition}
There is a representation of $\Bqt$ with basis $[S;i_k,\ldots,i_1]=[S;\underline{i}]$ where $S\subset \{1,\ldots,n\}$, $i_{\alpha}\neq i_{\beta}$ and $i_{\alpha}\notin S$. The action of the generators is given by   
\begin{align*}
\Delta_{p_m}[S;\underline{i}]& =\sum_{j\in S}a_j^m+\sum_{\alpha=1}^{k}a_{i_{\alpha}}^m,\\
z_{\alpha}[S;\underline{i}]&=a_{i_{\alpha}}[S;\underline{i}],
\\
T_{\alpha}[S;\underline{i}]&=\frac{(q-1)a_{i_{\alpha+1}}}{a_{i_{\alpha}}-a_{i_{\alpha+1}}}[S;\underline{i}]+
\frac{a_{i_{\alpha}}-qa_{i_{\alpha+1}}}{a_{i_{\alpha}}-a_{i_{\alpha+1}}}[S;s_\alpha\underline{i}]
\\
d_{-}[S;\underline{i}]&=[S\cup i_k;i_{k-1},\ldots,i_{1}].
\\
d_{+}[S;\underline{i}]&=\sum_{\ell}\prod_{j\in S}\psi(a_j,a_\ell)  \prod_{\alpha=1}^{k}\psi(a_{i_{\alpha}},a_\ell)\frac{a_\ell-ta_{i_{\alpha}}}{a_\ell-qta_{i_{\alpha}}}[S;\underline{i},\ell].
\end{align*}
\end{proposition}

\subsection{The Parabolic Gieseker Space as a Tensor Product}

In Section \ref{sec: Gieseker} (see \eqref{eq:GiesekerRepn}) we constructed a representation $$V^{(r)}=\bigoplus_{k,n}K^{\widetilde{T}}_{loc}(\M^{par}(r;n,n+k))$$ of $\Bqt^{\ext}$ using the geometry of parabolic Gieseker moduli spaces of rank $r$ sheaves. It turns out that this representation fits well into our tensor product construction. Recall from Section \ref{subsec: Gieseker} that $a_1, \dots, a_{r}$ denote the coordinate characters of the maximal torus $(\C^{*})^{r}$ of $GL(r)$. 

\begin{theorem}\label{thm: gieseker as tensor power}
Suppose $V$ is the polynomial representation. Then,
$$
V^{(r)}=V(a_1)\boxtimes V(a_2)\boxtimes \cdots \boxtimes V(a_r).
$$
\end{theorem}

\begin{proof}
First, recall that by Theorem \ref{thm: polynomial is calibrated} $V$ is calibrated with underlying weighted poset of partitions $\mathcal{P}$. The good chains correspond to standard tableaux in horizontal strips between two partitions. Furthermore, we have
$$
c(\lambda;x)=-\Lambda\left(-x^{-1}+(1-q)(1-t)B_{\lambda}x^{-1}+1\right)
$$
so we can write
$$
M_{\psi}(\lambda;u)=\Lambda\left(-u^{-1}+(1-q)(1-t)B_{\lambda}u^{-1}\right)=\frac{1}{1-u^{-1}}\prod_{\sq\in \lambda}\frac{(1-\sq u^{-1})(1-qt\sq u^{-1})}{(1-q\sq u^{-1})(1-t\sq u^{-1})}.
$$
Here $\psi$ is as in Example \ref{ex: psi easy}.

The poset $\mathcal{P}(a_1)\times \cdots\times \mathcal{P}(a_r)$ consists of $r$-tuples of partitions. All good chains in the product are given by shuffles \eqref{eq: shuffle chains} of good chains in $\mathcal{P}$, which can be interpreted as standard tableaux in the union of $r$ horizontal strips. All coefficients of $d_{-}$ are equal to 1, so it remains to check the coefficients of $d_{+}$, which satisfy equation \eqref{eq: c from edges new} by Lemma \ref{lem:d+ coeff}.

Let 
$\lambda_{\bullet}=(\lambda_{\alpha})_{\alpha=1}^{r}$.
Then, the coefficients of $d_{+}:V^{(r)}_0\to V^{(r)}_{1}$ equal
$$
c(\lambda_{\bullet},x)=-\Lambda(-Ax^{-1}+(1-q)(1-t)B_{\lambda_{\bullet}}x^{-1}+1).
$$
If $x$ is addable for some $\lambda_{\alpha}$, then
\begin{align*}
c(\lambda_{\bullet},x)&=-\Lambda(-a_{\alpha}x^{-1}+(1-q)(1-t)B_{\lambda_{\alpha}}x^{-1}+1)
\prod_{\beta\neq \alpha}\Lambda(-a_{\beta}x^{-1}+(1-q)(1-t)B_{\lambda_{\beta}}x^{-1})
\\
&=c(\lambda_{\alpha};a_{\alpha}^{-1}x)\prod_{\beta\neq \alpha}M_{\psi;E(a_\beta)}(\lambda_{\beta};x)
\end{align*}
in agreement with \eqref{c product}. Hence,
$
M_{\psi}(\lambda_{\bullet};u)=\prod_{\alpha}M_{\psi;E(a_{\alpha})}(\lambda_{\alpha};u),
$
from which the result follows.
\end{proof}
 
\subsection{More Examples of Representations Arising from Tensor Products}
Using the tensor product structure defined in the previous section many more examples of representations of $\Bqt^{\ext}$ can be constructed.

\begin{example}
Let $E_1(\underline{b})$ be the trivial weighted poset from Section \ref{sec: trivial} with one element $\bullet$ and $p_m(\bullet)=b_m$ for some scalars $b_m$.

Let $E_2$ be an arbitrary excellent weighted poset. Then
$E=E_1(a_1)\times E_2(a_2)$ is the same poset as $E_2$ with 
$$
p_{m;E}(\lambda)=a_2^mp_m(\lambda)+a_1^mb_m.
$$
By Remark \ref{rem: shift rep} the resulting representation of $\Bqt$ is isomorphic to $E_2(a_2)$, but the action of $\Delta_{p_m}$ is shifted by $a_1^mb_m$.
\end{example}
 
Let $E$ be the weighted poset of nonnegative integers with $p_m(i)=1^m+q^m+\ldots+q^{(i-1)m}$, as in Section \ref{sec: linear posets}. 
For $\psi_1=\psi_2=\psi$ as in Example \ref{ex: psi easy}, we can use
$$
M_{\psi}(i;u)=\prod_{j=0}^{i-1}\psi(q^j,u)=\prod_{j=0}^{i-1}\frac{(1-q^{j}u^{-1})(1-tq^{j+1}u^{-1})}{(1-q^{j+1}u^{-1})(1-tq^{j}u^{-1})}=
\frac{(1-u^{-1})(1-tq^iu^{-1})}{(1-q^iu^{-1})(1-tu^{-1})}.
$$

The poset $E\times E(t)$ is {\bf not} excellent since, for example, both $1$ and $t$ are addable for $(0,0)$ and $t$ is addable for $(1,0)$. Also, $E$ and $E(t)$ are not in general position.
However, the set of partitions with at most 2 rows is an excellent sub-poset of $E\times E(t)$. More generally, we get the following result.

\begin{lemma}
The set $\mathcal{P}\setminus I_N$ of partitions with at most $N$ parts (see Lemma \ref{lem: restricted partitions}) 
is an excellent subposet of $E(1)\times E(t)\times \cdots\times E(t^{N-1})$. The edge functions \eqref{eq: c for partitions} agree with the edge functions from Theorem \ref{thm: c for tensor product}, with an appropriate choice of $M_{\psi}$.
\end{lemma}

\begin{proof}
We prove it by induction in $N$. Let $\mathcal{P}\setminus I_N$ be the poset of partitions with at most $N$ parts, so that clearly $\mathcal{P}\setminus I_1=E$ and the base case is done.  Let $\lambda=(\lambda_1,\ldots,\lambda_{N})$ and
$\mu=(\lambda_1,\ldots,\lambda_{N-1})$ so that $\mathcal{P}\setminus I_N$ can be realized as a subposet of $(\mathcal{P}\setminus I_{N-1})\times E(t^{N-1})$ by identifying $\lambda$ with $(\mu,\lambda_N(t^{N-1}))$.

Suppose that $x$ is an addable box for $\lambda$ and $\lambda\cup x$ has at most $N$ rows. Define 
\begin{align*}
&M_{\psi}(\lambda_N(t^{N-1});u)=\prod_{j=0}^{\lambda_{N}-1}\psi(q^jt^{N-1},u) \text{ 
\qquad  and }
\\
&M_{\psi}(\mu;u)=\frac{(1-t^{N-1}u^{-1})}{(1-u^{-1})}\prod_{\sq\in \mu}\frac{(1-\sq u^{-1})(1-qt \sq u^{-1})}{(1-q\sq u^{-1})(1-t \sq u^{-1})}.
\end{align*}
There are two cases:

a) Suppose $x$ is in the first $N-1$ rows so that $x$ is addable for $\mu$. Then, $q^jt^{N-1}x^{-1}$ is never equal to $1,q^{-1},t^{-1}$ or $(qt)^{-1}$, so $M_{\psi}(\lambda_N(t^{N-1});x)$ is well defined and nonzero. Hence,
\begin{align*}
c(\lambda;x)&=-\Lambda\Bigl(-x^{-1}+(1-q)(1-t)B_{\lambda}x^{-1}+1\Bigr)
\\
&=-\Lambda\Bigl(-x^{-1}+(1-q)(1-t)B_{\mu}x^{-1}+1\Bigr) \times \Lambda\Bigl((1-q)(1-t)x^{-1}\sum_{j=0}^{\lambda_{N}-1}q^jt^{N-1}\Bigr)
\\
&=c(\mu;x)M_{\psi}(\lambda_N(t^{N-1});x).
\end{align*}

b) Suppose $x$ is in the $N$-th row. Then $x$ is addable for $\lambda_N(t^{N-1})$ and and there is a box in $\mu$ below $x$. We claim that $M_{\psi}(\mu;u)$ has a well defined and nonzero limit at $u=x$. Indeed, for $x\neq t^{N-1}$ there will be exactly one factor of the form $(1-qt\sq u^{-1})$ and exactly one factor of the form $(1-t\sq u^{-1})$ which vanish at $u=x$; for $x=t^{N-1}$ the factors $(1-t^{N-1}u^{-1})$ and $(1-t\sq u^{-1})$ vanish. Thus,
\begin{align*}
&c(\lambda;x)=-\Lambda(-x^{-1}+(1-q)(1-t)B_{\lambda}x^{-1}+1)
\\
&=-\Lambda\Bigl(-x^{-1}+t^{N-1}x^{-1}+(1-q)(1-t)B_{\mu}x^{-1}\Bigr)\Lambda\Bigl(-t^{N-1}x^{-1}+(1-q)(1-t)x^{-1}\sum_{j=0}^{\lambda_{N}-1}q^jt^{N-1}+1\Bigr)
\\
&=M_{\psi}(\mu;x)c(\lambda_N,t^{1-N}x).
\end{align*}
\end{proof}

\section{A Duality Functor}\label{sec: duality}

In this section, we define a duality functor on the category of calibrated $\Bqt^{\ext}$ representations.

\subsection{A $\Bqt$ Involution} We start by defining an involution on $\Bqt$. This involution will not be $\C(q,t)$-linear.

\begin{definition}
We denote by $\aut: \C(q,t) \to \C(q,t)$ the automorphism given by $$\aut(q) = q^{-1}, \aut(t) = t^{-1}.$$ 
\end{definition}

We will also need to extend our field $\C(q,t)$ of definition by adjoining a square root $q^{1/2}$ of $q$. 

\begin{lemma}\label{lem:involution}
The following defines a $\aut$-linear anti-involution $\adj: \Bqt \to \Bqt$ of the algebra $\Bqt$.
\begin{align*}
\adj(\e_{k}T_{i}\e_{k}) &= \e_{k}T_{k-i}^{-1}\e_{k} 
&
\adj(\e_{k-1}d_{-}\e_{k}) &= q^{-k/2}\e_{k}d_{+}\e_{k-1} 
\\
\adj(\e_{k}z_{i}\e_{k}) &= \e_{k}z_{k+1-i}\e_{k}
&
\adj(\e_{k+1}d_{+}\e_{k}) &= q^{-(k+1)/2}\e_{k}d_{-}\e_{k+1}.
\end{align*}
\end{lemma}

\begin{remark}
The $\Bqt$ algebra can equivalently be defined as a category with objects given by idempotents $\e_{k}$, $k \in \Z_{\geq 0}$ and morphisms $T_i, z_i, d_\pm$ between them. In this setting, the anti-involution $\adj$ in Lemma \ref{lem:involution} may be thought of as a contravariant functor $\adj: \Bqt \to \Bqt$ that is the identity on objects. 
\end{remark}

\begin{proof}
    We need to verify that $\adj$ preserves all the defining $\Bqt$ relations in equations \eqref{eq:hecke relns}--\eqref{eq:qphi}. From the definitions, it is clear that relations \eqref{eq:hecke relns}--\eqref{eq:affine Hecke relns} are maintained by $\adj$.
    
    We check \eqref{eq:T d-}, starting with $d_{-}^{2}T_{k-1} = d_{-}^{2}$. Applying $\adj$ to both sides we obtain
    \begin{align*}
    \adj(d_{-}^{2}T_{k-1}) = \adj(T_{k-1})\adj(d_{-}^{2}) = q^{-(2k+1)/2}T^{-1}_{1}d_{+}^{2}
    \text{ \quad versus \quad} 
    \adj(d_{-}^{2}) = q^{-(2k+1)/2}d_{+}^{2}
    \end{align*}
    so the result follows from \eqref{eq:T d+}. 
    To verify relation $d_{-}T_{i} = T_{i}d_{-}$ for $1 \leq i \leq k-2$ we again apply $\adj$ to both sides and obtain 
 \begin{align*}
 \adj(d_{-}T_{i})=q^{-k/2}T_{k-i}^{-1}d_{+} \overset{\eqref{eq:T d+}}{=} q^{-k/2}d_{+}T_{k-1-i}^{-1}= \adj(T_{i}d_{-})
 \end{align*} 
Checking that $\adj$ preserves \eqref{eq:T d+} and \eqref{eq: d z} is similar.

    It remains to see that $\adj$ preserves \eqref{eq:phi} and \eqref{eq:qphi}. For the sake of brevity, let us denote by $\widetilde{\varphi} = d_{+}d_{-} - d_{-}d_{+} = (q-1)\varphi$ and note that the relations in \eqref{eq:phi} are equivalent to their counterparts with $\varphi$ replaced by $\widetilde{\varphi}$. We compute $\adj(\widetilde{\varphi})$:
    \[
    \adj(\widetilde{\varphi}) =   \adj(d_{+}d_{-} - d_{-}d_{+})
    =  q^{-k/2}q^{-k/2}d_{+}d_{-} - q^{-(k+1)/2}q^{-(k+1)/2}d_{-}d_{+} 
    =  q^{-k-1}(qd_{+}d_{-} - d_{-}d_{+}). 
    \]
   Similarly,
    \[
    \begin{array}{rl}
    \adj(qd_{+}d_{-} - d_{-}d_{+}) = & q^{-1}\adj(d_{+}d_{-}) - \adj(d_{-}d_{+}) =  q^{-k-1}\widetilde{\varphi}.
    \end{array}
    \]
    Thus, equating both sides we obtain
    \[
    \adj(z_{1}(qd_{+}d_{-} - d_{-}d_{+})) = q^{-k-1}(d_{+}d_{-} - d_{-}d_{+})z_{k} \overset{\eqref{eq:qphi}}{=} q^{-1}t^{-1}q^{-k-1}z_{1}(qd_{+}d_{-} - d_{-}d_{+}) =  \adj(qt\widetilde{\varphi}z_{k})
    \]
Lastly, we verify that $\adj$ preserves \eqref{eq:phi}. Consider first equation $q\widetilde{\varphi} d_{-} = d_{-}\widetilde{\varphi} T_{k-1}$. Applying $\adj$ to the left-hand side we have,
\begin{equation}\label{eqn:annoy}
\begin{array}{rl}
\adj({q\widetilde{\varphi} d_{-}}) = & q^{-1}q^{-k/2}q^{-k}d_{+}(qd_{+}d_{-} - d_{-}d_{+})  
=   q^{-(3k+2)/2}(qd_{+}\widetilde{\varphi} + (q-1)d_{+}d_{-}d_{+})
\end{array}
\end{equation}
Conversely, applying $\adj$ to the right-hand side of this equation we obtain
\begin{align}
\adj{(d_{-}\widetilde{\varphi}T_{k-1})} = & q^{-k-1}q^{-k/2}T_{1}^{-1}(qd_{+}d_{-} - d_{-}d_{+})d_{+} \nonumber \\ 
= & q^{-(3k+2)/2}\left(qT_{1}^{-1}\widetilde{\varphi} d_{+} + (q-1)T_{1}^{-1}d_{-}d_{+}d_{+}\right)\nonumber \\
= &  q^{-(3k+2)/2}\left(T_{1}\widetilde{\varphi} d_{+} + (q-1)\widetilde{\varphi} d_{+} + (q-1)T_{1}^{-1}d_{-}d_{+}d_{+})\right) \nonumber\\
= & q^{-(3k+2)/2}\left(T_{1}\widetilde{\varphi} d_{+} + (q-1)(\widetilde{\varphi} d_{+} + d_{-}d_{+}d_{+})\right)\nonumber \\
= & q^{-(3k+2)/2}\left(T_{1}\widetilde{\varphi} d_{+} + (q-1)d_{+}d_{-}d_{+}\right)\label{eqn:more annoy}
\end{align}
\noindent where we used that $qT_{1}^{-1} = T_{1} + q - 1$ followed by $T_{1}^{-1}d_{-} = d_{-}T_{1}^{-1}$ and $T_{1}^{-1}d_{+}d_{+} = d_{+}d_{+}$. Finally, using \eqref{eq:phi} it follows that \eqref{eqn:annoy} and \eqref{eqn:more annoy} coincide. 

Instead of repeating the computation for the last relation $T_{1}\widetilde{\varphi} d_{+} = qd_{+}\widetilde{\varphi}$, let us observe first that $\adj$ is indeed an involution on the generators. This is easy to see for $z_{i}$ and $T_{i}$. Now,
\[
\adj^2(d_{-}) = \adj(q^{-k/2}d_{+}) = q^{k/2}q^{-k/2}d_{-} = d_{-},
\]

\noindent and similarly for $\adj^2(d_{+}) = d_{+}$. Thus, from \eqref{eqn:annoy} and \eqref{eqn:more annoy} we can see that $\adj$ essentially interchanges both parts of \eqref{eq:phi}. The result follows.
\end{proof}

\begin{lemma}
The anti-involution $\adj$ on $\Bqt$ extends to a $\aut$-linear anti-involution on $\Bqt^{\ext}$ by setting:
$$
\adj(\Delta_{p_m})=-\Delta_{p_m}+z_1^{m}+\ldots+z_k^m.
$$
\end{lemma}

\begin{proof}
Since the $\Delta_{p_m}$ operators commute with all $z_i$ and amongst themselves, the operators $\adj(\Delta_{p_m})$ commute amongst themselves as well. Furthermore, $z_1^m+\ldots+z_k^m$ is central in the affine Hecke algebra, so $\adj(\Delta_{p_m})$ commutes with $T_i$ and $z_i$ as well. 

Checking the commutation relations with $d_+$ and $d_-$, yields
\begin{align*}
\adj(\Delta_{p_m}d_{-})&=q^{-k/2}d_{+}(-\Delta_{p_m}+z_1^{m}+\ldots+z_{k-1}^m)
\\
&=-q^{-k/2}d_{+}\Delta_{p_m}+q^{-k/2}(z_2^{m}+\ldots+z_k^{m})d_+
\\
&=-q^{-k/2}\Delta_{p_m}d_{+}+q^{-k/2}z_1^{m}d_{+}+q^{-k/2}(z_2^{m}+\ldots+z_k^{m})d_+=\adj(d_{-}\Delta_{p_m}).
\end{align*}
Similarly, 
\begin{align*}
\adj(\Delta_{p_m}d_{+})&=q^{-(k+1)/2}d_{-}(-\Delta_{p_m}+z_1^{m}+\ldots+z_{k+1}^m)
\\
&=q^{-(k+1)/2}(-\Delta_{p_m}+z_1^m+\ldots+z_k^m)d_{-}+q^{-(k+1)/2}d_{-}z_{k+1}^m=\adj(d_+\Delta_{p_m}+z_1^{m}d_+).
\end{align*}
Finally, to see $\adj$ is an involution on the generators, we note that since 
$
\adj(z_1^m+\ldots+z_k^m)=z_1^m+\ldots+z_k^m,
$
then $\adj^2(\Delta_{p_m})=\Delta_{p_m}$.
\end{proof}

\subsection{The Duality Functor} The anti-involution $\adj$ allows us to define a duality functor on the category of calibrated $\Bqt^{\ext}$ representations. For a $\C(q,t)$-vector space $V$, denote by $V^{*}$ its $\aut$-dual; that is, $V^{*}$ is the set of $\aut$-linear maps $V \to \C(q,t)$. Notice that if $V$ is a $\Bqt^{\ext}$-representation then $V^{*}$ becomes a $\Bqt^{\ext}$-representation by setting, for each $b \in \Bqt^{\ext}$,
\[
(b.f)(v) = f(\adj(b)v). 
\]
We will restrict to representations that decompose as a direct sums,
\begin{equation}\label{eq:rep direct sum}
V = \bigoplus_{k \geq 0}V_{k}, \qquad V_{k} := \e_{k}V.
\end{equation}
Abusing the notation, we will denote:
\begin{equation}\label{eq: duality bad}
V^{*} := \bigoplus_{k \geq 0}V_{k}^{*}
\end{equation}
which is naturally a $\Bqt^{\ext}$-representation since $\adj(\e_{k}) = \e_{k}$ for all $k$. Note, however, that even when restricting to the class of representations in \eqref{eq:rep direct sum},  the duality functor mapping to \eqref{eq: duality bad} need not be an involution (i.e. $(V^*)^*$ need not equal $V$) since the spaces $V_{k}$ are, in general, infinite-dimensional $\C(q,t)$-vector spaces. To overcome this problem, we will restrict to the class of \emph{calibrated} $\Bqt^{\ext}$-representations (see Definition \ref{def: calibrated}). Thus, from now on, $V$ will denote a calibrated $\Bqt^{\ext}$-representation.

Abusing the notation again, we will denote 
\begin{equation}\label{eq: duality good}
V^{*} := \bigoplus_{k \geq 0}(V_{k}^{*})_{l.f.}
\end{equation}
where $(V_{k}^{*})_{l.f.}$ is the set of vectors where the action of the subalgebra generated by $z_1, \dots, z_k$ and Delta operators is locally finite. It is easy to see that if $\{v_{\lambda}\}$ is a basis of $V_{k}$ consisting of simultaneous eigenvectors for $z_1, \dots, z_k$ and the Delta operators $\Delta_{p_{i}}$, then the $\aut$-dual basis $\{v_{\lambda}^{*}\}$ is a basis of $(V_{k}^{*})_{l.f.}$. In particular, the functor sending $V \mapsto V^*$ from \eqref{eq: duality good} is an involution on the category of calibrated representations.

\subsection{Dual Poset Representations} Now, let $E$ be an excellent weighted poset (see Def. \ref{def: excellent new}) with associated representation $V = V(E)$. Let us recall that $V_{k}(E)$ has a basis given by the set of good chains
\[
V_{k} = \bigoplus \C(q,t)[\lambda;\underline{w}]
\quad
\mathrm{so\ that}
\quad
(V_{k}^{*})_{l.f.} = \bigoplus \C(q,t)[\lambda; \underline{w}]^{*}. 
\]
We can explicitly describe the action of $\Bqt$ on $V(E)^{*}$.

\begin{proposition}
\label{prop: dual operators}
The space $V(E)^{*}$ is a $\Bqt$-module with action given by: 
 \begin{align*}
 z_{i}[\lambda;\underline{w}]^*& = \aut(w_{k-i+1})[\lambda; \underline{w}]^*,
\\
T_{i}[\lambda; \underline{w}]^* &= \begin{cases}  \frac{(q-1)\aut(w_{k-i})}{\aut(w_{k-i+1}) - \aut(w_{k-i})}[\lambda; \underline{w}]^{*} + \frac{q\aut(w_{k-i+1}) - \aut(w_{k-i})}{\aut(w_{k-i+1}) - \aut(w_{k-i})}[\lambda;s_{k-i}(\underline{w})]^* &; \; \text{if} \; s_{k-i} \; \text{is admissible}. \\ \frac{(q-1)\aut(w_{k-i})}{\aut(w_{k-i+1}) - \aut(w_{k-i})}[\lambda; \underline{w}]^{*} &; \; \text{else} \end{cases} 
\\ 
d_{+}[\lambda; \underline{w}]^{*} &= q^{(k+1)/2}\sum_{\lambda = \mu \cup x}[\mu; x,\underline{w}]^{*} 
\\
d_{-}[\lambda; \underline{w}]^{*} &=  q^{k/2}\aut\Bigl(c(\lambda; \underline{w})\Bigr)[\lambda; w_k, \dots, w_2]^{*}.
 \end{align*}

\end{proposition}
\begin{proof}
The formula for $z_{i}$ is obvious. The formula for $T_{i}$ is analogous and follows from the identity $\adj(T_{i}) = T_{k-i}^{-1} = q^{-1}(T_{k-i} + q -1)$. Hence, we verify only the actions of  $d_{+}$ and $d_{-}$. Beginning with $d_{+}$ we have,
\begin{align*}
 d_{+}[\lambda; \underline{w}]^{*}\Bigl([\mu;\underline{v}]\Bigr)   &=  [\lambda; \underline{w}]^{*}\Bigl(\adj(d_{+})[\mu; \underline{v}]\Bigr)= 
 [\lambda; \underline{w}]^{\ast}\Bigl(q^{-s/2}d_{-}[\mu; \underline{v}]\Bigr) 
 \\ &=  q^{s/2}[\lambda; \underline{w}]^{*}\Bigl([\mu\cup v_s; v_{s-1}, \dots, v_{1}]\Bigr) 
 =  q^{s/2}\delta_{s-1, k}\delta_{\lambda, \mu \cup v_s}\delta_{[\underline{w}], [v_{s-1}, \dots, v_{1}]} 
\end{align*}

\noindent Hence, we conclude that
\[
d_{+}[\lambda; \underline{w}]^{*} = q^{(k+1)/2}\sum_{\lambda = \mu \cup x}[\mu;x,\underline{w}]^{*}.
\]
Continuing with $d_{-}$, we follow a similar strategy:
\begin{align*}
 d_{-}[\lambda; \underline{w}]^{*}\Bigl([\mu; \underline{v}]\Bigr) &=  [\lambda; \underline{w}]^{*}\Bigl(\adj(d_{-})[\mu; \underline{w}]\Bigr) =
 q^{(s+1)/2}[\lambda; \underline{w}]^{*}\Bigl(d_{+}[\mu; \underline{v}]\Bigr) \\
 &= q^{(s+1)/2}[\lambda; \underline{w}]^{*}\Bigl(\sum_{x}c(\mu; \underline{v},x)[\mu;\underline{v},x]\Bigr) 
=  \delta_{k-1, s}q^{(s+1)/2}\aut\Bigl(c(\mu; \underline{v},x)\Bigr)\delta_{\lambda, \mu}
\delta_{[\underline{w}],[\underline{v},x]}
\end{align*}
and obtain,
\[
d_{-}[\lambda; \underline{w}]^{*} = q^{k/2}\aut \Bigl(c(\lambda; \underline{w})\Bigr)[\lambda; w_k, \dots, w_2]^{*}.
\]
\end{proof}

Given a $\Bqt^{\ext}$-homomorphism $f: V(E_1) \to V(E_2)$, we can also give a formula for the dual map $f^{*}: V(E_2)^{*} \to V(E_1)^{*}$ in terms of the dual bases $[\lambda; \underline{w}]^{*}$.

Recall from Section \ref{subsec:Homomorphisms} that given any poset $E$ we defined a new poset $E_{\mathbf{0}} := E\sqcup\{\mathbf{0}\}$, with $\lambda < \mathbf{0}$ for all $\lambda \in E$, such that for any $\Bqt^{ext}$-module homomorphism $f:V(E) \to V(E')$ there exists a map $F: E_{\bf0} \to E'_{\bf0}$ satisfying certain properties and vice versa (see Theorem \ref{thm: hom}). 

\begin{lemma}\label{lem: dual hom}
Let $E_1, E_2$ be excellent posets, and $f: V(E_1) \to V(E_2)$ a $\Bqt^{\ext}$-homomorphism. Assume that the edge functions on $E_1$ and $E_2$ are compatible as in Lemma \ref{lem: compatible edge}, and that $f$ is given as in Proposition \ref{prop:hom}. Then, $f^{*}: V(E_2)^{*} \to V(E_1)^{*}$ is given by:
\[
f^{*}\Bigl([\mu; \underline{w}]^{*}\Bigr) = \begin{cases} \aut(\alpha_{\lambda})[\lambda; \underline{w}]^{*} & ; \text{if there exists} \; \lambda \in E_1 \; \text{such that} \; F(\lambda) = \mu,  \; \\
& \;\; \;\; \lambda \cup \underline{w} \in E_1, \; \text{and} \;\; F(\lambda \cup \underline{w}) \neq \mathbf{0}, \\
0 & ; \; \text{else}.\end{cases}
\]
In particular, if there is a $\lambda$ for which $F(\lambda) = \mu$, then such a $\lambda$ is unique by the simple spectrum condition. 
\end{lemma}

\begin{proof}
The result follows from direct computation.
\end{proof}

Let us now produce a poset $E^{\vee}$ for which $V(E)^{*} \cong V(E^{\vee})$. 

\begin{definition} Given a poset $E$, define the \newword{dual poset $E^{\vee}$} as follows. 
The nodes $\lambda^{\vee}$ of $E^{\vee}$ are in bijection with the nodes $\lambda$ of $E$, but the order is reversed:
$$
\lambda^{\vee}\prec \mu^{\vee} \Leftrightarrow \lambda\succ \mu.
$$
The weights for $E^{\vee}$ are defined by
$$
p_m(\lambda^{\vee})=-\theta(p_m(\lambda)).
$$
\end{definition}

Clearly, $E^{\vee}$ is locally finite and graded by $|\lambda^{\vee}|=-|\lambda|$. If $\mu$ covers $\lambda$, then $\mu=\lambda\cup x$ and
$$
p_m(\mu^{\vee})=-\theta(p_m(\mu))=-\theta(p_m(\lambda))-\theta(x)^m=p_m(\lambda^{\vee})-\theta(x)^m,
$$
so that $\lambda^{\vee}=\mu^{\vee}\cup \theta(x)$.

\begin{lemma}
Assume $E$ is excellent. Then, $E^{\vee}$ is excellent as well. 
\end{lemma}
\begin{proof}
Let $[\lambda^{\vee}; x,y]$ be a two-step chain in $E^{\vee}$:
\[
\lambda^{\vee} \to \lambda^{\vee}\cup x \to \lambda^{\vee} \cup x \cup y =: \mu^{\vee}
\]
This corresponds to the following chain in $E$:
\[
\mu \to \mu \cup \aut(y) \to \mu \cup \aut(y) \cup \aut(x).
\]
Since $E$ is excellent, then $\aut(x) \neq \aut(y)$ and $\aut(x) \neq (qt)^{\pm 1} \aut(y)$, so $y \neq x$ and $y \neq (qt)^{\pm 1}x$. Thus, Condition (1) in Definition \ref{def: excellent new} is satisfied in $E^{\vee}$. 

Now, $\aut(x) = q\aut(y)$ if and only if $y = qx$, $\aut(x) = t\aut(y)$ if and only if $y = tx$; likewise, $[\mu; \aut(x),\aut(y)]$ is a chain in $E$ if and only if $[\lambda^{\vee}; x,y]$ is a chain in $E^{\vee}$. From here, it follows that Condition (2) in Definition \ref{def: excellent new} is satisfied in $E^{\vee}$.  Hence, $E^{\vee}$ is excellent.
\end{proof}

\begin{lemma}
\label{lem: c dual}
Assume that $\mu=\lambda\cup x$, so that $\lambda^{\vee}=\mu^{\vee}\cup \aut(x)$. Then the coefficients 
\begin{equation}
\label{eq: c dual}
c_{E^{\vee}}(\mu^{\vee};\aut(x)):=\aut(c_{E}(\lambda;x))
\end{equation}
satisfy \eqref{eq: monodromy}.
\end{lemma}

\begin{proof}
Suppose that $\mu=\lambda\cup x\cup y$, so that $\lambda^{\vee}=\mu^{\vee}\cup \aut(x)\cup \aut(y)$. Then
\begin{align*}
\frac{c_{E^{\vee}}(\mu^{\vee};\aut(x))c_{E^{\vee}}(\mu^{\vee}\cup \aut(x);\aut(y))}{c_{E^{\vee}}(\mu^{\vee};\aut(y))c_{E^{\vee}}(\mu^{\vee}\cup \aut(y);\aut(x))}&=\aut\left[
\frac{c_{E}(\lambda\cup y;x)c_{E}(\lambda;y)}{c_{E}(\lambda\cup x;y)c_{E}(\lambda;x)}\right]
\\
&=-\frac{(\aut(y)-t^{-1}\aut(x))(\aut(y)-q^{-1}\aut(x))(\aut(x)-q^{-1}t^{-1}\aut(y))}{(\aut(x)-t^{-1}\aut(y))(\aut(x)-q^{-1}\aut(y))(\aut(y)-q^{-1}t^{-1}\aut(x))}
\\ &=-\frac{(\aut(x)-t\aut(y))(\aut(x)-q\aut(y))(\aut(y)-qt\aut(x))}{(\aut(y)-t\aut(x))(\aut(y)-q\aut(x))(\aut(x)-qt\aut(y))}.
\end{align*}
\end{proof}

\begin{theorem}\label{thm: duality poset}
    Let $E$ be an excellent weighted poset, then $V(E)^{*} \cong V(E^{\vee})$ as $\Bqt^{\ext}$-representations. 
\end{theorem}

\begin{proof}
We have representation $V(E^{\vee})$ of $\Bqt^{ext}$, with basis the set of good chains $[\lambda^{\vee}; \aut(w_1), \dots, \aut(w_k)]$ in $E^{\vee}$. Setting $\mu^{\vee} := \lambda^{\vee} \cup \aut(w_{1}) \cup \cdots \cup \aut(w_k)$, we have a bijection between a basis of $V(E^{\vee})$ and a basis of $V(E)^{*}$:
\begin{equation}
\label{eq: dual bijection}
q^{n(k)}[\lambda^{\vee}; \aut(w_1), \dots, \aut(w_k)]^{\resc} \leftrightarrow [\mu; w_k, \dots, w_1]^{*},\quad n_k=-\frac{k(k+1)}{4}.
\end{equation}

By comparing Theorem \ref{thm: rescale} and Proposition \ref{prop: dual operators}, it is not hard to check that the actions of $\Delta_{p_m},z_i$ and $T_i$ in the respective bases agree. For $d_{-}$, by Proposition \ref{prop: dual operators} we get
$$d_{-}[\mu; w_k, \dots, w_1]^{*} =  q^{k/2}\aut\Bigl(c(\mu; w_k, \dots, w_1)\Bigr)[\mu; w_k, \dots, w_2]^{*}
$$
and, setting $\nu := \mu \cup w_k \cup \cdots \cup w_2$ (so that $\lambda = \nu \cup w_1$) by Lemma \ref{lem: c dual} we have
\begin{align*}
\aut(c_E(\mu;\underline{w}))&=\aut\left(q^kc_E(\nu;w_1)\prod_{i=2}^{k}\frac{w_1-tw_i}{w_1-qtw_i}\right) \\
&=
q^{-k}c_{E^{\vee}}(\lambda^{\vee};\aut(w_1))
\prod_{i=2}^{k}\frac{\aut(w_1)-t^{-1}\aut(w_i)}{\aut(w_1)-q^{-1}t^{-1}\aut(w_i)}
\\
&=q^{-k}c_{E^{\vee}}(\lambda^{\vee};\aut(w_1))\frac{t^{1-k}}{q^{1-k}t^{1-k}}\prod_{i=2}^{k}\frac{\aut(w_i)-t\aut(w_1)}{\aut(w_i)-qt\aut(w_1)}\\
&=
q^{-1}c_{E^{\vee}}(\lambda^{\vee};\aut(w_1))\prod_{i=2}^{k}\frac{\aut(w_i)-t\aut(w_1)}{\aut(w_i)-qt\aut(w_1)}.
\end{align*}
On the other hand, by Theorem \ref{thm: rescale}
we get
\begin{align*}
d_{-}&\left(q^{n_k}[\lambda^{\vee}; \aut(w_1), \dots, \aut(w_k)]^{\resc}\right) \\
&=q^{k-1+n_k-n_{k-1}}c(\lambda^{\vee};\aut(w_1))\prod_{i=2}^{k}\frac{\aut(w_i)-t\aut(w_1)}{\aut(w_i)-qt\aut(w_1)}q^{n_{k-1}}[\lambda^{\vee}\cup \aut(w_1);\aut(w_2)\ldots,\aut(w_k)]^{\resc}.
\end{align*}
Now note that, under \eqref{eq: dual bijection} we have
\[
q^{n_{k-1}}[\lambda^{\vee}\cup \aut(w_1); \aut(w_2), \dots, \aut(w_k)]^{\resc} = q^{n_{k-1}}[\nu^{\vee}; \aut(w_2), \dots, \aut(w_k)]^{\resc} \leftrightarrow [\mu; w_k, \dots, w_2]^{*}.
\]
Similarly, for $d_+$ by Theorem \ref{thm: rescale} we get
$$d_{+}\left(q^{n_k}[\lambda^{\vee}; \aut(w_1), \dots, \aut(w_k)]^{\resc}\right)=
q^{n_k-n_{k+1}}\sum_{x}q^{n_{k+1}}[\lambda^{\vee}; \aut(w_1), \dots, \aut(w_k),x]^{\resc}
$$
The statement now follows since
$$
k-1+n_k-n_{k-1}=\frac{k}{2}-1,\ n_k-n_{k+1}=\frac{k+1}{2}.
$$
\end{proof}

\subsection{Tensor Products of Dual Poset Representations}
 
With our definition of tensor product as in Section \ref{sec: tensor}, it is easy to see that it is not fully compatible with the duality, in particular, by Theorem \ref{thm: hom} there are no maps
$$
V(\{\bullet\}) \to V(E)\boxtimes V(E^{\vee})\to V(\{\bullet\})
$$
satisfying the definition of a rigid monoidal category \cite[2.1]{EGNO}.

Nevertheless, in this subsection we show that there is some partial compatibility between duality and tensor product.

\begin{definition}
Let $\psi(y,u)$ be a rational function with coefficients in $\C(q,t)$. We define another rational function $\psi^{\aut}$ by the equation
$$
\psi^{\aut}(\aut(y),\aut(u))=\aut\left[\psi(y,u)\right]^{-1}.
$$
\end{definition}

\begin{lemma}
\label{lem: M dual}
The function
\begin{equation}
\label{eq: M dual}
M_{\psi^{\aut},E^{\vee}}(\mu^{\vee},\aut(u))=\aut\left[M_{\psi,E}(\mu,u)\right].
\end{equation}
satisfies \eqref{eq: M psi} for the function $\psi^{\aut}$ and the dual poset $E^{\vee}$.
\end{lemma}

\begin{proof}
Suppose that $\mu=\lambda\cup y$, then $\lambda^{\vee}=\mu^{\vee}\cup \theta(y)$. We have
$$
M_{\psi,E}(\mu;u)=M_{\psi,E}(\lambda;u)\psi(y,u),\ M_{\psi,E}(\mu;u)\psi(y,u)^{-1}=M_{\psi,E}(\lambda;u).
$$
By applying $\aut$ we get
$$
M_{\psi^{\aut},E^{\vee}}(\mu^{\vee},\aut(u))\psi^{\aut}(\aut(y),\aut(u))=M_{\psi^{\aut},E^{\vee}}(\lambda^{\vee},\aut(u))
$$
as desired.
\end{proof}

\begin{theorem}
There exists an isomorphism
$$
V((E_1\times E_2)^{\vee})\simeq V(E_1^{\vee})\boxtimes V(E_2^{\vee}). 
$$
More precisely, we have a monoidal contravariant functor 
$
\mathbb{D}:\mathcal{C}_{\psi_1,\psi_2}\to \mathcal{C}_{\psi_1^{\aut},\psi_2^{\aut}}
$
defined by
$$
\mathbb{D}(E,c,M_{\psi_1},M_{\psi_2})=(E^{\vee},c_{E^{\vee}},M_{\psi_2^{\aut},E^{\vee}},M_{\psi_1^{\aut},E^{\vee}})
$$
where $c_{E^{\vee}}$ is defined by \eqref{eq: c dual} and $M_{\psi_i^{\aut},E^{\vee}}$  is defined by \eqref{eq: M dual}. 
\end{theorem}

\begin{proof}
First, we need to check that $(\psi_1^{\aut},\psi_2^{\aut})$ satisfy \eqref{eq: psi1 psi2}. Indeed,
\[
\frac{\psi_1^{\aut}(x,y)}{\psi_2^{\aut}(y,x)}=\aut\left[\frac{\psi_1(\aut(x),\aut(y))}{\psi_2(\aut(y),\aut(x))}\right]^{-1}=-
\frac{(x-ty)(x-qy)(y-qtx)}{(y-tx)(y-qx)(x-qty)}
\]
similarly to the proof of Lemma \ref{lem: c dual}.
By Lemma \ref{lem: c dual} the function $c_{E^{\vee}}$ satisfies \eqref{eq: monodromy}, and by Lemma \ref{lem: M dual} the functions $M_{\psi_i^{\aut},E^{\vee}}$ satisfy \eqref{eq: M psi}.  This implies that the functor $\mathbb{D}$ is well defined.

Next, we check the identity $(E_1\times E_2)^{\vee}=E_1^{\vee}\times E_2^{\vee}$ on the level of weighted posets. This follows from the bijection $(\lambda\times \mu)^{\vee}\leftrightarrow \lambda^{\vee}\times \mu^{\vee}$ and the fact that the eigenvalues of $\Delta_{p_m}$ are both equal to 
$$
-\aut(p_m(\lambda)+p_m(\mu))=-\aut(p_m(\lambda))-\aut(p_m(\mu)).
$$
Finally, by applying $\aut$ to \eqref{c product} we verify that $\mathbb{D}$ is monoidal.
\end{proof}

\begin{remark}
By Lemma \ref{lem: dual hom} and \eqref{eq: dual bijection}, we understand the behavior of the functor $\mathbb{D}$ on morphisms. Given a morphism $F: (E_1, c_1, M_{\psi_1; E_1}, M_{\psi_2; E_2}) \to (E_2, c_2, M_{\psi_1; E_2}, M_{\psi_2; E_2})$ in $\mathcal{C}_{\psi_1, \psi_2}$ then the morphism $\mathbb{D}(F)$ is given by
\[
\mathbb{D}(F)(\mu^{\vee}) = \begin{cases} \lambda^{\vee} & ; \; \text{if there exists} \; \lambda \in E_1 \; \text{such that} \; F(\lambda) = \mu, \\ \mathbf{0} & ; \; \text{else.} \end{cases}
\]
We note that if such $\lambda$ exists then it is unique by the simple spectrum condition. 
\end{remark}

\subsection{$\adj$ and the Weyl Involution of Symmetric Functions} In this section, we show that the involution $\adj: \Bqt^{\ext} \to \Bqt^{\ext}$ is an extension of the classical graded \newword{Weyl involution} $\omega: \Sym_{q,t} \to \Sym_{q,t}$ on symmetric functions \cite{Macdonald}\footnote{The involution $\omega$ is denoted by $\overline{\omega}$ in \cite{Shuffle}.}. The involution $\omega$ is defined by the fact that it is $\aut$-linear and
\[
\omega(e_{i}) = q^{-i}h_{i},
\]
where $e_{i}$ and $h_{i}$ are the elementary and complete symmetric functions of degree $i$, respectively. Note that the fact that $\omega$ is $\aut$-linear implies that it is indeed an involution. For any partition $\lambda$, we have
\[
\omega(s_{\lambda}) = q^{-|\lambda|}s_{\lambda^{t}}
\]
where $s_{\lambda}$ is the Schur function corresponding to $\lambda$.

The algebra $\Sym_{q,t}$ embeds into $\Bqt$ (and thus also into $\Bqt^{\ext}$) as follows. Let us denote by $\Aq \subseteq \Bqt$ the subalgebra generated by the idempotents $\e_{k}$ together with the elements $d_{+}, d_{-}, T_{i}$. Note that the involution $\adj$ preserves the algebra $\Aq$. 

\begin{theorem}[\cite{Shuffle}]
There is an algebra isomorphism $\Sym_{q,t} \cong \e_{0}\Aq\e_{0}$. 
\end{theorem}
\begin{proof}
    The fact that we have an isomorphism \emph{of vector spaces} is \cite[Theorem 5.2]{Shuffle}. Hence, need only verify that this isomorphism holds at the level of algebras. We recall how the definition of the vector space isomorphism $\phi: \e_{0}\Aq\e_{0} \to \Sym_{q,t}$ given in \cite{Shuffle}.

    For $k \geq 0$, let $V_k := \Sym_{q,t}[y_1, \dots, y_k]$, so that $V_0 =\Sym_{q,t} $. The algebra $\Aq$ acts on $V := \bigoplus_{k \geq 0}V_{k}$ by restricting the polynomial representation of $\Bqt$ from Section \ref{sec: polynomial}. It is a result of \cite{Shuffle} that the operator $y_i = q^{i-k}T_{i-1}^{-1}\cdots T_1^{-1}\varphi T_{k-1}\cdots T_i$ acts on $V_k$ by multiplication by the variable $y_i \in V_k$.  
    The isomorphism $\phi: \e_0\Aq\e_0 \to \Sym_{q,t}$ then sends an element in $g \in \e_0\Aq\e_0$ to its action $g(1)$ on $1 \in \Sym_{q,t}$. 

In particular, 
\begin{align*}
\phi(\e_{0}d_{-}y_1^{i-1}d_{+}\e_{0}) = & \e_{0}d_{-}y_1^{i-1}d_{+}\e_0(1)  
=  \e_0d_{-}(y_1^{i-1}) 
=  \e_{0}\Bigl(-y_{1}^{i-1}\sum_{n \geq 0}(-y_1^{-n})e_{n}|_{y_{1}^{-1}}\Bigr) 
=  (-1)^{i-1}e_{i}.
\end{align*}
Since the elementary symmetric functions freely generate $\Sym_{q,t}$ as a commutative algebra and $\phi$ is bijective, to see $\phi$ is an algebra isomorphism it suffices to show that
\begin{align*}
\phi(\e_{0}d_{-}y_{1}^{i-1}d_{+}\e_{0}d_{-}y_{1}^{j-1}d_{+}) = & (-1)^{j-1}\e_{0}d_{-}y_1^{i-1}d_{+}(e_{j}) \\
= & (-1)^{j-1}\e_0d_{-}(y_1^{i-1}e_{j}[X + (q-1)y_1]) \\ 
= & -(-1)^{j-1}y_{1}^{i-1}e_{j}[X+(q-1)y_1 - (q-1)y_1]\sum_{n \geq 0}(-y_1)^{-n}e_{n}|_{y_{1}^{-1}} \\
= & (-1)^{j-1}(-1)^{i-1}e_{j}e_{i}
\end{align*}
and the result follows. 
\end{proof}

As remarked above, the involution $\adj$ preserves the algebra $\e_0\Aq\e_0 \subseteq \Bqt$ and therefore induces an involution on the algebra $\Sym_{q,t} \cong \e_0\Aq\e_0$.

\begin{theorem}\label{thm: omega}
We have
\[
\adj|_{\Sym_{q,t}} = \omega. 
\]
\end{theorem}
\begin{proof}
As we have seen above, the elementary symmetric function $e_{i}$ corresponds to the element
\[
(-1)^{i-1}\e_{0}d_{-}y_{1}^{i-1}d_{+}\e_{0} \in \e_0\Aq\e_0 \subseteq \Bqt
\]
so we need to apply $\adj$ to this element. Note that this is
\[
(-1)^{i-1}q^{-1}\e_0d_{-}\adj(y_1)^{i-1}d_{+}\e_0
\]
so our first job is to find $\adj(y_1)$. We have
\begin{align*}
\adj(y_1) &=  \adj\left(\frac{1}{q-1}(d_+d_- - d_-d_+)\right) 
=  \frac{q^{-2}}{q^{-1} - 1}(qd_{+}d_{-} - d_{-}d_{+}) \\
&=  \frac{-q^{-1}}{q-1}(qd_{+}d_{-} - d_{-}d_{+} + d_{+}d_{-} - d_{+}d_{-}) 
=  -q^{-1}(y_1 + d_{+}d_{-}).
\end{align*}
Thus,
\[
\adj((-1)^{i-1}\e_0d_{-}y_{1}^{i-1}d_{+}\e_{0}) = (-1)^{i-1}(-q)^{-i+1}q^{-1}\e_0d_{-}(y + d_{+}d_{-})^{i-1}d_{+}\e_0. 
\]
Note that
\begin{align*}
\e_0d_{-}(y+d_{+}d_{-})^{i-1}d_{+}\e_0 &=  \e_0d_{-}\Bigl(\sum_{\substack{a_1 + a_2 + \cdots  + a_{2k} = i-1 \\ a_{s} = 0 \Rightarrow a_{s+1} = 0, s \geq 2}}y^{a_1}(d_+d_-)^{a_2}\cdots y^{a_{2k-1}}(d_+d_-)^{a_{2k}}\Bigr)d_{+}\e_0 \\
&\mapsto  \sum_{k \geq 0} \Bigl(\sum_{\substack{d_1 + \cdots + d_k = i \\ d_1, \dots, d_k > 0}}(-1)^{i-k}e_{d_{1}}\cdots e_{d_{k}} \Bigr)
\end{align*}
where in the last part we have used the fact that $\e_0d_{-}y^{a}d_{+}\e_0 \mapsto (-1)^{a}e_{a+1}$ and $\e_0d_{-}(d_{+}d_{-})^{b}d_{+}\e_0 \mapsto e_{1}^{b+1}$, so that we have:
\begin{itemize}[leftmargin=*]
\item Assuming $a_1, \dots, a_{2k} > 0$:
\begin{align*}
&\e_0d_{-}\left(y^{a_1}(d_{+}d_{-})^{a_2}\cdots y^{a_{2k-1}}(d_{+}d_{-})^{a_{2k}}\right)d_{+}\e_0 \\
&=(\e_0d_{-}y^{a_1}d_{+}\e_0)(\e_0d_{-}(d_{+}d_{-})^{a_2 - 2}d_{+}\e_0)(\e_0d_{-}y^{a_3}d_{+}\e_0)\cdots (\e_0d_{-}y^{a_{2k-1}}d_{+}\e_0)(\e_0 d_{-}(d_{+}d_{-})^{a_{2k} - 1}d_{+}\e_0)  \\
&\mapsto (-1)^{a_1 + a_3 + \cdots + a_{2k-1}}e_{a_1+1}e_{a_3+1}\cdots e_{a_{2k-1}+1}e_{1}^{a_{2} -1}e_1^{a_4 - 1}\cdots e_1^{a_{2k-2} - 1}e_1^{a_{2k}},
\end{align*}
where the term $\e_0(d_{-}(d_{+}d_{-})^{a_{2j} - 2}d_{+})\e_0$ is skipped if $a_{2j} = 1$.

\item Assuming $a_1, a_{2}, \dots, a_{2k-1} \neq 0, a_{2k} = 0$:
\begin{align*}
&\e_0d_{-}\left(y^{a_1}(d_{+}d_{-})^{a_2}\cdots y^{a_{2k-1}}\right)d_{+}\e_0 = 
\\
&=(\e_0d_{-}y^{a_1}d_{+}\e_0)(\e_0d_{-}(d_{+}d_{-})^{a_2 - 2}d_{+}\e_0)(\e_0d_{-}y^{a_3}d_{+}\e_0)\cdots (\e_0d_{-}y^{a_{2k-1}}d_{+}\e_0) \\
&\mapsto (-1)^{a_1 + a_3 + \cdots + a_{2k-1}}e_{a_1+1}e_{a_3+1}\cdots e_{a_{2k-1}+1}e_{1}^{a_{2} -1}e_1^{a_4 - 1}\cdots e_1^{a_{2k-2} - 1},
\end{align*}
where again the term $\e_0(d_{-}(d_{+}d_{-})^{a_{2j} - 2}d_{+})\e_0$ is skipped if $a_{2j} = 1$.

\item Assuming $a_1 = 0$ but $a_{2}, \dots, a_{2k} \neq 0$:
\begin{align*}
&\e_0d_{-}\left((d_{+}d_{-})^{a_2}\cdots y^{a_{2k-1}}(d_{+}d_{-})^{a_{2k}}\right)d_{+}\e_0 
\\&= (\e_0d_{-}(d_{+}d_{-})^{a_2 - 1}d_{+}\e_0)(\e_0d_{-}y^{a_3}d_{+}\e_0)\cdots (\e_0d_{-}y^{a_{2k-1}}d_{+}\e_0)(\e_0 d_{-}(d_{+}d_{-})^{a_{2k} - 1}d_{+}\e_0)
\\
 &\mapsto 
(-1)^{a_3 + \cdots + a_{2k-1}}e_{a_3+1}\cdots e_{a_{2k-1}+1}e_{1}^{a_{2}}e_1^{a_4 - 1}\cdots e_1^{a_{2k-2} - 1}e_1^{a_{2k}},
\end{align*}

\item Assuming $a_{1}, a_{2k} = 0$, $a_2, \dots, a_{2k-1} \neq 0$: 
\begin{align*}
&\e_0d_{-}\left((d_{+}d_{-})^{a_2}\cdots y^{a_{2k-1}}\right)d_{+}\e_0 
\\
&= (\e_0d_{-}(d_{+}d_{-})^{a_2 - 1}d_{+}\e_0)(\e_0d_{-}y^{a_3}d_{+}\e_0)\cdots (\e_0d_{-}y^{a_{2k-1}}d_{+}\e_0)
\\
&\mapsto 
(-1)^{a_3 + \cdots + a_{2k-1}}e_{a_3+1}\cdots e_{a_{2k-1}+1}e_{1}^{a_{2}}e_1^{a_4 - 1}\cdots e_1^{a_{2k-2} - 1}.
\end{align*}
    
\end{itemize} We conclude that, upon the isomorphism $\Sym_{q,t} \cong \e_0\Aq\e_0$ we get
\[
\adj(e_{i}) = q^{-i}\sum_{k = 1}^{i}(-1)^{i-k} \sum_{\substack{d_1 + \cdots + d_k = i \\ d_1, \dots, d_k > 0}}e_{d_{1}}\cdots e_{d_{k}}.
\]
Finally, to confirm that indeed $\adj(e_{i})= q^{-i}h_{i}$ it suffices to verify that it satisfies the defining relation 
\[\sum_{i=0}^m (-1)^{m-i}e_{i}h_{m-i}=0.\]
Computing we obtain, 
\begin{align*}
&\sum_{i=0}^m (-1)^{m-i}e_{i}q^{m-i}\adj(e_{m-i}) 
= \sum_{i=0}^m e_{i} \Big(\sum_{k = 1}^{m-i}(-1)^{k}\sum_{\substack{d_1 + \cdots + d_k = m-i \\ d_1, \dots, d_k > 0}}e_{d_{1}}\cdots e_{d_{k}}\Bigr)
\\
&= \sum_{k = 1}^{m} (-1)^k\sum_{\substack{d_1 + \cdots + d_k = m \\ d_1, \dots, d_k > 0}}e_{d_{1}}\cdots e_{d_{k}} 
+ \sum_{i=1}^m e_{i}\Bigl(\sum_{k = 1}^{m-i} (-1)^{k}\sum_{\substack{d_1 + \cdots + d_k = m-i \\ d_1, \dots, d_k > 0}}e_{d_{1}}\cdots e_{d_{k}}\Bigr)
\\
&= -e_m + \sum_{k = 2}^{m} (-1)^k   
\sum_{i=1}^{m-k+1}e_i 
 \Bigl(\sum_{\substack{d_2 + \cdots + d_k = m-i \\ d_2, \dots, d_k > 0}}e_{d_{2}}\cdots e_{d_{k}} \Bigr)
+  \sum_{i=1}^{m} e_{i} \Bigl(\sum_{k = 1}^{m-i} (-1)^{k}\sum_{\substack{d_1 + \cdots + d_k = m-i \\ d_1, \dots, d_k > 0}}e_{d_{1}}\cdots e_{d_{k}}\Bigr)\\
&=\sum_{i=1}^{m-1}e_i \Bigl(\sum_{k=2}^{m-i+1}(-1)^k \sum_{\substack{d_2 + \cdots + d_k = m-i \\ d_2, \dots, d_k > 0}}e_{d_{2}}\cdots e_{d_{k}} \Bigr)
+ \sum_{i=1}^{m-1} e_{i} \Bigl(\sum_{k = 1}^{m-i} (-1)^{k}\sum_{\substack{d_1 + \cdots + d_k = m-i \\ d_1, \dots, d_k > 0}}e_{d_{1}}\cdots e_{d_{k}}\Bigr)
\\
&=\sum_{i=1}^{m-1}e_i \Bigl(\sum_{k=1}^{m-i}(-1)^{k+1} \sum_{\substack{d_2 + \cdots + d_{k+1} = m-i \\ d_2, \dots, d_{k+1} > 0}}e_{d_{2}}\cdots e_{d_{k+1}} \Bigr)
+ \sum_{i=1}^{m-1} e_{i} \Bigl(\sum_{k = 1}^{m-i} (-1)^{k}\sum_{\substack{d_1 + \cdots + d_k = m-i \\ d_1, \dots, d_k > 0}}e_{d_{1}}\cdots e_{d_{k}}\Bigr)\\
&=0.
\end{align*}
\end{proof}

\begin{remark} It is tempting to try relate $\adj$ to the involution $\mathcal{N}$ on $\Aqt$ described in \cite{Shuffle} which when restricted to $\Sym_{q,t}$ is equal to $\nabla \omega$. Unfortunately, a direct relation between these is unclear. Whereas $\adj$ clearly preserves the subalgebra $\mathbb{A}_q$ generated by $d_\pm, T_i$, the involution $\mathcal{N}$ does not. 
 \end{remark}

\section{Towards a Classification of Calibrated Representations}\label{sec: classification}

\subsection{Reconstructing a Poset from a Representation}\label{sec: reconstruct poset} 

In the previous sections we started from a poset $E$ and constructed a representation of $\Bqt^{\ext}$. In this section we  we start with an \emph{arbitrary} calibrated representation $V$ and (under certain assumptions) reconstruct a poset from it. 

Let $V$ be a calibrated representation of $\Bqt^{ext}$. Recall that this means the following:
\begin{enumerate}
    \item $V = \bigoplus_{k \geq 0} V_{k}$, with the idempotent $\e_{k}$ acting as the projection to the summand $V_{k}$.
    \item Each summand $V_{k}$ comes equipped with an eigenbasis $\cB_{k}$ consisting of simultaneous eigenvectors for $z_1, \dots, z_k$ and $\Delta_{p_m}$. Moreover, each element $v_{k} \in \cB_{k}$ is (up to scalar) uniquely determined by its weight.
\end{enumerate}

For $v_{k} \in V_{k}$, we will denote by $\zeta_1(v_k), \dots, \zeta_k(v_k)$ its eigenvalues under $z_1, \dots, z_k$ (so that $z_{i}(v_k) = \zeta_i(v_k)v_k$). Similarly, we denote its eigenvalue under   $\Delta_{p_m}$ by $p_{m}(v_k)$.
\medskip

\textbf{We will make the following simplifying assumptions:}

\begin{assumption}
\label{ass: d-}
For every $v_{k} \in \cB_{k}$, $d_{-}^{k}(v_k) \neq 0$. 
\end{assumption}

\begin{assumption}\label{ass: semisimple}
    We will assume that $V = \bigoplus_{k \geq 0} V_{k}$ is a calibrated representation of $\Bqt^{ext}$ such that $V_{k}$ is a completely reducible $\AH_{k}$-representation for every $k \geq 0$. 
\end{assumption}

Let us consider the set $\cB := \bigsqcup_{k \geq 0}\cB_{k}$.

\begin{definition}
We define the set $E := \cB/\!\sim$ where the equivalence relation on $\cB$ is given by:
\[
\cB_{k} \ni v_{k} \sim v_{s} \in \cB_s \qquad \text{if} \qquad p_{m}(v_k) - \sum_{i = 1}^{k}\zeta_{i}^{m}(v_k) = p_{m}(v_s) - \sum_{j = 1}^{s}\zeta_{j}^{m}(v_s) \, \text{for every} \, m\geq 1.
\]  
\end{definition}

For $v_k \in \cB_k \subseteq \cB$, we denote its equivalence class by $[v_k] \in E$. Now endow $E$ with a partial order structure. First, we define a weighting on $E$. Let,
\begin{equation}\label{eqn:delta operators on the quotient}
p_{m}[v_k] := p_{m}(v_k) - \sum_{i = 1}^{k}\zeta_{i}^{m}(v_k).
\end{equation}

Note that if $v_k \in \cB_k$, $d_{-}v_k$ is a simultaneous eigenvector for $z_1, \dots, z_{k-1}$ and $\Delta_{p_m}$, and by Assumption \ref{ass: d-} we have $d_-(v_k)\neq 0$. Up to rescaling, we may assume that $d_{-}v_k \in \cB_{k-1}$.

\begin{lemma}
The relation
\[
[v_{k}] \prec [d_{-}v_{k}]
\]
is the covering relation for a partial order on $E$.
\end{lemma}
\begin{proof}
    We need to show that no cycles appear in the transitive closure of $\prec$. Note that
    \begin{equation}\label{eqn: difference in weights}
    p_{m}[d_{-}v_k] - p_{m}[v_k] = \zeta^{m}_{k}(v_k) \; \text{for all} \; m \geq 1.
    \end{equation}
    Thus, if we had a cycle $[v_{k_1}] \prec [v_{k_2}] \prec \cdots \prec [v_{k_{s}}] = [v_{k_{1}}]$ we would have
    \[
    \sum_{i = 1}^{s}\zeta_{k_{i}}^{m}(v_{k_{i}}) = 0 \; \text{for all} \; m \geq 1
    \]
    and it follows that $\zeta_{k_{i}}(v_{k_{i}}) = 0$ for all $i$, which contradicts our standing assumption that the elements $z_{i}$ act with nonzero eigenvalues. 
\end{proof}

\begin{remark}
    Note that \eqref{eqn: difference in weights} implies that the collection $p_{m}$ is indeed a weighting on the poset $E$. 
\end{remark}

\begin{lemma}
\label{lem: inject into chains}
Let $v_k \in \cB_k$, then \begin{equation}\label{eqn:from basis to chains}
    c_k(v_k) := ([v_{k}] \prec [d_{-}v_{k}] \prec [d_{-}^{2}v_k] \prec \cdots \prec [d_{-}^{k}v_{k}])
    \end{equation}
    is a maximal chain in $E$. Furthermore, the map $c_k:\cB_k\to \Ch(E)$  is injective. 
\end{lemma}

\begin{proof}
By definition of the partial order on $E$, $[v_{k}] \prec [d_{-}v_{k}]$ is a cover relation, so $c_k(v_k)$ is indeed a maximal chain.

    Assume that $v_k, v'_k \in \cB_k$ are such that $c(v_k) = c(v'_k)$. It follows from \eqref{eqn: difference in weights} that $\zeta_i(v_k) = \zeta_i(v'_k)$ for $i = 1, \dots, k$. Since $c(v_k) = c(v'_k)$ implies, in particular, that $[v_k] = [v'_k]$, we then obtain from \eqref{eqn:delta operators on the quotient} that $p_m(v_k) = p_{m}(v'_k)$ for every $m$. Now the result follows by the simple spectrum assumption. 
\end{proof}

\subsection{Calibrated Representations from Posets}\label{sec: calibrated from posets}

By Lemma \ref{lem: inject into chains}, we can assume that the calibrated representation $V(\cB)=\bigoplus_{k=0}^{\infty}V_k(\cB)$ is given by some underlying weighted poset $E$, and some subset $\cB$ of the set of maximal chains in $E$.

The joint eigenbasis in $V_k(\cB)$ for the operators $\Delta_{p_k}$ and $z_i$ is given by $[\lambda;\underline{w}]\in \cB$:  
\begin{equation}
\label{eq: Delta z new}
\Delta_{p_m}[\lambda;\underline{w}]=(p_m(\lambda)+w_1^m+\ldots+w_k^m)[\lambda;\underline{w}],\qquad z_i[\lambda;\underline{w}]=w_i[\lambda;\underline{w}].
\end{equation}

We will say that a chain $[\lambda;\underline{w}]$ \newword{appears} in $\cB$ if there is a corresponding eigenvector in $V(\cB)$ satisfying \eqref{eq: Delta z new}.

\begin{lemma}
\label{lem: d plus minus basic}
a) We have 
\begin{equation}
\label{eq: d minus new}
d_{-}[\lambda;\underline{w}]=d(\lambda;\underline{w})[\lambda\cup w_k;w_{k-1},\ldots,w_{1}]
\end{equation}
for some coefficient $d(\lambda;\underline{w})$.

b) We have
\begin{equation}
\label{eq: d plus new}
d_{+}[\lambda;\underline{w}]=\sum_{x}c(\lambda;\underline{w},x)
[\lambda;\underline{w},x]
\end{equation}
for some coefficients $c(\lambda;\underline{w},x)$.
\end{lemma}

\begin{proof}
a) $d_{-}[\lambda;\underline{w}]$ is an eigenvector for $\Delta_{p_m}$ with eigenvalue $p_m(\lambda)+\sum_{i = 1}^{k} w_i^m$,
and an eigenvector for $z_1,\ldots,z_{k-1}$ with eigenvalues $w_1,\ldots,w_{k-1}$.

b) Consider the projection of $d_{+}[\lambda;\underline{w},x]$ to the $z_1$-eigenspace with eigenvalue $x$. The result is an eigenvector for $\Delta_{p_m}$ with eigenvalue $p_m(\lambda)+\sum w_i^m+x^m$, and for $z_2,\ldots,z_{k+1}$ with the eigenvalues $w_1,\ldots,w_k$. 
\end{proof}

\begin{remark}
If $[\lambda;\underline{w}]$ appears in $\cB$ but some of the chains in the right hand side of \eqref{eq: d minus new} or \eqref{eq: d plus new} do not appear in $\cB$, we assume that the corresponding coefficients are zero.  
\end{remark}

By Assumption \ref{ass: d-} all coefficients $d(\lambda;\underline{w})$ in \eqref{eq: d minus new} are nonzero. In particular, if $[\lambda;\underline{w}]$ appears in $\cB_k$ then
$[\lambda\cup w_k;w_{k-1},\ldots,w_{1}]$ appears in $\cB_{k-1}$.

\begin{lemma}
\label{lem: Ti admissible}
If $[\lambda;\underline{w}]$ appears in $\cB_k$ and $s_i$ is an admissible transposition for $w$ then $[\lambda;s_i(\underline{w})]$ appears in $\cB_k$. Furthermore, we can normalize the basis in $V_k(\cB)$ such that the action of $T_i$ is given by \eqref{eq: T Ram new}.
\end{lemma}

\begin{proof}
The operators $T_i$ commute with $\Delta_{p_m}$ and with $z_1^{m}+\ldots+z_k^{m}$, hence $T_i[\lambda;\underline{w}]$ belongs to the span of 
$[\lambda;\underline{w}']$ where $\underline{w}'$ is a permutation of $\underline{w}$. In view of the Assumption \ref{ass: semisimple}, the result now follows from Theorem \ref{thm: Ram new}.
\end{proof}

Our final assumption concerns the bases in the components $V_0(\cB)$ and $V_1(\cB)$.

\begin{assumption}\label{ass: d+}
All chains $[\lambda]$ appear in $\cB_0$ (in other words, $\cB_0=E$) and all chains $[\lambda;x]$ appear in $\cB_1$ 
provided that $x$ is an addable weight for $\lambda$.
Furthermore, the coefficients $c(\lambda;x)$  in \eqref{eq: d plus new} are nonzero for all addable weights $x$.
\end{assumption}

It turns out that these easy assumptions give strong constraints on the poset $E$ and the basis $\cB$. 

\begin{lemma}
\label{lem: V1 V2 nonzero}
Suppose that $y$ is addable for $\lambda\cup x$. Then:

1) If $y=tx$ then $c(\lambda;x,y)=0$.

2) If $y\neq tx$ then $c(\lambda;x,y)\neq 0$ and $y\neq qtx$. In particular, $[\lambda;x,y]$ appears in $\cB_2$.
\end{lemma}

\begin{proof}
By Assumption \ref{ass: d+}, the chains $[\lambda;x]$, $[\lambda\cup x]$ and $[\lambda\cup x;y]$ appear in $\cB$. Suppose that $[\lambda;x,y]$ appears in $\cB_2$. We have
$$
d_{-}[\lambda;x]=d(\lambda;x)[\lambda\cup x],\qquad d_{-}[\lambda;x,y]=d(\lambda;x,y)[\lambda\cup x;y].
$$ 
Then the equation
$$z_1(qd_{+}d_{-}-d_{-}d_{+})=qt(d_{+}d_{-}-d_{-}d_{+})z_k$$
implies
$$
y(qd(\lambda;x)c(\lambda\cup x;y)-d(\lambda;x,y)c(\lambda;x,y))=qt(d(\lambda;x)c(\lambda\cup x;y)-d(\lambda;x,y)c(\lambda;x,y))x
$$
or, equivalently,
\begin{equation}\label{eqn: v1 v2 nonzero}
q(y-tx)d(\lambda;x)c(\lambda\cup x;y)=(y-qtx)d(\lambda;x,y)c(\lambda;x,y).
\end{equation}

By our assumptions, $d(\lambda;x),c(\lambda\cup x;y)$ and $d(\lambda;x,y)$ are all nonzero, so either $y=tx$ and $c(\lambda;x,y)=0$, or $y\neq tx$ and $c(\lambda;x,y)\neq 0$. Note that, by \eqref{eqn: v1 v2 nonzero} the latter option implies that $y \neq qtx$.

Finally, suppose that $[\lambda;x,y]$ does not appear in $\cB_2$. Then the coefficient at $[\lambda\cup x;y]$ in $d_{-}d_{+}[\lambda;x]$ must vanish, and similarly to the above we get $q(y-tx)d(\lambda;x)c(\lambda\cup x;y)=0$. By our assumptions, this implies $y=tx$.
\end{proof}

\begin{lemma}\label{lem: excellent}
Assumptions \ref{ass: d-} and \ref{ass: d+} imply that $E$ is excellent. 
\end{lemma}

\begin{proof}
Assume $[\lambda;x,y]$ is a chain in $E$, and $y\neq tx$. By Lemma \ref{lem: V1 V2 nonzero} the coefficient $c(\lambda;x,y)$ is nonzero, so $[\lambda;x,y]$ appears in $d_+^2[\lambda]$ with a nonzero coefficient $c(\lambda;x)c(\lambda;y,x)$ (and, in particular, appears in $\cB_2$).  If $[\lambda;y,x]$ is a chain and $x\neq ty$, then it also appears in $d_+^2[\lambda]$ with a nonzero coefficient (and, in particular, appears in $\cB_2$).

Now we use the relation $T_1d_{+}^2=d_{+}^2$ which implies that $d_{+}^2[\lambda]$ is an eigenvector of $T_1$ with eigenvalue 1.  Since $T_1$ preserves the span of $[\lambda;x,y]$ and $[\lambda;y,x]$, we have $y\neq x$ and the following cases:

1) $y=qx$, then $[\lambda;x,y]$ is an eigenvector for $T_1$ with eigenvalue 1, and $[\lambda;y,x]$ is an eigenvector for $T_1$ with eigenvalue $-q$. This is impossible, so $[\lambda;y,x]$ is not a chain in $E$.

2) $y=q^{-1}x$, then $[\lambda;x,y]$ is an eigenvector for $T_1$ with eigenvalue $-q$, contradiction. So this case is not possible.

3) $y\neq q^{\pm 1}x$. Then the eigenvector for $T_1$ must contain both $[\lambda;x,y]$ and $[\lambda;y,x]$ with nonzero coefficients, and we conclude that $[\lambda;y,x]$ is a chain in $E$.

Finally, assume that $y=tx$ and $[\lambda;y,x]$ is a chain in $E$. Then by Lemma \ref{lem: V1 V2 nonzero} the chain $[\lambda;y,x]$ appears in $d_+^2[\lambda]$ with nonzero coefficient (and, in particular, appears in $\cB_2$). Then by the above $[\lambda;x,y]$ also must appear in $d_+^2[\lambda]$ with nonzero coefficient, which contradicts Lemma \ref{lem: V1 V2 nonzero}(1).
\end{proof}

\begin{lemma}
\label{lem: must be good}
Assume that $[\lambda;\underline{w}]$ appears in $\cB_k$. Then the chain $[\lambda;\underline{w}]$ is good.
\end{lemma}

\begin{proof}
By Lemma \ref{lem: Ti admissible}, if $[\lambda;\underline{w}]$  appears in $\cB_k$ and $s_i$ is an admissible transposition for $w$ then $[\lambda;s_i(\underline{w})]$ also appears in $\cB_k$.

Suppose that $[\lambda;\underline{w}]$  appears in $\cB_k$ and $w_i=tw_{i+1}$. Then $s_i$ is admissible, but $[\lambda; s_i(\underline{w})]$ is not a chain in $E$ by Lemmas \ref{lem: no inverse new} and \ref{lem: excellent}, contradiction.

Now suppose that $w_i=tw_j$ for $i<j$. By Lemma \ref{lem: not good} we can find a sequence of admissible transpositions which transforms $w$ to $w'$ with $w'_i= tw'_{i+1}$. Contradiction.
\end{proof}

The following result is clear, and will be used to simplify the coefficients for $d_{-}$.

\begin{proposition}
\label{prop: normalize symmetric}
Assume that we renormalize the basis 
$$
[\lambda;\underline{w}]\mapsto \varphi(\lambda;\underline{w})[\lambda;\underline{w}]
$$
where $\varphi(\lambda;\underline{w})$ is nonzero and symmetric under admissible transpositions of $\underline{w}$. Then the formulas for $T_i$ do not change.
\end{proposition}

First we simplify the equations for $d_{-}$.

\begin{lemma}
\label{lem: d minus symmetry}
Assume that the action of $T_i$ is normalized as in \eqref{eq: T Ram new} and $[\lambda;\underline{w}]$ appears in $\cB_k$. Then we have the following:

a) Assume that $s_i$ is an admissible transposition for $\underline{w}$, $1\le i\le k-2$. Then $d(\lambda;\underline{w})=d(\lambda;s_i(\underline{w}))$.

b) Assume that $s_{k-1}$ is an admissible transposition for $w$. Then
$$
d(\lambda;\underline{w})d(\lambda\cup w_k;w_{k-1}\ldots,w_{1})=d(\lambda;s_{k-1}(\underline{w}))d(\lambda\cup w_{k-1};w_k, w_{k-2}\ldots,w_{1}).
$$
\end{lemma}

\begin{remark}
By Lemma \ref{lem: Ti admissible} and Assumption \ref{ass: d-} all coefficients in Lemma \ref{lem: d minus symmetry} are nonzero.
\end{remark}

\begin{proof}
Let $\alpha_i=\frac{(q-1)w_{i+1}}{w_i-w_{i+1}}$ and $\beta_i=\frac{w_i-qw_{i+1}}{w_i-w_{i+1}}$. 

a) For $1\le i\le k-2$ equation $T_id_{-}=d_{-}T_i$ translates to 
\begin{align*}
\alpha_id&(\lambda;\underline{w})[\lambda\cup w_k;w_{k-1}\ldots,w_{1}]+
\beta_id(\lambda;\underline{w})[\lambda\cup w_k;s_i(w_{k-1}\ldots,w_{1})]
\\
&=\alpha_id(\lambda;\underline{w})[\lambda\cup w_k;w_{k-1}\ldots,w_{1}]+
\beta_id(\lambda;s_i(\underline{w}))[\lambda\cup w_k;s_i(w_{k-1}\ldots,w_{1})]
\end{align*}
which implies
$$
\beta_id(\lambda;\underline{w})=\beta_id(\lambda;s_i(\underline{w})).
$$
Since $s_i$ is admissible, $\beta_i\neq 0$ and $d(\lambda;\underline{w})=d(\lambda;s_i(\underline{w}))$.

b) The equation $d_-^2T_{k-1}=d_{-}^2$ translates to
$$
\alpha_{k-1}d(\lambda;\underline{w})d(\lambda\cup w_k;w_{k-1}\ldots,w_{1})+
\beta_{k-1}d(\lambda;s_{k-1}(\underline{w}))d(\lambda\cup w_{k-1};w_{k},w_{k-2},\ldots,w_1)
$$
$$
=d(\lambda;\underline{w})d(\lambda\cup w_k;w_{k-1}\ldots,w_{1}).
$$
Since $\alpha_{k-1}+\beta_{k-1}=1$, we can rewrite this as
$$
\beta_{k-1}d(\lambda;\underline{w})d(\lambda\cup w_k;w_{k-1}\ldots,w_{1})=
\beta_{k-1}d(\lambda;s_{k-1}(\underline{w}))d(\lambda\cup w_{k-1};w_{k},w_{k-2},\ldots,w_1).
$$
Since $s_{k-1}$ is admissible, $\beta_{k-1}\neq 0$ and 
we get the desired equation.
\end{proof}

\begin{corollary}
Assume that the action of $T_i$ is normalized as in \eqref{eq: T Ram new} and $[\lambda;\underline{w}]$ appears in $\cB_k$. The product 
\begin{equation}
\label{eq: product d}
d(\lambda;\underline{w})d(\lambda\cup w_k;w_{k-1}\ldots,w_{1})\cdots d(\lambda\cup w_k\cup\dots \cup w_2;w_1).
\end{equation}
is nonzero and symmetric under admissible transpositions of $w_1\ldots,w_k$.
\end{corollary}

\begin{proof}
Suppose that $s_i$ is an admissible transposition. All the factors $d(\lambda\cup w_k\dots\cup w_j;w_{j-1},\ldots w_1 )$ for $j>i+2$ are symmetric under $s_i$ by Lemma \ref{lem: d minus symmetry}(a). All the factors for $j\le i$ are clearly symmetric under $s_i$ since $\lambda\cup w_k\dots\cup w_j$ is symmetric. We are left with two factors
$$
d(\lambda\cup w_k\dots\cup w_{i+2};w_{i+1},\ldots w_{1})d(\lambda\cup w_k\dots\cup w_{i+1};w_{i},\ldots w_{1})
$$
whose product is symmetric by Lemma \ref{lem: d minus symmetry}(b).
\end{proof}

\begin{lemma}
\label{lem: normalize d minus}
The representation $V(\cB)$ is isomorphic to a representation where all coefficients of $d_{-}$ are equal to $1$, and the coefficients of $T_i$ are given by \eqref{eq: T Ram new}.
\end{lemma}

\begin{proof}
By Lemma \ref{lem: Ti admissible} we can assume that the coefficients of $T_i$ are given by \eqref{eq: T Ram new}.

Now we normalize the basis $[\lambda;\underline{w}]$ by the product \eqref{eq: product d}. Since it is symmetric under admissible transpositions, the coefficients for $T_i$ will not change. On the other hand, the coefficients for $d_{-}$ will be multiplied by
$$
\frac{d(\lambda\cup w_k;w_{k-1}\ldots,w_{1})\cdots d(\lambda\cup w_k\cup\dots \cup w_2;w_1).}{d(\lambda;\underline{w})d(\lambda\cup w_k;w_{k-1}\ldots,w_{1})\cdots d(\lambda\cup w_k\cup\dots \cup w_2;w_1)}=\frac{1}{d(\lambda;\underline{w})}
$$
and become all equal to 1.
\end{proof}

From now on we assume that all coefficients of $d_{-}$ are equal to 1, and the coefficients of $T_i$ as given by \eqref{eq: T Ram new}. This is allowed by Lemma \ref{lem: normalize d minus} and  the computations simplify significantly.
  
\begin{lemma}
\label{lem: c from edges}
Assume $[\lambda;\underline{w}]$ appears in $\cB_k$. 
If $[\lambda;\underline{w},x]$ is a good chain then it appears in $\cB_{k+1}$ and 
\begin{equation}
\label{eq: c from edges}
c(\lambda;\underline{w},x)=q^{k}c(\lambda\cup \underline{w};x)\prod_{i = 1}^{k}\frac{x-tw_i}{x-qtw_i}.
\end{equation}
Conversely, if \eqref{eq: c from edges} is satisfied for all good chains then the relation \eqref{eq:qphi} is satisfied on $V(\cB)$.
\end{lemma}

\begin{proof}
We prove the statement by induction in $k$. The base case $k=1$ follows from Assumption \ref{ass: d+}. By Lemma \ref{lem: must be good} $[\lambda;\underline{w}]$ is good, hence by Assumption \ref{ass: d-} $[\lambda\cup w_k;w_{k-1}\ldots,w_{1}]$ is good and appears in  $\cB_{k-1}$.
We look at the equation 
$$
z_1(qd_{+}d_{-}-d_{-}d_{+})=qt(d_{+}d_{-}-d_{-}d_{+})z_k
$$
which implies
$$
x(qc(\lambda\cup w_k;w_{k-1}\ldots,w_{1},x)-c(\lambda;\underline{w},x))=
qt(c(\lambda\cup w_k;w_{k-1}\ldots,w_{1},x)-c(\lambda;\underline{w},x))w_k
$$
and 
\begin{equation}
\label{eq: qphi induction}
q(x-tw_k)c(\lambda\cup w_k;w_{k-1}\ldots,w_{1},x)=(x-qtw_k)c(\lambda;\underline{w},x).
\end{equation}
Assume that the chain $[\lambda;\underline{w},x]$ is good, then $[\lambda;\underline{w}]$ and $[\lambda\cup w_k;w_{k-1}\ldots,w_{1},x]$ are good too. Also, $x\neq tw_k$ by definition of a good chain and $x\neq qtw_k$ by Lemma \ref{lem: no qt}.
By the assumption of induction, $[\lambda\cup w_k;w_{k-1}\ldots,w_{1},x]$ appears in $\cB_{k-1}$ and 
$c(\lambda\cup w_k;w_{k-1}\ldots,w_{1},x)\neq 0$, therefore
by \eqref{eq: qphi induction} we get
$c(\lambda;\underline{w},x)\neq 0$. In particular, $[\lambda;\underline{w},x]$ appears in $\cB_{k+1}$ and we conclude that
$$
c(\lambda;\underline{w},x)=q\frac{x-tw_k}{x-qtw_k}c(\lambda\cup w_k;w_{k-1}\ldots,w_{1},x).
$$
which by the assumption of induction implies \eqref{eq: c from edges}.

Assume now that  the chain $[\lambda;\underline{w},x]$ is not good. If $[\lambda\cup w_k;w_{k-1}\ldots,w_{1},x]$ is not good either, then it does not appear in $\cB_k$, and the corresponding terms in $d_{+}d_{-}$ and $d_{-}d_{+}$ both vanish. If  $[\lambda\cup w_k;w_{k-1}\ldots,w_{1},x]$ is good but $[\lambda;\underline{w},x]$ is not good, then we must have $x=tw_{k}$ and $c(\lambda;\underline{w},x)$
vanishes by \eqref{eq: c from edges}, so \eqref{eq: qphi induction} (and thus \eqref{eq:qphi}) holds.
\end{proof}

\begin{corollary}
\label{cor: only good chains appear}
A chain $[\lambda;\underline{w}]$ appears in $\cB_k$ if and only if it is good.
\end{corollary}

\begin{proof}
This follows from Lemma \ref{lem: must be good} and Lemma \ref{lem: c from edges}.
\end{proof}

\begin{theorem}\label{thm: classification}
Under Assumptions \ref{ass: d-}, \ref{ass: semisimple}, and \ref{ass: d+}, the representation $V=V(\cB)$ is isomorphic to $V(E,c)$ from Definition \ref{def: calibrated from excellent}.
\end{theorem}

\begin{proof}
We combine all of  the results in this section. By Lemma \ref{lem: excellent} the weighted poset $E$ is excellent.
By Corollary \ref{cor: only good chains appear} the basis $\cB$ of $V(\cB)$ is given by all good chains in $E$. By Theorem \ref{thm: Ram new} the coefficients of $T_i$ are given by \eqref{eq: T Ram new}, and by Lemma \ref{lem: normalize d minus} we can normalize the basis so that all coefficients of $d_{-}$ are equal to 1. 

Finally, by Lemma \ref{lem: c from edges} we conclude that the coefficients of $d_{+}$ are given by \eqref{eq: c from edges} which agrees with \eqref{eq: c from edges new}.
\end{proof}

\bibliographystyle{plain}
\bibliography{bibliography.bib}

\begin{thebibliography}{10}

\bibitem{milo}
Milo~James Bechtloff~Weising.
\newblock Stable-limit non-symmetric {M}acdonald functions.
\newblock {\em arXiv preprint arXiv:2307.05864}, 2023.

\bibitem{BCMN}
L\'{e}a Bittmann, Alex Chandler, Anton Mellit, and Chiara Novarini.
\newblock Type {$A$} {DAHA} and doubly periodic tableaux.
\newblock {\em Adv. Math.}, 416:Paper No. 108919, 58, 2023.

\bibitem{BNS2}
Chris Bowman, Emily Norton, and José Simental.
\newblock Unitary representations of cyclotomic {H}ecke algebras at roots of
  unity: Combinatorial classification and {B}{G}{G} resolutions.
\newblock {\em Journal of the Institute of Mathematics of Jussieu}, page
  1–52, 2022.

\bibitem{EHA}
Igor Burban and Olivier Schiffmann.
\newblock On the {H}all algebra of an elliptic curve, {I}.
\newblock {\em Duke Math. J.}, 161(7):1171, 2012.

\bibitem{CGM}
Erik Carlsson, Eugene Gorsky, and Anton Mellit.
\newblock The {$\Bbb{A}_{q,t}$} algebra and parabolic flag {H}ilbert schemes.
\newblock {\em Math. Ann.}, 376(3-4):1303--1336, 2020.

\bibitem{Shuffle}
Erik Carlsson and Anton Mellit.
\newblock A proof of the shuffle conjecture.
\newblock {\em J. Amer. Math. Soc.}, 31(3):661--697, 2018.

\bibitem{Cherednik}
I.~V. Cherednik.
\newblock Special bases of irreducible representations of a degenerate affine
  {H}ecke algebra.
\newblock {\em Funktsional. Anal. i Prilozhen.}, 20(1):87--88, 1986.

\bibitem{DHIN}
Zajj Daugherty, Iva Halacheva, Mee~Seong Im, and Emily Norton.
\newblock On calibrated representations of the degenerate affine periplectic
  {B}rauer algebra.
\newblock {\em Surv. Math. Appl.}, 16:207--222, 2021.

\bibitem{EGNO}
Pavel Etingof, Shlomo Gelaki, Dmitri Nikshych, and Victor Ostrik.
\newblock {\em Tensor categories}, volume 205 of {\em Mathematical Surveys and
  Monographs}.
\newblock American Mathematical Society, Providence, RI, 2015.

\bibitem{FT}
B.~L. Feigin and A.~I. Tsymbaliuk.
\newblock Equivariant {$K$}-theory of {H}ilbert schemes via shuffle algebra.
\newblock {\em Kyoto J. Math.}, 51(4):831--854, 2011.

\bibitem{GZ}
I.~M. Gel'fand and M.~L. Cetlin.
\newblock Finite-dimensional representations of the group of unimodular
  matrices.
\newblock {\em Doklady Akad. Nauk SSSR (N.S.)}, 71:825--828, 1950.

\bibitem{GH}
{N}icolle {G}onz\'alez and {M}atthew {H}ogancamp.
\newblock A skein theoretic {C}arlsson-{M}ellit algebra.
\newblock {\em (in preparation)}.

\bibitem{Goodberry}
Ben Goodberry.
\newblock Type a partially-symmetric macdonald polynomials.
\newblock {\em arXiv preprint arXiv:2311.12216}, 2023.

\bibitem{Griffeth}
Stephen Griffeth.
\newblock Unitary representations of cyclotomic rational {C}herednik algebras.
\newblock {\em J. Algebra}, 512:310--356, 2018.

\bibitem{GV}
I.~Grojnowski and M.~Vazirani.
\newblock Strong multiplicity one theorems for affine {H}ecke algebras of type
  {A}.
\newblock {\em Transform. Groups}, 6(2):143--155, 2001.

\bibitem{HHLRU}
J.~Haglund, M.~Haiman, N.~Loehr, J.~B. Remmel, and A.~Ulyanov.
\newblock A combinatorial formula for the character of the diagonal
  coinvariants.
\newblock {\em Duke Math. J.}, 126(2):195--232, 2005.

\bibitem{HaglundBook}
James Haglund.
\newblock {\em The {$q$},{$t$}-{C}atalan numbers and the space of diagonal
  harmonics}, volume~41 of {\em University Lecture Series}.
\newblock American Mathematical Society, Providence, RI, 2008.
\newblock With an appendix on the combinatorics of Macdonald polynomials.

\bibitem{HMZ}
James Haglund, Jennifer Morse, and Mike Zabrocki.
\newblock A compositional shuffle conjecture specifying touch points of the
  dyck path.
\newblock {\em Canadian Journal of Mathematics}, 64, 08 2010.

\bibitem{ImNorton}
Mee~Seong Im and Emily Norton.
\newblock Irreducible calibrated representations of periplectic {B}rauer
  algebras and hook representations of the symmetric group.
\newblock {\em J. Algebra}, 560:442--485, 2020.

\bibitem{IonWu}
Bogdan Ion and Dongyu Wu.
\newblock The stable limit {DAHA} and the double {D}yck path algebra.
\newblock {\em Journal of the Institute of Mathematics of Jussieu}, page
  1–46, 2022.

\bibitem{kirillov}
Alexander Kirillov, Jr.
\newblock {\em Quiver representations and quiver varieties}, volume 174 of {\em
  Graduate Studies in Mathematics}.
\newblock American Mathematical Society, Providence, RI, 2016.

\bibitem{Kleshchev}
Alexander Kleshchev.
\newblock Completely splittable representations of symmetric groups.
\newblock {\em J. Algebra}, 181(2):584--592, 1996.

\bibitem{KleshchevRam}
Alexander Kleshchev and Arun Ram.
\newblock Homogeneous representations of {K}hovanov-{L}auda algebras.
\newblock {\em J. Eur. Math. Soc. (JEMS)}, 12(5):1293--1306, 2010.

\bibitem{Macdonald}
I.~G. Macdonald.
\newblock {\em Symmetric functions and {H}all polynomials}.
\newblock Oxford Classic Texts in the Physical Sciences. The Clarendon Press,
  Oxford University Press, New York, second edition, 2015.
\newblock With contribution by A. V. Zelevinsky and a foreword by Richard
  Stanley, Reprint of the 2008 paperback edition [ MR1354144].

\bibitem{Mellit}
Anton Mellit.
\newblock Toric braids and {$(m,n)$}-parking functions.
\newblock {\em Duke Math. J.}, 170(18):4123--4169, 2021.

\bibitem{nakajimabook}
Hiraku Nakajima.
\newblock {\em Lectures on {H}ilbert schemes of points on surfaces}, volume~18
  of {\em University Lecture Series}.
\newblock American Mathematical Society, Providence, RI, 1999.

\bibitem{NY}
Hiraku Nakajima and K{o}ta Yoshioka.
\newblock Lectures on instanton counting.
\newblock In {\em Algebraic structures and moduli spaces}, volume~38 of {\em
  CRM Proc. Lecture Notes}, pages 31--101. Amer. Math. Soc., Providence, RI,
  2004.

\bibitem{NY2}
Hiraku Nakajima and K{o}ta Yoshioka.
\newblock Instanton counting on blowup. {I}. 4-dimensional pure gauge theory.
\newblock {\em Invent. Math.}, 162(2):313--355, 2005.

\bibitem{negut}
Andrei Negu\cb{t}.
\newblock Moduli of flags of sheaves and their {K}-theory.
\newblock {\em Algebraic Geometry}, 2(1):19--43, 2015.

\bibitem{Ram}
Arun Ram.
\newblock Affine {H}ecke algebras and generalized standard {Y}oung tableaux.
\newblock {\em J. Algebra}, 260(1):367--415, 2003.
\newblock Special issue celebrating the 80th birthday of Robert Steinberg.

\bibitem{SV}
Olivier Schiffmann and Eric Vasserot.
\newblock The elliptic {H}all algebra and the {$K$}-theory of the {H}ilbert
  scheme of {$\Bbb A^2$}.
\newblock {\em Duke Math. J.}, 162(2):279--366, 2013.

\bibitem{SuzukiVazirani}
Takeshi Suzuki and Monica Vazirani.
\newblock Tableaux on periodic skew diagrams and irreducible representations of
  the double affine {H}ecke algebra of type {A}.
\newblock {\em Int. Math. Res. Not.}, (27):1621--1656, 2005.

\end{thebibliography}

\end{document}